\documentclass[reqno,11pt]{amsart}
\usepackage{amsmath}
\usepackage{mathrsfs}
\usepackage{multicol}
\usepackage[left=2.5cm,right=2.5cm,top=2.5cm,bottom=2.5cm]{geometry}
\usepackage[colorlinks=true, pdfstartview=FitV, linkcolor=blue, 
            citecolor=blue, urlcolor=blue]{hyperref}
            
\makeatletter
\@namedef{subjclassname@2020}{\textup{2020} Mathematics Subject Classification}
\makeatother

\usepackage{tikz}
\usepackage{pgfplots}
\usetikzlibrary{hobby} 
\usetikzlibrary{angles,quotes} 

\numberwithin{equation}{section}
\theoremstyle{plain}
\newtheorem{theorem}{Theorem}[section]
\newtheorem*{theorem*}{Main result}
\newtheorem{proposition}[theorem]{Proposition}
\newtheorem{lemma}[theorem]{Lemma}

\newtheorem{corollary}[theorem]{Corollary}
\theoremstyle{definition}
\newtheorem{definition}[theorem]{Definition}
\newtheorem{remark}[theorem]{Remark}

\newcommand{\R}{{\mathbb R}}
\newcommand{\N}{{\mathbb N}}

\newcommand{\dx}{{\rm d}x}
\newcommand{\ds}{{\rm d}s}
\newcommand{\dt}{{\rm d}t}
\renewcommand{\d}{{\rm d}}
\newcommand{\dist}{\mathrm{dist}\,}

\def\im{\mbox{Im}\,}

\title[] {On an obstacle problem for the Brakke flow with a generalized right-angle boundary condition}

\author[K. Nik] {Katerina Nik} 
\address[Katerina Nik]{University of Vienna,  Faculty of Mathematics, 
	Oskar-Morgenstern-Platz 1, 1090 Vienna, Austria}
\email{katerina.nik@univie.ac.at}
\author[K. Takasao] {Keisuke Takasao}
\address[Keisuke Takasao]{Department of Mathematics, Kyoto University, 
	Kitashirakawa-Oiwakecho Sakyo Kyoto 606-8502, Japan}
\email{k.takasao@math.kyoto-u.ac.jp}


\begin{document}

\begin{abstract} 
We study  Brakke's mean curvature flow with obstacles and with a right-angle boundary condition. Assuming that the obstacles have $C^{1,1}$-boundaries we prove that a weak solution exists globally in time. To show the existence we apply the phase-field method and thus investigate the singular limit of the Allen-Cahn equation with forcing term and homogeneous Neumann bounday condition. We also construct sub- and supersolutions that correspond to the obstacles.

\end{abstract}
\maketitle

\section{Introduction} 
\label{sec:intro}
A family $\{M_t\}_{t\geq 0}$ of hypersurfaces in $\R^n$ is referred to as the {\it mean curvature flow} (MCF) if the velocity vector $v$ of $M_t$ is equal to its mean curvature vector $h$ at each point and time, that is,
\begin{equation}\label{eq:MCF}
	v=h \quad \text{ on } M_t, \ t>0. 
\end{equation}
The MCF has attracted extensive research interest over the past few decades, as it represents one of the fundamental geometric evolution problems encountered in various fields such as differential geometry, general relativity, image processing and material science. 
A key research question involves the global-in-time existence of MCF. 
Starting with a smooth hypersurface $M_0$, it is possible to find a smoothly evolving MCF up to a finite time, until it encounters singularities like vanishing or pinching. To extend the flow beyond these singularities, a concept of weak (or generalized) solution to \eqref{eq:MCF} is required.

The first attempt in this direction was made by Brakke in his seminal work \cite{brakke}, where he proposed a definition of a weak solution to \eqref{eq:MCF} within the framework of geometric measure theory. This solution, commonly referred to as {\it Brakke flow}, is characterized by a family of integral varifolds $\{V_t\}_{t\geq 0}$ that satisfy an integral inequality corresponding to \eqref{eq:MCF}. Brakke proved the global-in-time existence of  Brakke flow in $\mathbb{R}^n$, given a general integral varifold $V_0$ (with a mild finiteness assumption) as an initial datum. However, a significant concern regarding his existence result was the lack of assurance that the obtained Brakke flow is nontrivial; that is, irrespective of the initial $M_0$, setting $M_t= \emptyset$ for all $t>0$ yields a trivial Brakke flow $V_t=0$ for all times. Addressing this concern, a recent contribution by Kim and Tonegawa \cite{kim-tonegawa} showed that a nontrivial Brakke flow exists for all times by refining Brakke's original argument and utilizing Huisken's monotonicity formula \cite{huisken}. 
Furthermore, an alternative proof for the global-in-time existence of a nontrivial Brakke flow in $\mathbb{R}^n$ was provided by Ilmanen  \cite{ilmanen,ilmanen1}, employing the phase-field method through the Allen-Cahn equation and elliptic regularization. This methodology was also applied by Mizuno and Tonegawa \cite{mizuno-tonegawa}, who were the pioneers in formulating the MCF {\it with}  a right-angle Neumann boundary condition in the sense of Brakke. Their research focused on the singular limit of the Allen-Cahn equation with a zero Neumann boundary condition in bounded and convex domains, which was later extended to nonconvex domains by Kagaya \cite{kagaya}. The interest in right-angle Neumann boundary conditions for MCF has been notably advanced by contributions from Edelen \cite{edelen}, and Abels and Moser\cite{abels-moser}.

On a different note, Mugnai and Röger \cite{mugnai-roger1,mugnai-roger2} showed the global-in-time existence of a weak solution, known as {\it $L_2$-flow}, for \eqref{eq:MCF} with an additional $L_2$-forcing term. Their concept of $L_2$-flow bears similarities to the Brakke flow. Moreover, the regularity of Brakke flow, initially investigated by Brakke \cite{brakke}, has recently been thoroughly examined by White  \cite{white}, Kasai and Tonegawa \cite{kasai-tonegawa}, Lahiri \cite{lahiri}, Tonegawa \cite{tonegawa2014}, and Stuvard and Tonegawa \cite{stuvard-tonegawa}. Results for other types of weak solutions to \eqref{eq:MCF} include the existence of viscosity solutions via the level set method, achieved independently by Chen, Giga, and Goto \cite{chen-giga-goto} and by Evans and Spruck \cite{evans-spruck}, as well as the existence of weak solutions constructed using a variational approach, as examined by Almgren, Taylor, and Wang \cite{almgren-taylor-wang}, and Luckhaus and Sturzenhecker \cite{luckhaus-sturzenhecker}.

Considering MCF's applicability as a simplified model for cell motility \cite{elliott-stinner-venkataraman,mizuhara-beryland-rybalko-zhang},  it is natural to consider the corresponding obstacle problem. This problem involves regions, known as obstacles, that the evolving surface $M_t$ cannot penetrate. Almeida, Chambolle, and Novaga  \cite{almeida-chambolle-novaga}  established the global-in-time existence of weak solutions for dimensions $n \geq 2$ in the obstacle problem for \eqref{eq:MCF}  using a variational approach, and also demonstrated short-time existence and uniqueness of $C^{1,1}$-solutions for $n = 2$ under the condition of compact $C^{1,1}$-boundaries for the obstacles.  The latter result was extended to $n \geq 2$ by Mercier and Novaga \cite{mercier-novaga}, who also proved the global-in-time existence and uniqueness of graphical viscosity solutions, provided that the boundaries of the obstacles are represented as graphs.
In the context of viscosity solutions utilizing the level set method, Mercier  \cite{mercier} established the global-in-time existence and uniqueness of continuous viscosity solutions satisfying \eqref{eq:MCF} with a forcing term $k$ on its right-hand side accounting for obstacles. Ishii, Kamata, and Koike  \cite{ishii-kamata-koike} demonstrated the global-in-time existence and uniqueness of Lipschitz viscosity solutions when $k \equiv 0$ and under specific regularity conditions. The large-time behavior of viscosity solutions with a constant driving force $k$ was investigated by Giga, Tran, and Zhang \cite{giga-tran-zhang}. Recently, Takasao obtained a global-in-time existence theorem for the weak solution to \eqref{eq:MCF} in dimensions $n=2,3$ and with obstacles having $C^{1,1}$-boundaries, in the sense of Brakke. 
This particular attention to dimensions $n=2,3$ is due to findings from prior research \cite{roger-schatzle}.\\

\noindent Let $\Omega \subset \mathbb{R}^n$ denote a bounded domain with a smooth boundary $\partial \Omega$. Assume that the obstacles $O_+$ and $O_-$ are both compactly contained in $\Omega$,  have $C^{1,1}$-boundaries and satisfy $\text{dist}(O_+,O_-)>0$. Consider the family $\{M_t\}_{t\geq 0}$ of smooth hypersurfaces in $\overline{\Omega}$ with smooth boundaries $\partial M_t \subset \partial \Omega$. In this paper we prove, in the framework of Brakke, the global-in-time existence of a weak solution to MCF with obstacles and subject to a right-angle Neumann boundary condition. 

Here, we impose that there exists an open set $U_t \subset \Omega$ such that
$M_t = \partial U_t \setminus \partial \Omega$, $O_+ \subset U_t$, and $ U_t \cap O_- =\emptyset$
 for any $t\geq 0$. 
To achieve our findings, we employ the phase-field method, specifically analyzing the singular limit of the Allen-Cahn equation. This approach is influenced by \cite{mercier-novaga}, treating the Allen-Cahn equation in this paper as a formal approximation of: 
\begin{equation} \label{eq:MCFOB}
	\begin{cases}
	v=h+g \nu_{_{M_t}}
		& \quad   \text{on } M_t \cap \Omega, \, t>0,  
		\\[2mm]
\nu_{_{\partial M_t}} \perp \partial \Omega & \quad \text{on } \partial M_t, \, t>0.
	\end{cases}
\end{equation}
Here, $\nu_{_{M_t}}$ and $\nu_{_{\partial M_t}}$ represent the outward unit normal vector to $M_t$ and $\partial M_t$, respectively, and the function $g$ is defined as follows
\begin{equation*}
	g(x) =
	\begin{cases}
		c_1 & \text{if $x \in \overline{O_+}$}, \\[1mm]
		-	c_1 & \text{if $x \in \overline{O_-}$}, \\[1mm]
		0 & \text{otherwise},
	\end{cases}
\end{equation*}
with the constant $c_1>1$ given in  \eqref{eq:c1} below. To be more precise, 
we will introduce a tailored forcing term for the Allen-Cahn equation (see \eqref{prob:AC} below) and construct sub- and supersolutions representing the obstacles. Furthermore, we will show that the solution to our modified Allen-Cahn equation, including the forcing term and Neumann boundary condition, converges to the global weak solution to \eqref{eq:MCFOB} in the sense of Brakke. The core results of our study are summarized as follows (see Theorems \ref{maintheorem1} and \ref{maintheorem2}).

\begin{theorem*}
Let $u^\varepsilon := u^\varepsilon (x,t)$ be the unique solution to the Allen-Cahn equation 
\eqref{prob:AC} for $\varepsilon >0$. Define a Radon measure $\mu_t^\varepsilon$ by 
$$
\mu_t^\varepsilon := \frac{1}{\sigma} \left( \frac{\varepsilon |\nabla u^\varepsilon(\cdot, t)|^2}{2} + \frac{W(u^\varepsilon(\cdot, t))}{\varepsilon}\right) \mathscr{L}^n \lfloor_{\Omega},
$$
where $\sigma := \int_{-1}^1 \sqrt{W(s)} \, \d s$.   Assuming suitable conditions on the initial hypersurface $M_0$ and choosing a suitable initial datum  $u_0 ^\varepsilon$,  
there exists a subsequence $\{\varepsilon_i\}_{i \in \N}$ converging to $0$ ($i \to \infty$) along with  a set of Radon measures $\{ \mu_t \}_{t\geq 0}$ on $\overline{\Omega}$ such that $\mu _t ^{\varepsilon _{i}} \rightharpoonup \mu_t$  ($i \to \infty$) in the sense of measure for all $t\geq 0$. Furthermore, $\mu_t$ is a Brakke flow with a generalized right-angle condition on $\partial \Omega$, that is, 
\begin{enumerate}
	\item[(1)]$\mu_t$ has a bounded first variation on $\overline{\Omega}$ for a.e. $t\geq 0$,\\[-0.3cm]
	\item[(2)] $\mu_t$ is $(n-1)$-rectifiable on $\overline{\Omega}$ and integral on $\Omega \setminus \overline{O_+ \cup O_-}$ for a.e. $t\geq 0$,\\[-0.3cm]
	\item[(3)] $\mu_t$ satiesfies Brakke's inequality up to the boundary, 
	with an adjusted first variation formula on $\partial \Omega$.
\end{enumerate}
\end{theorem*}
Central to proving the aforementioned result is the examination of the vanishing of the discrepancy measure defined as
$$
\xi_t^\varepsilon := \frac{1}{\sigma} \left( \frac{\varepsilon |\nabla u^\varepsilon(\cdot, t)|^2}{2} - \frac{W(u^\varepsilon(\cdot, t))}{\varepsilon}\right) \mathscr{L}^n \lfloor_{\Omega}.
$$
This paper introduces several novelties: (i) The sub- and supersolutions discussed in \cite{takasao} on $\R^n$ are not suitable for our purposes, prompting us to construct new "local" sub- and supersolutions for the obstacles. (ii) We provide a proof of the rectifiability of $\mu_t$ across all dimensions, extending the results of \cite{takasao} which were limited to $n=2,3$.
(iii) We introduce a family of explicit initial data for the Allen-Cahn equation, specifically designed  to suit our requirements.

The organization of the paper is as follows:
We begin by clarifying notation, gathering essential definitions related to geometric measure theory, and detailing assumptions on the obstacles in Section \ref{sec:preliminaries}. Following this, Section \ref{sec:AC} introduces the specific Allen-Cahn equation under study in this paper. 
In Section \ref{sec:wellposedness}, we establish that the Allen-Cahn equation is globally well-posed in  time and provide a standard $L_\infty$-estimate for its solution $u^\varepsilon$.
Section \ref{sec:energyestimates} examines the up-to-the-boundary monotonicity formula for the measure $\mu_t^{\varepsilon}$ and shows that the growth rate of the discrepancy measure $\xi_t^{\varepsilon}$ is bounded by a negative power of $\varepsilon$. Additionally, we ascertain the upper bound for the density of $\mu_t^{\varepsilon}$.
The vanishing of the discrepancy measure is the focus of Section \ref{sec:vanishing}. In Section \ref{sec:subsupersolution}, we construct supersolutions and subsolutions to the Allen-Cahn equation, necessary to show that solutions to \eqref{eq:MCFOB} do not penetrate the obstacles. Finally, Section \ref{sec:mainresults} is dedicated to proving the main result.



\section{Notations, basic definitions and assumptions}
\label{sec:preliminaries}
In this section, we revisit key concepts from geometric measure theory, while also establishing the notation and assumptions that will be used throughout the paper.

\subsection*{Basic notations} 
We denote $T$ as a positive, finite final time and use $n$ to represent a positive integer. For a point $x \in \R^n$ and $r>0$, the ball of radius $r$ centered at $x$ is defined as 
\[
B_r(x) := \{ y \in \R^n \,|\, |x-y| < r \}.
\]
Given a positive integer $k$, with $k \leq n$, $\mathscr{L}^k$ denotes the $k$-dimensional Lebesgue measure on $\R^k$, while $\mathscr{H}^k$ stands for the $k$-dimensional Hausdorff measure on $\R^n$. We represent the restriction of $\mathscr{H}^k$ to a set $A$ by $\mathscr{H}^k \lfloor_{A}$. The measure of a unit ball in $\R^k$ is denoted as
\[
\omega_k := \mathscr{L}^k(\{x \in \R^k \,|\, |x| < 1\}).
\]
For any Radon measure $\mu$ on $\R^n$, a function $\phi \in C_c (\R^n)$, and a $\mu$-measurable set $A$, we frequently use
$$
\mu (\phi) := \int_{\R^n} \phi \, \d \mu, \quad \mu(A) := \int_A \, \d \mu.
$$
We define the support of $\mu$, denoted as $\text{spt} \mu$, as the set of points $x \in \R^n$ such that $\mu(B_r(x)) > 0$ for any $r > 0$.

For vectors $a=(a_1,\dots,a_n)$ and $b=(b_1,\dots,b_n)$ in $\R^n$, we define the tensor product $a \otimes b$ to be the $n\times n$-matrix with entries $(a_i b_j)$. For matrices $A=(a_{ij})$ and $B=(b_{ij})$ in $\R^{n \times n}$, we specify their scalar product as 
$$
A \cdot B := \sum_{i,j=1}^n a_{ij} b_{ij}.
$$
Next, we introduce some concepts related to varifolds, directing the reader to \cite{tonegawa} for an in-depth exploration. These concepts are instrumental in defining Brakke's mean curvature flow with a generalized right angle condition.

\subsection*{The Grassmann manifold and varifolds}
Let $\mathbf{G}(n,n-1)$ denote the space of $(n-1)$-dimensional subspaces of $\R^n$. When considering  $S \in \mathbf{G}(n,n-1)$ with its corresponding unit normal vector $\nu_S$, the matrix $\text{Id} - \nu_S \otimes \nu_S \in \R^{n\times n}$ is the unique orthogonal projection of $\R^n$ onto $S$, and we often identify $S$ with this matrix. Here, $\text{Id}$ refers to the identity matrix in $\R^{n\times n}$. 

In this paper, for a bounded open set $\Omega \subset \R^n$, the focus shifts to various entities on $\overline{\Omega}$ instead of $\Omega$ itself. Henceforth, we will assume $U \subset \R^n$ to be either open or compact.
Let $\mathbf{G}_{n-1} (U):= U \times \mathbf{G} (n,n-1)$. 

A {\it general $(n-1)$-varifold} $V$ on $U$ is a Radon measure defined on $\mathbf{G}_{n-1} (U)$, and the set of all $(n-1)$-varifolds in $U$ is denoted by $\mathbf{V}_{n-1} (U)$. 
For $V \in  \mathbf{V}_{n-1} (U)$, let $\| V \|$ be the {\it weight measure} of $V$ (with no fear of confusion with the operator norm), namely the measure defined on $U$ by 
\[
\| V \| (\phi) 
:=\int _{\mathbf{G}_{n-1} (U)} \phi (x) \, \d V (x,S) \quad \text{for every} 
\ \phi \in C_c (U).
\]
We say that $V \in \mathbf{V} _{n-1} (U)$ is a {\it rectifiable varifold}, if there are some 
$\mathscr{H}^{n-1}$-measurable and countably $(n-1)$-rectifiable set $M \subset U$ 
as well as a non-negative function $\theta \in L_{1, \text{loc}}(\mathscr{H}^{n-1} \lfloor_{M})$ such that 
\begin{equation}
	\label{eq:tangent}
V (\phi) = \int _{M} \phi (x,\text{Tan}_x M) \theta (x)\, \d \mathscr{H}^{n-1} (x) \quad \text{for all} 
\ \phi \in C_c (\mathbf{G}_{n-1} (U)).
\end{equation}
Here, $\text{Tan}_x M$ is the approximate tangent space to $M$ at $x$, which exists for 
$\mathscr{H}^{n-1}$-a.e. $x \in M$. 
If additionally $\theta (x)$ is integer-valued for $\mathscr{H}^{n-1}$-a.e. $x \in M$, $V$ is said to be an {\it integral varifold}.

A rectifiable $(n-1)$-varifold is uniquely determined by its weight measure through the identity \eqref{eq:tangent}. Due to this, we say that a Radon measure $\mu$  on $U$ is rectifiable (or integral) when one can associate a rectifiable (or integral) varifold $V$ such that $\| V \| =\mu$. The set of all rectifiable and integral $(n-1)$-varifolds in $U$ is denoted by $\mathbf{RV}_{n-1}(U)$ and $\mathbf{IV}_{n-1}(U)$, respectively. 

\subsection*{First variation and generalized mean curvature}
For a varifold $V \in \mathbf{V}_{n-1}(U)$, we define the {\it first variation} $\delta V $ of $V$ by
\[
\delta V (g) := \int _{\mathbf{G}_{n-1} (U)} \nabla g (x) \cdot S \, \d V (x,S) \quad \text{for any }
\ g \in C_c ^1 (U;\R^n).
\]
Define $\|\delta V\|$ as the total variation of $\delta V$, assuming it exists. If $\|\delta V\|$ is locally bounded (in the case $U=\overline{\Omega}$ this means $\|\delta V\|(\overline{\Omega}) < \infty$), the Riesz representation theorem \cite[Theorem 1.38]{evans-gariepy} and the Lebesgue decomposition theorem  \cite[Theorem 1.31]{evans-gariepy} can be applied to $\delta V$ with respect to $\|V\|$. Consequently, there are a $\|V\|$-measurable function $h(V,\cdot):U\to \R^n$, a Borel set $Z\subset U$ with $\|V\|(Z)=0$, and a $\|\delta V\| \lfloor_{Z}$-measurable function $\nu^{\text{sing}}:Z \to \R^n$ with $|\nu^{\text{sing}}|=1$ $\|\delta V\|$-a.e. on $Z$, such that 
\begin{equation}\label{def:MCVec}
\delta V (g) =- \int _U h(V,\cdot) \cdot g  \, \d \| V\| 
+ \int _Z  \nu^{\text{sing}}\cdot g \, \d \|\delta V\|
\quad \text{for all} \ g \in C_c (U;\R^n).
\end{equation}
Here, the vector field $h(V, \cdot)$ is called the {\it generalized mean curvature vector} of $V$. The vector field $\nu^{\text{sing}}$ is referred to as the {\it (outer-pointing) generalized normal} of $V$, and the Borel set $Z$ is called the  {\it generalized boundary} of $V$. A detailed account of varifolds with boundary can be found in \cite{mantegazza}.

\subsection*{Assumption for obstacles} 
Here, we outline our assumptions regarding the obstacles. 
Let $\Omega \subset \R^n$, $n \geq 2$, be a bounded domain with a smooth boundary $\partial \Omega$. The outer unit normal vector of $\partial \Omega$ is denoted by $\nu$. Consider two open sets $O_+,O_- \subset \subset \Omega$ satisfying the following conditions:
\begin{itemize}
	\item[(A1)] There exists $R_0 >0$ such that 
	\begin{equation*}
		O_+ =\bigcup_{B_{R_0} (x)
			\subset O_+} B_{R_0} (x)
		\quad \text{and} \quad
		O_- =\bigcup_{B_{R_0} (x)
			\subset O_-} B_{R_0} (x) \qquad \text{(interior ball condition)}.
	\end{equation*}
	\item[(A2)] There exists $R_1 >0$ such that $\dist (O_+, O_-) \geq R_1$;
\end{itemize}
see Fig.\ref{fig:obstacles} for a depiction of the obstacles $O_+$ and $O_-$. It is noted that if both $O_+$ and $O_-$ have $C^{1,1}$-boundaries, then condition (A1) is met for some $R_0 > 0$, according to \cite{aikawa-et-al}. 

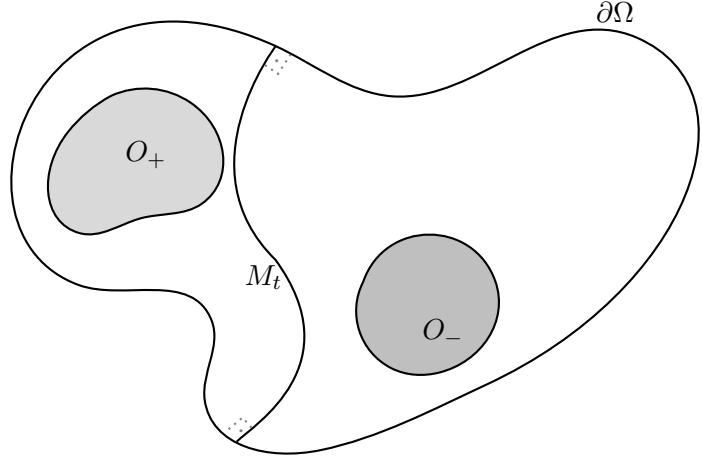
\begin{figure}[h]
	\centering
\begin{tikzpicture}[use Hobby shortcut, scale=1.5, rotate=-500] 
	\draw[thick, closed=true] 
	(0,0) .. 
	(2.7,0.5) .. 
	(2,1.5) .. 
	(2.5,2) ..
	(2,3) ..
	(0.5,3.5) .. 
	(-2.5,2) .. 
	(-0.5,1) .. 
	(0,0);
	
\draw[thick, fill=gray!30, xshift=0.9cm, yshift=-1.1cm] 
(0.5,0.5) .. 
(1,0.5) .. 
(1.5,1.25) .. 
(1,1.5) .. 
(0.5,1.75) .. 
(0,1) .. 
(0.5,0.5); 

\draw[thick, fill=gray!50, scale=1.5, xshift=-1cm, yshift=-0.6cm] 
(1.5,2) .. 
(1.7,2.5) .. 
(1.5,2.7) .. 
(0.9,2.5) .. 
(1.5,2); 

	\node at (-2.5,1.7) {$\partial \Omega$};
	\node at (1.5,0.0) {$O_+$};
	\node at (0.5,2.9) {$O_-$};
	\node at (1.4,1.5) {$M_t$};

\draw[thick] (0,0) .. controls (0,0) and (1.33,0.3) .. (1.22,1.45) .. controls (1.3,2.6) and (2.3,2.45) .. (2.52,2.465);

\node[thick, gray] at (0.08,0.11) {$\cdot$};
\draw[rotate around x=30,thick, gray, dotted] (0.2,0.1) -- (0.2,0.3) -- (-0.01,0.21);
\draw[rotate around x=5, thick, gray, dotted] (2.425,2.55) -- (2.425,2.425) -- (2.6,2.425); 
\node[thick, gray] at (2.435,2.425) {$\cdot$};
\end{tikzpicture}
	\caption{Sketch of  $O_+$ and $O_-$.}
\label{fig:obstacles}
\end{figure}

Let $U_0 \subset \Omega$ be an open set, and define $M_0 := \partial U_0 \setminus \partial \Omega$. We assume that $M_0$ exhibits  $C^1$-regularity up to the boundary $\partial \Omega$ and fulfills the following properties: 
\begin{itemize}
	\item[(A3)] $M_0 \perp \partial \Omega$ on $\partial \Omega \cap \partial M_0$.\\[-0.2cm]
	\item[(A4)] $O_+ \subset \subset U_0$ and $O_- \subset \subset \Omega \setminus \overline{U_0}$.
	That is, $\dist (M_0 ,O_\pm) >0$.\\[-0.2cm]
	\item[(A5)] Let $A:= \{ (y_1, y') \in \R \times \R^{n-1} \mid |y_1| < 1, \, |y'| <2 \}$. There
	exists a $C^1$-diffeomorphism $\Phi :A\to \Phi (A) \subset \R^n $ such that 
	$\Phi (y) \in M_0$ if and only if $y_1 =0$ and $|y'|<1$, and 
	$\Phi (y) \in \partial \Omega \cap \Phi (A) $ if and only if $|y'| =1$.
	In addition, $\Phi (A \cap \{ y_1 >0 , |y'| <1\}) \subset U_0$.
\end{itemize}
Since $M_0$ can be represented by a $C^1$-graph locally, 
there exist $D_0 >0$ and $R_2 \in (0,1)$ such that
\begin{equation}\label{density}
	\sup _{ B_r (x) \subset \Omega, \,  r \in (0,R_2) } \frac{ \mathscr{H}^{n-1} (M_0 \cap B_r (x)) }
	{\omega_{n-1} r^{n-1}} \leq D_0 \qquad \text{(upper bound of density)}.
\end{equation}

The assumptions (A1)-(A5) are supposed to hold throughout the paper, without mentioning it every time.

\section{Allen-Cahn equation with forcing term and Neumann boundary condition}
\label{sec:AC}

As mentioned in the introduction, our approach to proving the global existence of the weak solution to the mean curvature flow \eqref{eq:MCFOB} involves approximating \eqref{eq:MCFOB} through the Allen-Cahn equation. Consequently, this section is dedicated to presenting the Allen-Cahn equation, augmented with an additional forcing term originating from the obstacles, and subject to a homogeneous Neumann boundary condition.

In this section and throughout the remainder of the paper, we assume that a function $W$ satisfies the following conditions: 
\begin{enumerate}
	\item[(W1)]  $W \in C^{\infty}(\R)$, $W (s) \geq 0$ for any $s \in \R$, and $W(s)=0$ if and only if $s=\pm 1$. \\
	
	\item[(W2)] There exists $\gamma \in (-1,1)$ such that $W' (s) > 0$ for any $s \in (-1, \gamma)$
	and $W' (s) < 0$ for any $s \in (\gamma, 1)$.\\
	
	\item[(W3)] There exist $\alpha \in (0,1)$ and $\beta >0$ such that
	$W'' (s) \geq \beta $ for any $s \in [-1, -\alpha] \cup [\alpha ,1]$.\\
\end{enumerate}

It should be noted that these conditions basically require $W$ to be W-shaped with two non-degenerate minima at $s= \pm 1$. As a typical example of such $W$ we give  $W(s)=\tfrac{(1-s^2)^2}{2}$, see Fig. \ref{fig:fig2}, for which we can take $\alpha= \tfrac{1}{\sqrt{2}}$, $\beta =1$ and $\gamma =0$. 
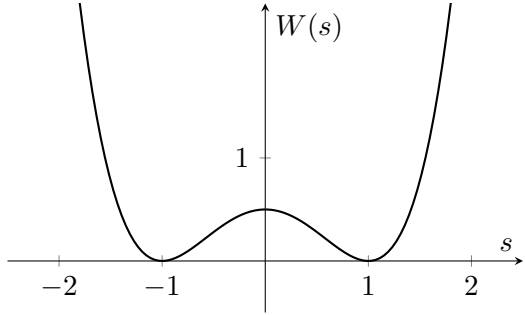
\begin{figure}[h]
	\centering
	\begin{tikzpicture}
		\begin{axis}[
			axis lines = middle,
			xlabel = \( s \),
			ylabel = \( W(s) \),
			xmin = -2.5,
			xmax = 2.5,
			ymin = -0.5, 
			ymax = 2.5, 
			samples = 100,
			domain = -2:2,
			ytick={0,1}, 
			axis equal image, 
			]
			\addplot[black,thick] {(1-x^2)^2/2};
		\end{axis}
	\end{tikzpicture}
	\caption{Plot of the function \( W(s) = \frac{{(1-s^2)^2}}{2} \).}
	\label{fig:fig2}
\end{figure}

According to the basic existence theory for ordinary differential equtaions, there exists a unique function 
$q \in C ^\infty (\R)$ such that
\begin{equation}\label{eq:q}
	q(0)=0, \qquad \lim_{s \to \pm \infty} q(s) = \pm 1, \qquad \text{and} \qquad 
	q' (s) = \sqrt{2W (q(s))}, \quad s \in \R.
\end{equation}
We remark that $q(s) = \tanh s$ in the case of  $W(s)=\tfrac{(1-s^2)^2}{2}$.  For $s\in \R$ and $\varepsilon >0$, we define $q^{\varepsilon} (s) := q(\tfrac{s}{\varepsilon})$. Then $ q^\varepsilon$ satisfies
\begin{equation}\label{eq:prop1q}
	\frac{\varepsilon ( q^\varepsilon_s (s))^2 }{2}
	=\frac{W(q ^{\varepsilon} (s) )}{\varepsilon} 
	\qquad 
	\text{and}
	\qquad
	q^\varepsilon_{ss} (s)
	= \frac{W'(q ^{\varepsilon} (s))}{\varepsilon^2} 
\end{equation}
for any $s \in \R$. It follows from the first equality in \eqref{eq:prop1q} that 
\begin{equation}
\label{def:sigma}
	\int_{\R} \left( 	\frac{\varepsilon ( q^\varepsilon_s (s))^2 }{2} + \frac{W(q ^{\varepsilon} (s) )}{\varepsilon} \right)  \ds 
	=  	\int_{\R}  \sqrt{2W(q ^{\varepsilon} (s))} \,  q^\varepsilon_s (s) \, \ds 
	= \int_{-1}^1  \sqrt{2W(a)} \,  \d a =: \sigma,
\end{equation}
meaning that the Radon measure $\mu_t^\varepsilon$, which is defined below in \eqref{def:dm}, needs to be normalized by $\sigma$. 

Next, by assumpsions (W1)-(W3), we have
\begin{equation}\label{eq:prop3q}
	\frac{W'(s)}{\sqrt{2W (s)}} \leq -\sqrt{\beta}, \quad s \in [\alpha, 1)
	\qquad
	\text{and}
	\qquad
	\frac{W'(s)}{\sqrt{2W (s)}} \geq \sqrt{\beta}, \quad s \in (-1, -\alpha],
\end{equation}
because for any $s \in (-1,-\alpha]$, there exists $t \in (-1,s)$ such that
\begin{equation}\label{eq:CMV}
\frac{(W' (s))^2}{2W(s)}
=
\frac{(W' (s))^2 -(W' (-1)) ^2}{2W(s) -2W(-1)} 
= 
\frac{2W' (t) W'' (t)}{2W'(t)} = W'' (t) \geq \beta,
\end{equation}
due to Cauchy's mean-value theorem. The case of $s \in [\alpha, 1)$ is similar. Let us point out that $\tfrac{W'(s)}{\sqrt{2W (s)}} $ is smooth for $s \in (-1,1)$. 
Introducing $s_\gamma ^\varepsilon \in (-1,1)$ by $q^\varepsilon ( s_\gamma ^\varepsilon ) = \gamma$, we observe that  $s^\varepsilon _\gamma =o(1)$ and that, by assumption (W2), we have 
\begin{equation}\label{eq:prop4q}
W' (q^\varepsilon (s)) >0 ,\quad s \in (-\infty, s_\gamma ^\varepsilon )
\qquad
\text{and}
\qquad
W' (q^\varepsilon (s)) <0 ,\quad s \in ( s_\gamma ^\varepsilon,\infty).
\end{equation}
Lastly, by \eqref{eq:q} we have $s = \int _0 ^{q(s)} \tfrac{1}{\sqrt{2W(t)}} \, \dt$ 
and by \eqref{eq:prop3q} we have  $W(t) \geq \frac{\beta}{2} (1-t)^2$ for any $t \in [\alpha,1]$. 
This implies that 
\begin{equation*}
\begin{split}
s= & \, \int _0 ^\alpha \frac{1}{\sqrt{2W(t)}} \dt
+ \int _\alpha ^{q(s)} \frac{1}{\sqrt{2W(t)}} \dt 
\leq  
\int _0 ^\alpha \frac{1}{\sqrt{2W(t)}} \dt
+ \int _\alpha ^{q(s)} \frac{1}{\sqrt{\beta} (1-t)} \dt \\
= & \, 
\int _0 ^\alpha \frac{1}{\sqrt{2W(t)}} \dt + \frac{1}{\sqrt{\beta}} \log (1-\alpha) 
- \frac{1}{\sqrt{\beta}} \log (1-q(s)),
\end{split}
\end{equation*}
whenever $q(s) \geq \alpha$.
Therefore, there exists a constant $c=c(W,\alpha,\beta)>0$ such that
\begin{equation}\label{eq:prop5q}
0 \leq 1-q(s) \leq c \, e^{-\sqrt{\beta} s}, \qquad s \in[q^{-1}(\alpha),\infty).
\end{equation}
Next let $\{\varepsilon _i \}_{i\in \N}$ be a positive sequence with $\varepsilon _i \to 0$ as $i \to \infty$ and $\varepsilon _i \in (0,1)$ for any $i \in \N$. For simplicity, we often write $\varepsilon _i $ 
as $\varepsilon$.  Let $0< \delta< \min \{ \text{dist} ( O_+,\partial \Omega ), \text{dist} ( O_-,\partial \Omega ) \}$
and let $g^{\varepsilon_i} \in C^\infty(\Omega)$ be a smooth function such that 
\begin{equation}\label{eq:g}
g^{\varepsilon_i}(x) =
\begin{cases}
		c_1 & \text{if $\dist(x,O_+) \leq \sqrt{\varepsilon_i}$}, \\[1mm]
	-	c_1 & \text{if $\dist(x,O_-) \leq \sqrt{\varepsilon_i}$}, \\[1mm]
		0 & \text{if $\min\{\dist(x,O_+),\dist(x,O_-)\} \geq  2\sqrt{\varepsilon_i}$},
\end{cases}
\end{equation}
where
\begin{equation}\label{eq:c1}
 c_1 := \frac{n \left(2R_0+ \frac{3}{2}\delta\right)^4}{\delta R_0\left(R_0+ \frac{3}{4}\delta\right)^3} >1. 
\end{equation}
In addition, we assume that $\sup_{x \in \Omega} |\nabla g^{\varepsilon_i}(x) | \leq c_2 \varepsilon_i^{-\frac{1}{2}}$, 
where the constant $c_2>0$ is not dependent on $i$.
Recall that the constant $R_0$ is defined in assumption (A1) within Section \ref{sec:preliminaries}.  Note that we have  $\sup_{x \in \Omega} |g^{\varepsilon_i}(x) | \leq c_1$.

For $r^{i} \in C^1(\overline{\Omega})$, $i \in \N$, we define the initial datum $u_0^{\varepsilon_i}$ by 
\begin{equation*}
	u_0^{\varepsilon_i }(x):= q^{\varepsilon_i}(r^{i}(x)), \qquad x \in \overline{\Omega}, \, i \in \N. 
\end{equation*}
Moreover, we introduce a Radon measure $\mu_0^{\varepsilon_i}$ on $\Omega$ by setting 
\begin{equation*}
\mu_0^{\varepsilon_i}(\phi) := \frac{1}{\sigma} \int_{\Omega} \phi(x) \left( 	\frac{\varepsilon_i  | \nabla 	u_0^{\varepsilon_i }(x)|^2 }{2} + \frac{W(u_0^{\varepsilon_i }(x) )}{\varepsilon_i} \right) \dx, 
\qquad   \phi \in C(\overline{\Omega}), \, i \in \N. 
\end{equation*}

For  $u_0^{\varepsilon_i }$ and  $\mu_0^{\varepsilon_i }$ we have the following properties, the proof of which is provided in the Appendix. 
\begin{proposition}\label{prop:initial}
	There exists $r^i \in C^1 (\overline \Omega)$, $i \in \N$,  such that the following holds true: 
	\begin{itemize}

		\item[(1)] For sufficiently large $i$, we have
		\begin{equation*}
			\sup _{x \in \Omega} |\nabla r ^i (x)| \leq 1,
			\qquad
			\sup _{x \in \Omega} | u_0 ^{\varepsilon_i} (x) | < 1,
			\qquad
			\sup _{x \in \Omega} | \nabla u_0 ^{\varepsilon_i} (x) | \leq \frac{c_3}{\varepsilon _i},
		\end{equation*}
		and
		\begin{equation*}
			\sup _{x \in \Omega} \left( 
			\frac{\varepsilon _i | \nabla u_0 ^{\varepsilon_i} (x) |^2 }{2}
			-\frac{W(u_0 ^{\varepsilon_i} (x) )}{\varepsilon_i}
			\right) \leq 0,
		\end{equation*}
		where $u_0 ^{\varepsilon_i} (x) = q^{\varepsilon_i} (r^i (x) )$ and where $c_3>0$ is a constant depending only on $W$. 
		
		\item[(2)] $u_0 ^{\varepsilon_i}$ satisfies the homogeneous Neumann boundary condition
		\begin{equation*}
			\nabla u_0 ^{\varepsilon_i} (x) \cdot \nu (x) =0, \qquad x \in \partial \Omega, \ i \in \N.
		\end{equation*}
		\\[-2mm]
		
		\item[(3)] $\mu_0 ^{\varepsilon _i} \rightharpoonup \mathscr{H}^{n-1} \lfloor_{M_0}$
		in $\overline{\Omega}$ as Radon measures, that is, 
		\begin{equation*}
		 \int_\Omega  \phi \, \d \mu_0 ^{\varepsilon _i} \to \int _{M_0} \phi \, \d\mathscr{H}^{n-1} \quad 
			\text{ for any } \ \phi \in C(\overline{\Omega}). 
		\end{equation*}
		\\[-2mm]
		
		\item[(4)] There exists $E_1>0$ such that
		\[
		\sup _{i \in \N}  \mu _0 ^{\varepsilon _i} (\Omega) \leq E_1.
		\] 
		\item[(5)] For sufficiently large $i$, we have 
		\[
		\sup _{x \in \Omega, \, r \in (0,R_2)} \frac{\mu_0 ^{\varepsilon _i} (B_r (x)) }{\omega_{n-1} r^{n-1}} \leq 2 D_0. 
		\]
	\end{itemize}
\end{proposition}
\medskip

 In this paper, we consider the following modified Allen-Cahn equation: 
\begin{equation} \label{prob:AC}
	\begin{cases}
		\varepsilon_i u_t^{\varepsilon_i}  = \varepsilon_i \Delta u^{\varepsilon_i}
		- \dfrac{W'(u^{\varepsilon_i})}{\varepsilon_i} +g^{\varepsilon_i} \, \sqrt{2W(u^{\varepsilon_i})}, 
		& \: (x,t) \in \Omega \times (0,\infty), 
		\\
		 \nabla u^{\varepsilon_i} \cdot \nu = 0,
		&\: (x,t) \in  \partial\Omega \times (0,\infty),
		\\
		u^{\varepsilon_i}(0,x) = u_0^{\varepsilon_i}(x). 
		&\: x \in \Omega.
	\end{cases}
\end{equation}
\smallskip

\begin{remark}
 According to Proposition \ref{prop:initial}, we have $ \sup_{x \in \Omega} |u_0^{\varepsilon}(x)| <1$, and by applying  the maximum principle  we obtain $\sup_{x \in \Omega, \, t\in [0, T)} |u^{\varepsilon}(x,t)| < 1$ for the solution 
$u^{\varepsilon}$ to \eqref{prob:AC} and  for $T\in (0,\infty)$. This is proven in Proposition 
\ref{prop:cp} below. From this we can define a function $r^{\varepsilon}: \overline{\Omega} 
\times [0,T) \to \R$ by 
\begin{equation*}
	u^{\varepsilon}(x,t) = q^{\varepsilon}(r^{\varepsilon}(x,t)), 
\end{equation*}
that is, $r^{\varepsilon}(x,t) = (q^{\varepsilon})^{-1} (u^{\varepsilon}(x,t))$ since $q^{\varepsilon}: \R 
\to (-1,1)$ is bijective.  
\end{remark}



\section{Well-posedness of the Allen-Cahn equation  (\ref{prob:AC})}
\label{sec:wellposedness}
In this section, we establish that the Allen-Cahn equation \eqref{prob:AC} admits a unique global solution. Furthermore, we present a pointwise estimate for this solution. For $p\in (1,\infty)$, we introduce a bounded linear operator $A \in \mathcal{L}( W^2_{p,\mathcal{B}}(\Omega), L_p(\Omega))$ by setting
\[
Au:=-\Delta u, \quad \text{ for all } u \in W^2_{p,\mathcal{B}}(\Omega):= \{ u \in W^2_p(\Omega) \mid \nabla u \cdot \nu =0 \text{ on } \partial \Omega\}.
\]
It is well-known, e.g. \cite[Section 3.8]{tanabe}, that $-A$ generates an analytic semigroup on $L_p(\Omega)$.  Let $\varrho \in (0,1)$. Consider the smooth functions  $H^\varrho, \, J^\varrho : \mathbb{R} \to \mathbb{R}$, defined by 
\[ 
		H^\varrho (s) := 
	\begin{cases}
	\sqrt{2W(s)} \quad &\text{ if } |s| \leq 1-\varrho, 
		\\[1mm]
		0 		\quad &\text{ if } |s| >1, 
	\end{cases}
\quad \text{ and } \quad 
	J^\varrho (s) := 
\begin{cases}
	W'(s) \quad &\text{ if } |s| \leq 1-\varrho, 
	\\[1mm]
	0 		\quad &\text{ if } |s| > 1.
\end{cases}
\]
Setting  $F^{\varepsilon,\varrho}(x,s) := -\varepsilon^{-2} J^\varrho(s) + \varepsilon^{-1} g^\varepsilon(x) H^\varrho(s)$, we have that it belongs to $C^\infty (\overline{\Omega}\times \R,\R)$ and that both 
$|F^{\varepsilon,\varrho}(x,s)|$ and  $|\partial_s F^{\varepsilon,\varrho}(x,s)|$ are bounded by 
a positive constant for any $(x,s) \in \overline{\Omega}\times \R$. We can then apply \cite[Theorem 16.3]{amann} which immediately implies the following result: 

\begin{theorem}\label{thm:wp}
	Given any initial value $u_0^\varepsilon \in L_p(\Omega)$, the semilinear parabolic 
	boundary value problem
\begin{equation} \label{prob:modAC}
		\begin{cases}
			u_t^{\varepsilon,\varrho} - \Delta u^{\varepsilon, \varrho}
			= F^{\varepsilon,\varrho}(\cdot,u^{\varepsilon, \varrho}),
			& \: (x,t) \in \Omega \times (0,\infty), 
			\\
			\nabla u^{\varepsilon,\varrho} \cdot \nu = 0,
			&\: (x,t) \in  \partial\Omega \times (0,\infty),
			\\
			u^{\varepsilon,\varrho}(0,x) = u_0^{\varepsilon}(x), 
			&\: x \in \Omega
		\end{cases}
	\end{equation}
possesses a unique solution 
\begin{equation*}
	u^{\varepsilon,\varrho} \in C([0,\infty), L_p(\Omega))\cap C^{2,1}(\overline{\Omega}\times (0,\infty), \R). 
\end{equation*}
\end{theorem}
\bigskip

\begin{remark}
	The solution $u^{\varepsilon,\varrho}$ is smooth on $\overline{\Omega}\times (0,\infty)$ if the initial value $u_0^\varepsilon$ is smooth.  
\end{remark}

The classeical maximum principle implies the following a priori estimate.

\begin{proposition}\label{prop:cp}
	Let $u^\varepsilon$ be a solution to \eqref{prob:AC}. Then 
\[
\sup_{x \in \Omega, \, t\in [0, T)} |u^{\varepsilon}(x,t)| < 1 \quad \text{ for any } T\in(0,\infty).
\]
\end{proposition}


\begin{proof}
We only prove that $\sup_{x \in \Omega, \, t\in [0, T)} u^{\varepsilon}(x,t) <1$, since 
$u^{\varepsilon} >-1$ is similarly proved. Assume that 
	\begin{equation*}
		\{t \in [0,T) \mid \text{ there exists} \ x \in \Omega \ 
		\text{such that} \ u^{\varepsilon}(x,t) =1\} \neq \emptyset
	\end{equation*}
	and put 
	\begin{equation*}
		t_0 := \inf \{ t \in [0,T) \mid \text{ there exists}
		\ x \in \Omega \  \text{such that} \ u^{\varepsilon}(x,t) =1\}.  
	\end{equation*}
	Then $t_0 \in (0,T)$ by $\sup_{x \in \Omega} u_0^{\varepsilon}(x) <1$ according to Proposition \ref{prop:initial}.  Let $u^{\varepsilon}_+$ 
	be a solution to 
	\begin{equation} \label{prob:supersol}
		\varepsilon (u^{\varepsilon}_+)_t = \varepsilon \Delta u^{\varepsilon}_+
		- \dfrac{W'(u^{\varepsilon}_+)}{\varepsilon} +c_1 \, \sqrt{2W(u^{\varepsilon}_+)}
		\quad \text{ in }  \  \Omega \times (0,t_0)
	\end{equation}
	with boundary condition $ \nabla u^{\varepsilon}_+ \cdot \nu  = 0$ on 
	$\partial\Omega \times (0,t_0)$ and initial datum $u^{\varepsilon}_+(x,0)= \sup_{x \in \Omega} 
	u_0^{\varepsilon}(x)$, where $c_1$ is defined in \eqref{eq:c1}. 
	Then $u^{\varepsilon}_+$ is a supersolution to 
	\eqref{prob:AC} in $\Omega \times (0,t_0)$. Since the initial datum is constant 
	and the fact that there exists $c_+=c_+ (\varepsilon ,W,c_1)>0$ such that
	$- \dfrac{W'(s)}{\varepsilon} +c_1 \, \sqrt{2W(s)} \leq c_+ (1-s)$ for any $s \in(0,1)$, 
	we can easily verify that 
	the solution $u^{\varepsilon}_+$ to \eqref{prob:supersol} depends only on $t$ and 
	satisfies $u^{\varepsilon}_+(t)<1$ for any $t\in (0,t_0)$. Therefore, applying the maximum 
	principle \cite[Theorem 2.9]{lieberman}, 
	we deduce that $u^{\varepsilon}(x,t) \leq u^{\varepsilon}_+(t)<1$ for any 
	$(x,t) \in \Omega\times [0,t_0)$.  This is a contradiction to $u^{\varepsilon}(x,t_0)=1$ for 
	some $x \in \Omega$ and thus $u^{\varepsilon}(x,t)<1$ for 
	any $(x,t) \in \Omega \times [0,T)$. Using the maximum principle 
	again, we obtain $\sup_{x \in \Omega, \, t\in [0, T)} |u^{\varepsilon}(x,t)| <1$. 
\end{proof}

We are now in position to prove the following result.

\begin{theorem}\label{thm:regularity solution}
For every initial datum $ u_0 ^\varepsilon \in C^1 (\overline{\Omega}) $ with 
$\nabla u _0 ^\varepsilon \cdot \nu =0$ on $\partial \Omega$
and $\sup _{x \in \Omega} |u_0 ^\varepsilon (x) | <1$, as constructed in Proposition \ref{prop:initial}, the Allen-Cahn equation  \eqref{prob:AC}
has a unique solution 
\begin{equation*}
	u^{\varepsilon} \in C([0,\infty), L_p(\Omega))\cap C^{2,1}(\overline{\Omega}\times (0,\infty), \R).
\end{equation*}
\end{theorem}

\begin{proof}
Let $u^{\varepsilon, \varrho}$ be the solution to \eqref{prob:modAC}
with initial value $ u_0 ^\varepsilon \in C^1 (\overline{\Omega}) $ as given in the statement. 
Let $u^{\varepsilon}_+$ be the supersolution to \eqref{prob:AC} as constructed in  Proposition \ref{prop:cp}, and let $\{ T_k\} _{k=1} ^\infty$ be a positive sequence such that $T_1 < T_2 < \cdots \to \infty$ as $k\to \infty$. 
For any $T_k>0$, there exists $\delta_k >0$ such that 
$\sup _{t \in [0,T_k )} u^{\varepsilon }_+ (t) <1-\delta_k $.
Set $\varrho_k = \frac12 \delta_k$.
Then $u^{\varepsilon}_+$ is also the supersolution to \eqref{prob:modAC} with $\varrho = \varrho _k$,
up to the time $t=T_k$, because $0\leq u^{\varepsilon}_+ \leq 1-\varrho_k$
implies $H^{\varrho_k} (u^{\varepsilon}_+) = \sqrt{2W (u^{\varepsilon}_+)}$ and
$J^{\varrho_k} (u^{\varepsilon}_+) = W' (u^{\varepsilon}_+)$.
Therefore, the maximum principle implies
$\sup _{(x,t) \in \Omega \times [0,T_k)} u^{\varepsilon, \varrho_k} (x,t) <1-\delta_k <1-\varrho_k$. 
Repeating the same argument for the subsolution, we have
$\sup _{(x,t) \in \Omega \times [0,T_k)} | u^{\varepsilon, \varrho_k } (x,t) |<1-\varrho_k$,
by taking a smaller $\varrho_k$, if necessary. 
Hence $u^{\varepsilon,\varrho_k} $ is a solution to
\eqref{prob:AC} in $\Omega \times (0,T_k)$. Note that
$u^{\varepsilon, \varrho_k } = u^{\varepsilon, \varrho_{k+1} }$ in $\Omega \times [0,T_k)$
by the maximum principle and we may define 
$u^{\varepsilon} (x,t) := u^{\varepsilon, \varrho_k } (x,t)$ where $T_k > t$.
Then $u^{\varepsilon}$ satisfies \eqref{prob:AC}.
\end{proof}


\section{Energy estimates}
\label{sec:energyestimates}
In this section we prepare some results which are  crucial for the limiting process of  \eqref{prob:AC} as $i \rightarrow \infty$. Let us first define the measures that correspond to the surface $M_t$ from Section \ref{sec:intro}.

\begin{definition}
	\label{def:rmdm}
Recall that $\sigma$ is defined in \eqref{def:sigma}. Let  $u^{\varepsilon_i}$ be a solution for \eqref{prob:AC}.  We define two Radon measures $\mu_t^{\varepsilon_i}$ and $\xi_t^{\varepsilon_i}$ in the following way: 
\begin{equation}
\label{def:rm}
	\mu_t^{\varepsilon_i}(\phi) := \frac{1}{\sigma} \int_{\Omega} \phi(x) \left( 	\frac{\varepsilon_i  | \nabla 	u^{\varepsilon_i }(x,t)|^2 }{2} + \frac{W(u^{\varepsilon_i }(x,t) )}{\varepsilon_i} \right) \dx
\end{equation}
and 
\begin{equation}
\label{def:dm}
	\xi_t^{\varepsilon_i}(\phi) := \frac{1}{\sigma} \int_{\Omega} \phi(x) \left( 	\frac{\varepsilon_i  | \nabla 	u^{\varepsilon_i }(x,t)|^2 }{2} - \frac{W(u^{\varepsilon_i }(x,t) )}{\varepsilon_i} \right) \dx
\end{equation}
for any $\phi \in C(\overline{\Omega})$. 
The measure $\xi_t^{\varepsilon_i}$ is called the {\it discrepancy measure}. 
\end{definition}

In the following subsections, we establish the monotonicity formula and show a standard energy estimate for \eqref{prob:AC}. We also derive an upper bound for the discrepancy measure $\xi_t^{\varepsilon_i}$ in terms of a negative power of $\varepsilon_i$. Combining these results we then prove the upper density ratio bound for  $	\mu_t^{\varepsilon_i}$ which is independent of  $\varepsilon_i$. 
\subsection{Up-to the boundary monotonicity formula}
Our first task is to obtain some up-to the boundary monotonicity formula for  the Allen-Cahn equation \eqref{prob:AC}. 

Note that the monotonicity formula for the MCF in $\R^n$ without a boundary and no obstacles was established by Huisken \cite{huisken}, and its adaptation for the corresponding Allen-Cahn equation was demonstrated by Ilmanen \cite{ilmanen}. When considering the Allen-Cahn equation \eqref{prob:AC} without obstacles (i.e., $g^\varepsilon=0$) but including boundary conditions, Mizuno and Tonegawa \cite{mizuno-tonegawa} derived the boundary monotonicity formula specifically for convex domains, employing a boundary reflection method. Subsequently, Kagaya \cite{kagaya} extended their findings to include non-convex domains as well.

To formulate our statement and to apply the reflection argument as done in \cite{mizuno-tonegawa}, we introduce additional notation, following closely the conventions of \cite{kagaya,mizuno-tonegawa}.  Let 
\begin{equation*}
	\kappa := \Vert \text{principal curvature of } \partial \Omega \Vert_{L_\infty(\partial \Omega )}
\end{equation*}
and set, for $d>0$, 
\begin{equation*}
	N_d :=\{x \in \R^n \mid \dist (x,  \partial \Omega)<d \}, 
\end{equation*}
which denotes the interior tubular neighborhood of $\partial \Omega$. 
Because $ \partial \Omega$ is smooth and compact, we have $0<\tfrac{1}{\kappa}<\infty$. 
Moreover, there exists a sufficiently small constant 
\begin{equation*}
c_4 \in (0,\tfrac{1}{6\kappa}]
\end{equation*}
only depending on $\partial \Omega$  and $O_\pm$ such that for any point $x \in N_{6c_4}$ there 
uniquely exists a point $\zeta(x)\in \partial \Omega$ such that
 $ \dist (x,  \partial \Omega)= |x-\zeta(x)|$, and $g^{\varepsilon _i} \equiv 0$ on $N_{c_4} \cap \overline{\Omega}$ for any $i \in \N$. 
 Thus, we define the reflection point $\tilde{x}$ of $x$ 
 with respect to $\partial \Omega$ as 
\begin{equation*}
	\tilde{x}:= 2\zeta(x)-x
\end{equation*}
 and the reflection ball $\tilde{B}_r(a)$ of $B_r(a)$ with respect to  $\partial \Omega$ as
\begin{equation*} 
  \tilde{B}_r(a):= \{x \in \R^n \mid  |\tilde{x}-a|<r\}.
\end{equation*}
For a detailed figure of $N_d$, $\tilde{x}$, and $ \tilde{B}_r(a)$, we refer to \cite{kagaya,mizuno-tonegawa}.  We next fix a cut-off function $\eta = \eta(|x|) \in C^\infty ([0,\infty))$ such that 
\begin{equation*}
	0 \leq \eta \leq 1, \quad \eta ' \leq 0, \quad \text{spt}\eta \subset [0,\tfrac{c_4}{2}), \quad 
	\eta =1 \text{ on } [0,\tfrac{c_4}{4}]. 
\end{equation*}
For $0<t<s$ and $x,y \in \R^n$, we define the $(n-1)$-dimensional backward heat kernel 
$\rho_{(y,s)}(x,t)$  as
\begin{equation}\label{eq:hkernel}
	\rho_{(y,s)}(x,t):= \frac{1}{\big( 4\pi(s-t) \big)^{\frac{n-1}{2}}} \, e^{-\frac{|x-y|^2}{4(s-t)}}.
\end{equation}
For $0<t<s$ and $x,y \in N_{c_4}$, we introduce the reflected backward heat kernel 
$\tilde{\rho}_{(y,s)}(x,t):= \rho_{(y,s)}(\tilde{x},t)$ and define a truncated version of 
$\rho_{(y,s)}$ and $\tilde{\rho}_{(y,s)}$ by setting
\begin{equation}\label{eq:trunchkernel}
	\rho_{1,(y,s)}(x,t):= \eta(|x-y|) 	\rho_{(y,s)}(x,t), 
	\qquad \rho_{2,(y,s)}(x,t) := 	\eta(|\tilde{x}-y|) \tilde{\rho}_{(y,s)}(x,t). 
\end{equation}
Observe that, for  $x \in N_{2c_4}\setminus N_{c_4}$ and $y \in N_{\frac{c_4}{2}}$, 
\begin{equation*}
	|\tilde{x}-y| \geq |\tilde{x}-\zeta(y)| - |\zeta(y)-y|  > c_4 - \frac{c_4}{2}=  \frac{c_4}{2}
\end{equation*}
and hence we may smoothly define $\rho_{2,(y,s)}=0$ for $x\in \Omega \setminus N_{c_4}$ and $y\in N_{\frac{c_4}{2}}$. 

\begin{proposition}[Up-to the boundary monotonicity formula]\label{prop:mf}
	Let $u^{\varepsilon}$ be a solution to \eqref{prob:AC}. There exist constants $0<c_5,c_6< \infty$ depending only on $n$, $E_1$ and $\partial \Omega$ such that 
	\begin{equation}\label{eq:mformula_b}
		\begin{split}
		\frac{\d}{\dt}&\left( e^{c_5(s-t)^\frac{1}{4}} \int_{\Omega} \big( \rho_{1,(y,s)}(x,t) + \rho_{2,(y,s)}(x,t) \big)\,   \d \mu_t^{\varepsilon}(x)\right)
		\\
		& \leq  e^{c_5(s-t)^\frac{1}{4}} \left( c_6 + \int_{\Omega} \frac{\rho_{1,(y,s)}(x,t) + \rho_{2,(y,s)}(x,t)}{2(s-t)} \, \d\xi_t^{\varepsilon}(x)\right)
		\end{split}
	\end{equation}
for all $0<t<s$ and $y \in N_{\frac{c_4}{2}}\cap \overline{\Omega}$. Furthermore, there exists a constant $0<c_7<\infty$ depending only on $n$, $c_1$, $W$ and  $\partial \Omega$ such that 
\begin{equation}\label{eq:mformula_i}
	\begin{split}
		\frac{\d}{\dt} &\left( e^{c_7 (s-t)} \int_{\Omega} \rho_{1,(y,s)}(x,t) \,    \d \mu_t^{\varepsilon}(x) \right) 
		\\
		& \leq e^{c_7 (s-t)} \left( \int_{\Omega} \frac{\rho_{1,(y,s)}(x,t)}{2(s-t)} \, \d\xi_t^{\varepsilon}(x) 
		 + c_7 e^{- \frac{c_4^2}{64(s-t)}}  \mu_t^{\varepsilon}\big( B_{\frac{c_4}{2}}(y)  \big)\right)
		 \end{split}
\end{equation}
for all $0<t<s$ and $y \in \Omega \setminus N_{\frac{c_4}{2}}$. 

\end{proposition}

\begin{proof}
We first verify the boundary monotonicity formula \eqref{eq:mformula_b}. To do so, we observe that  $g^\varepsilon =0 $ for $y \in N_{c_4}\cap \overline{\Omega}$ 
and that the convexity of $\Omega$ is not used in the proof of \eqref{eq:mformula_b} in \cite[Proposition 3.1]{mizuno-tonegawa}. Hence we can refer to \cite{mizuno-tonegawa}  for the detailed proof. To verify the interior monotonicity formula \eqref{eq:mformula_i}, we follow similar arguments as in the proof of \cite[Proposition 4.3]{takasao-tonegawa}. We note that the authors of \cite{takasao-tonegawa} study the singular limit problem of an Allen-Cahn type equation with a transport term on $\mathbb{T}^n := (\mathbb{R}/\mathbb{N})^n$ or $\mathbb{R}^n$, hence we have to expand their argument for our Neumann problem with forcing term. 
We define
\begin{equation*}
	L^\varepsilon := u_t^\varepsilon - g^\varepsilon \frac{\sqrt{2W(u^\varepsilon)}}{\varepsilon}
\end{equation*}
and note that by the first equation in \eqref{prob:AC} we have $L^\varepsilon= \Delta u^\varepsilon -   \tfrac{W'(u^\varepsilon)}{\varepsilon^2}$. To simplify notation, we write in what follows $ \rho_1 = \rho_{1,(y,s)}(x,t)$. By integration by parts and Young's inequality, we obtain 
\begin{align} \label{eq:mon1}
	\nonumber
	& \sigma \, \frac{\d}{\dt} \int_{\Omega} \rho_1 \, \d \mu_t^\varepsilon (x) = 	\frac{\d}{\dt} \int_{\Omega} \rho_1 \left( \frac{\varepsilon |\nabla u^\varepsilon|^2}{2} + \frac{W(u^\varepsilon)}{\varepsilon}\right) \dx
	\\
	\nonumber 
	&= \int_{\Omega} \left\{ \partial_t \rho_1  \left( \frac{\varepsilon |\nabla u^\varepsilon|^2}{2} + \frac{W(u^\varepsilon)}{\varepsilon}\right) + \rho_1 \left( \varepsilon \nabla u^\varepsilon \cdot 
	\nabla u_t^\varepsilon  +\frac{W'(u^\varepsilon)}{\varepsilon^2} \,  u_t^\varepsilon \right) \right\}
	\dx
	\\
	\nonumber 
	&= \int_{\Omega}  \left\{ \partial_t \rho_1  \left( \frac{\varepsilon |\nabla u^\varepsilon|^2}{2} + \frac{W(u^\varepsilon)}{\varepsilon}\right) + \rho_1  \varepsilon \left( - \Delta u^\varepsilon + 
	\frac{W'(u^\varepsilon)}{\varepsilon^2} \right) u_t^\varepsilon - \varepsilon ( \nabla \rho_1 \cdot \nabla u^\varepsilon ) u_t^\varepsilon  \right\} \dx
	\\
	\nonumber 
	&= \int_{\Omega}  \bigg\{ \partial_t \rho_1  \left( \frac{\varepsilon |\nabla u^\varepsilon|^2}{2} + \frac{W(u^\varepsilon)}{\varepsilon}\right) + \rho_1 \varepsilon (-L^\varepsilon) 
	\left( L^\varepsilon + g^\varepsilon \frac{\sqrt{2W(u^\varepsilon)}}{\varepsilon}\right) 
	    \\& \hspace{1.5cm} 
       - \varepsilon ( \nabla \rho_1 \cdot \nabla u^\varepsilon ) 	\left( L^\varepsilon + g^\varepsilon \frac{\sqrt{2W(u^\varepsilon)}}{\varepsilon}\right) \bigg\} \dx
    \\
    \nonumber 
   &= \int_{\Omega} \left\{ \partial_t \rho_1  \left( \frac{\varepsilon |\nabla u^\varepsilon|^2}{2} + \frac{W(u^\varepsilon)}{\varepsilon}\right) 
   -   \varepsilon  \rho_1 	\left( L^\varepsilon + g^\varepsilon \frac{\sqrt{2W(u^\varepsilon)}}{\varepsilon}\right) 	\left( L^\varepsilon + \frac{\nabla \rho_1 \cdot \nabla u^\varepsilon}{\rho_1}\right) \right\} \dx 
   \\
   \nonumber 
   &= \int_{\Omega} \bigg\{ \partial_t \rho_1  \left( \frac{\varepsilon |\nabla u^\varepsilon|^2}{2} + \frac{W(u^\varepsilon)}{\varepsilon}\right) 
    -   \varepsilon  \rho_1  \left( L^\varepsilon + \frac{\nabla \rho_1 \cdot \nabla u^\varepsilon}{\rho_1}\right)^2 + \varepsilon \left(  L^\varepsilon \nabla \rho_1 \cdot \nabla u^\varepsilon +   \frac{(\nabla \rho_1 \cdot \nabla u^\varepsilon)^2}{\rho_1} \right)
        \\& \hspace{1.5cm} \nonumber 
       -    \varepsilon   \rho_1   g^\varepsilon \frac{\sqrt{2W(u^\varepsilon)}}{\varepsilon} 
        \left( L^\varepsilon + \frac{\nabla \rho_1 \cdot \nabla u^\varepsilon}{\rho_1}\right)
           \bigg\} \dx 
      \\
      \nonumber 
  &\leq  \int_{\Omega} \left\{ \partial_t \rho_1  \left( \frac{\varepsilon |\nabla u^\varepsilon|^2}{2} + \frac{W(u^\varepsilon)}{\varepsilon}\right) + \varepsilon \left(  L^\varepsilon \nabla \rho_1 \cdot \nabla u^\varepsilon +   \frac{(\nabla \rho_1 \cdot \nabla u^\varepsilon)^2}{\rho_1} \right) 
  + \frac{1}{2\varepsilon} \, \rho_1 (g^\varepsilon)^2 W(u^\varepsilon)  \right\} \dx . 
\end{align}
Moreover, integrating by parts again, we see that 
\begin{equation*}
	 \int_{\Omega}  \varepsilon  L^\varepsilon \nabla \rho_1 \cdot \nabla u^\varepsilon \, \dx 
	 = \int_{\Omega} \left\{ - \varepsilon (\nabla u^\varepsilon \otimes \nabla u^\varepsilon) 
	 \cdot \nabla^2 \rho_1 + \left( \frac{\varepsilon |\nabla u^\varepsilon|^2}{2} + \frac{W(u^\varepsilon)}{\varepsilon}\right) \Delta \rho_1  \right\} \dx , 
\end{equation*}
where the boundary integral vanishes due to $\text{spt} \rho_1 \subset \Omega$ for $y \in 
\Omega \setminus N_{\frac{c_4}{2}}$. Inserting this equation back into \eqref{eq:mon1} and using the fact that $\sigma \mu_t^\varepsilon = \varepsilon |\nabla u^\varepsilon|^2 - \sigma \xi_t^\varepsilon $, we deduce
\begin{align} \label{eq:mon2}
	\nonumber
	\sigma \ \frac{\d}{\dt} \int_{\Omega} \rho_1 \, \d \mu_t^\varepsilon (x)  
	 &\leq - \sigma  \int_{\Omega}    \big( \partial_t \rho_1  +  \Delta \rho_1  \big) \,  \d \xi_t^\varepsilon (x) 
	 + \int_{\Omega} \biggl\{ \varepsilon |\nabla u^\varepsilon|^2  \bigg( \partial_t \rho_1  +  \Delta \rho_1 
	 \\
	 &  \hspace{1.5cm}  - 
	\frac{\nabla u^\varepsilon \otimes \nabla u^\varepsilon} {|\nabla u^\varepsilon|^2} \cdot 
	\nabla^2 \rho_1 
	+  \frac{(\nabla \rho_1 \cdot \nabla u^\varepsilon)^2}{\rho_1|\nabla u^\varepsilon|^2}   \bigg)   + \frac{1}{2\varepsilon} \,  \rho_1 (g^\varepsilon)^2 W(u^\varepsilon) \biggr\} \,  \dx. 
\end{align}
Let us note that the backward heat kernel $\rho = \rho_{(y,s)}(x,t)$ satisfies 
\begin{equation*}
	\partial_t \rho +  \Delta \rho = - \frac{\rho}{ 2(s-t)}, \qquad
	\partial_t \rho +  \Delta \rho  - \frac{\nabla u^\varepsilon \otimes \nabla u^\varepsilon} {|\nabla u^\varepsilon|^2} \cdot 
	\nabla^2 \rho
	+  \frac{(\nabla \rho \cdot \nabla u^\varepsilon)^2}{\rho|\nabla u^\varepsilon|^2} =0,
\end{equation*}
and when computing this with $\rho_1$ instead of $\rho$ we get additional terms from the differentiation of $\eta$. Integrating these terms we see that they can be bounded by 
$c \mu_t^\varepsilon  (B_{\frac{c_4}{2}}(y)) e^{- \frac{c_4^2}{64(s-t)}}$ with a constant $c=c(n,c_4)>0$ since $|\nabla^j \rho| \leq c  e^{- \frac{c_4^2}{64(s-t)}}$ for any $x,y\in \Omega$ with $|x-y|>\tfrac{1}{4}$ and $j=0,1$. It follows then from \eqref{eq:mon2} that 
\begin{align*}
\sigma \	 \frac{\d}{\dt} \int_{\Omega} \rho_1 \, \d \mu_t^\varepsilon (x)  
	  \leq 
	  \sigma  \int_{\Omega}\frac{ \rho_1}{2(s-t)} \, \d \xi_t^\varepsilon (x) + c \mu_t^\varepsilon  (B_{\frac{c_4}{2}}(y)) e^{- \frac{c_4^2}{64(s-t)}} + \frac{\sigma}{2}  c_1^2\int_{\Omega} 
	   \rho_1 \, \d \mu_t^\varepsilon (x)  . 
\end{align*}
Hence, with a suitable choice of $c_7$ depending only on $n$, $c_1$, $W$, and $\partial \Omega$, we 
get \eqref{eq:mformula_i}. 
\end{proof}

\subsection{Energy (in)equality}\label{subsec:energyinequality}
Next we show a standard energy equality and estimate for \eqref{prob:AC}. 
We let $k(s):= \int_0^s \sqrt{2 W(a)} \, \d a$ and set 
\begin{equation*}
	E^\varepsilon (t):= \int_{\Omega} \left( \frac{\varepsilon  | \nabla 	u^{\varepsilon }(x,t)|^2 }{2} + \frac{W(u^{\varepsilon}(x,t) )}{\varepsilon} \right) \dx - \int_{\Omega} g^\varepsilon (x) \, k(u^{\varepsilon }(x,t)) \, \dx, \quad  t \in [0,\infty),
\end{equation*}
 for the solution $u^{\varepsilon }$ to \eqref{prob:AC}. 

\begin{proposition}\label{prop:energyestimate}
Let $u^{\varepsilon }$ be a solution to \eqref{prob:AC}. For any $T\in (0,\infty)$ it holds that
\begin{equation*}
	E^\varepsilon (T) + \int_0^T \int_\Omega \varepsilon \,  (u^{\varepsilon }_t)^2 \, \dx\,  \dt = 	E^\varepsilon (0)
\end{equation*}
and 
\begin{equation*}
	\sup_{\varepsilon \in (0,1)} E^\varepsilon (T) \leq \sup_{\varepsilon \in (0,1)} E^\varepsilon (0)
	\leq c_8
\end{equation*}
for some positive constant $c_8$ depending only on $n, E_1, c_1, W,$ and $\Omega$.  
\end{proposition}

\begin{proof}
By integration by parts we get, for  $t\in (0,\infty)$, 
\begin{align*}
	\frac{\d}{\dt}  E^\varepsilon (t) &= \int_{\Omega} \left( \varepsilon \, \nabla u^{\varepsilon } 
	\cdot \nabla u^{\varepsilon }_t + \frac{W'(u^{\varepsilon })}{\varepsilon} \, u^{\varepsilon }_t \right) \dx - \int_{\Omega} g^\varepsilon  \, \sqrt{2 W(u^{\varepsilon })} \, u^{\varepsilon }_t \, \dx
	\\
	&=  \int_{\Omega} \left( - \varepsilon \, \Delta u^{\varepsilon } + \frac{W'(u^{\varepsilon })}{\varepsilon} - g^\varepsilon  \, \sqrt{2 W(u^{\varepsilon })}   \right) u^{\varepsilon }_t  \, \dx 
	= -  \int_{\Omega}  \varepsilon \,  (u^{\varepsilon }_t)^2 \, \dx, 
\end{align*}
which implies that the energy $E^\varepsilon (\cdot)$ is decreasing. Integrating this equality, the energy equality follows. For the second statement, due to Proposition \ref{prop:initial} and the assumption (W1), there exists a constant $c>0$ such that $|k(u^{\varepsilon }_0(x))| \leq c$ for $x \in \Omega$, hence we have 
 \begin{align*}
 E^\varepsilon (0) &= \int_{\Omega} \left( \frac{\varepsilon  | \nabla 	u^{\varepsilon }_0(x)|^2 }{2} + \frac{W(u^{\varepsilon}_0(x) )}{\varepsilon} \right) \dx - \int_{\Omega} g^\varepsilon (x) \, k(u^{\varepsilon }_0(x)) \, \dx
 \\
 & \leq \sigma \, \mu_0^\varepsilon(\Omega) + c_1 \int_\Omega |k(u^{\varepsilon }_0(x))| \, \dx
 \leq \sigma \, \mu_0^\varepsilon(\Omega) + c_1 c 	\mathscr{L}^n(\Omega). 
 \end{align*}
Applying Proposition \ref{prop:initial} again finishes the proof.
\end{proof}

By an argument similar to that in the proof of Proposition \ref{prop:energyestimate}, we can prove the following. 

\begin{corollary}
\label{cor:boundmeas}
Let $u^{\varepsilon }$ be a solution to \eqref{prob:AC}.  For any $T \in (0,\infty)$, we have that
\begin{equation*}
	\mu^\varepsilon_T (\Omega) + \frac{1}{\sigma} \int_0^T \int_\Omega \varepsilon \,  (u^{\varepsilon }_t)^2 \, \dx\,  \dt = 		\mu^\varepsilon_0 (\Omega) + \frac{1}{\sigma} \int_{\Omega} g^\varepsilon(x)\big( k(u^\varepsilon(x,T)) - k(u^\varepsilon_0(x))  \big)\, \d x ,
\end{equation*}
and there exists a positive constant $c_9$ depending only on $n,E_1,c_1,W$, and $\Omega$ such  that 
\begin{equation*}
	\sup_{\varepsilon \in (0,1)} 	\mu^\varepsilon_T (\Omega) \leq c_9. 
\end{equation*}
\end{corollary}

\begin{proof}
By the definition of $E_\varepsilon$ and Proposition \ref{prop:energyestimate}, we find that 
\begin{equation*}
\begin{split}
	\mu^\varepsilon_T (\Omega) + \frac{1}{\sigma} \int_0^T \int_\Omega \varepsilon \,  (u^{\varepsilon }_t)^2 \, \dx\,  \dt &= \frac{1}{\sigma} E_\varepsilon(T) + \frac{1}{\sigma} \int_{\Omega} g^\varepsilon (x) \, k(u^{\varepsilon }(x,T)) \, \dx + \frac{1}{\sigma} \int_0^T \int_\Omega \varepsilon \,  (u^{\varepsilon }_t)^2 \, \dx\,  \dt
	\\
	&= \frac{1}{\sigma} E_\varepsilon(0) + \frac{1}{\sigma} \int_{\Omega} g^\varepsilon (x) \, k(u^{\varepsilon }(x,T)) \, \dx
	\\
	&=\mu^\varepsilon_0 (\Omega) + \frac{1}{\sigma} \int_{\Omega} g^\varepsilon(x)\big( k(u^\varepsilon(x,T)) - k(u^\varepsilon_0(x))  \big)\, \d x,
\end{split}
\end{equation*}
so the first statement of the corollary is proved. Now, the fact that $\Vert g^\varepsilon \Vert_{L_\infty (\Omega)}\leq c_1$, Proposition \ref{prop:cp}, and arguing as in the proof of Proposition \ref{prop:energyestimate}, we infer that 
\begin{equation*}
\begin{split}
\mu^\varepsilon_T (\Omega) &\leq \mu^\varepsilon_0 (\Omega) + \frac{1}{\sigma} c_1 \int_{\Omega} \big( \vert k(u^\varepsilon (x,T)) \vert + \vert k(u^\varepsilon_0 (x)) \vert \big) \, dx 
\\
&\leq  \mu^\varepsilon_0 (\Omega) + \frac{1}{\sigma} c_1 c \mathscr{L}^n(\Omega)
\end{split}
\end{equation*}
for some constant $c>0$. So, the proof is complete by applying Proposition \ref{prop:initial} (4). 
\end{proof}

\begin{remark}
One can prove the monotone decreasing property for 
the limit of $\mu_t ^\varepsilon (\Omega)$ with respect to $t$, as outlined in Corollary \ref{cor:monotoneofmu}.
\end{remark}

\subsection{Upper bound for the discrepancy measure} We now estimate the growth rate of the discrepancy measure. To ease the notation, let us put
\begin{equation*}
	\xi^\varepsilon (x,t) := 
	\frac{\varepsilon |\nabla u^\varepsilon(x,t)|^2}{2} - \frac{W(u^\varepsilon(x,t))}{\varepsilon}, 
	\quad (x,t)\in \overline{\Omega}\times[0,\infty),  
\end{equation*}
for the solution $u^\varepsilon$ to \eqref{prob:AC}. 
In the sequel, we use a parabolic rescaling: Let 
\begin{equation}\label{eq:redomain}
	\Omega^\varepsilon:= \{ \hat{x} \in \R^n  \mid  \varepsilon \hat{x} \in \Omega\}
\end{equation}
and define the functions 
\begin{equation*}\label{eq:refunctions} 
\begin{split}
	\hat{u}^\varepsilon (\hat{x}, \hat{t}):= u^\varepsilon (\varepsilon \hat{x}, \varepsilon^2\hat{t}), 
	\qquad &\hat{x} \in \overline{\Omega^\varepsilon}, \ \hat{t} \in [0,\infty), \\
	\hat{g}^\varepsilon (\hat{x}) := g^\varepsilon (\varepsilon\hat{x}), \hspace{1.45cm}  &\hat{x} \in \Omega^\varepsilon.
\end{split}
\end{equation*}
Then $	\hat{g}^\varepsilon $ satisfies 
\begin{equation}\label{eq:reproperties}
	\Vert 	\hat{g}^\varepsilon \Vert_{L_\infty(\Omega^\varepsilon)} \leq c_1 \quad  \text{and}	\quad \Vert 	\nabla_{\hat{x}}\hat{g}^\varepsilon \Vert_{L_\infty(\Omega^\varepsilon)} \leq c_2,
\end{equation}
where we have used that $\varepsilon^{-\frac{1}{2}} < \varepsilon^{-1}$. 
For the solution $u^\varepsilon$ to \eqref{prob:AC}, we have 
\begin{equation}\label{eq:reAC}
	\hat{u}^\varepsilon_{\hat{t}} = \Delta_{\hat{x}}\hat{u}^\varepsilon - W'(\hat{u}^\varepsilon) + 
	\varepsilon  \hat{g}^\varepsilon \sqrt{2 W(\hat{u}^\varepsilon)}, \qquad (\hat{x}, \hat{t}) \in \Omega^\varepsilon \times (0,\infty)
\end{equation}
and $\nabla_{\hat{x}}\hat{u}^\varepsilon \cdot \nu^\varepsilon  =0$ for $(\hat{x},\hat{t}) \in \partial \Omega^\varepsilon \times (0,\infty)$, where $\nu^\varepsilon $ is the outward unit normal 
vector on $ \partial \Omega^\varepsilon$ and where $\nabla_{\hat{x}}$  and $\Delta_{\hat{x}}$ denote the gradient and Laplacian with respect to $\hat{x}$.  Furtheremore, we note that 
\begin{equation} \label{eq:recurvature}
		\kappa^\varepsilon := \Vert \text{principal curvature of } \partial \Omega^\varepsilon \Vert_{L_\infty(\partial \Omega^\varepsilon )} = \varepsilon \kappa.
\end{equation}

\begin{lemma}\label{lem:rebound}
For the solution $u^\varepsilon$ to \eqref{prob:AC}, there exists a constant $c_{10}>0$ depending only on $c_1$,  $c_2$, $c_3$, $c_4$, and $W$  such that 
\begin{equation}
	\label{eq:gradientestimate}
	\sup_{\Omega \times[0,\infty)} \varepsilon \vert \nabla u^\varepsilon\vert \leq c_{10}
\end{equation} 
for any $0<\varepsilon <1$. 
\end{lemma}

\begin{proof}
For the rescaled solution to \eqref{eq:reAC}, we obtain by Proposition \ref{prop:initial} (1) and 
Proposition  \ref{prop:cp} that 
\begin{equation}\label{eq:rebounds}
	\Vert \hat{u}^\varepsilon \Vert_{L_\infty(\Omega^\varepsilon \times[0,\infty))} \leq 1 \quad \text{and} \quad \sup_{\hat{x}\in \Omega^\varepsilon} |\nabla_{\hat{x}} \hat{u}^\varepsilon (\hat{x},0) |\leq c_3.
\end{equation}
Using this together with \eqref{eq:reproperties} and \eqref{eq:recurvature} it follows from a standard gradient estimate (see \cite[Theorem V.7.2]{ladyzenskaja} or \cite[Theorem 4.30]{lieberman}) that $\sup_{\Omega^\varepsilon \times[0,\infty)} \vert \nabla_{\hat{x}} \hat{u}^\varepsilon \vert$ is uniformly bounded for $0<\varepsilon<1$. Hence we obtain \eqref{eq:gradientestimate}. 
\end{proof}

\begin{proposition}\label{prop:bounddiscrepancy}
	Let $u^\varepsilon$ be a solution to \eqref{prob:AC}. There exists a constant $c_{11}>0$ 
	depending only on $n,\kappa, c_1, c_2, c_3, c_4, W$ and $\Omega$ such that 
	\begin{equation}\label{eq:bounddiscrepancy}
		\sup_{\Omega\times[0,\infty)} \xi^\varepsilon \leq c_{11} \varepsilon^{-\frac{1}{2}}
	\end{equation} 
for any $\varepsilon \in (0,1)$. 
\end{proposition}

\begin{proof}
	The proof follows closely the arguments developed in \cite[Proposition 5.1]{kagaya}, but we include it here for the sake of completeness. Let $\Omega^\varepsilon$ and $\hat{u}^\varepsilon$ be defined as above. Let us introduce the function 
	\begin{equation}\label{eq:auxfunction1}
		\hat{\xi}^\varepsilon (\hat{x},\hat{t}) := 
		\frac{|\nabla_{\hat{x}} \hat{u}^\varepsilon(\hat{x},\hat{t})|^2}{2} - W(\hat{u}^\varepsilon(\hat{x},\hat{t})) - 
		G(\hat{u}^\varepsilon(\hat{x},\hat{t})) + \varepsilon \phi(\hat{x}), 
		\qquad (\hat{x},\hat{t})\in \overline{\Omega^\varepsilon}\times[0,\infty), 
	\end{equation}
where $G\in C^\infty(\R)$ and $\phi \in C^\infty( \overline{\Omega^\varepsilon})$ will be chosen later.  We compute $\partial_{\hat{t}} \hat{\xi}^\varepsilon - \Delta_{\hat{x}} \hat{\xi}^\varepsilon$ and obtain 
\begin{equation}\label{eq:auxfunction2}
	\begin{split}
\partial_{\hat{t}} \hat{\xi}^\varepsilon - \Delta_{\hat{x}} \hat{\xi}^\varepsilon &= 
\nabla_{\hat{x}} \hat{u}^\varepsilon \cdot \nabla_{\hat{x}}\partial_{\hat{t}}  \hat{u}^\varepsilon 
- (W'+G') \partial_{\hat{t}}  \hat{u}^\varepsilon  - | \nabla^2_{\hat{x}} \hat{u}^\varepsilon|^2  
-\nabla_{\hat{x}} \hat{u}^\varepsilon \cdot  \nabla_{\hat{x}}(\Delta_{\hat{x}} \hat{u}^\varepsilon)
\\
&\quad + (W'+G')\Delta_{\hat{x}} \hat{u}^\varepsilon + (W''+G'') | \nabla_{\hat{x}} \hat{u}^\varepsilon|^2  - \varepsilon \Delta_{\hat{x}} \phi
\end{split}
\end{equation}
for $\hat{x} \in \overline{\Omega^\varepsilon}$ and $\hat{t} \in (0,\infty)$. Here we wrote and will write $W'=W'(\hat{u}^\varepsilon)$, $G'=G'(\hat{u}^\varepsilon)$ and so forth for simplicity's sake. Inserting \eqref{eq:reAC} into \eqref{eq:auxfunction2} gives 
\begin{equation}\label{eq:auxfunction3}
	\begin{split}
		\partial_{\hat{t}} \hat{\xi}^\varepsilon - \Delta_{\hat{x}} \hat{\xi}^\varepsilon &= 
		W'(W'+G') - |\nabla^2_{\hat{x}} \hat{u}^\varepsilon|^2 + G''  |\nabla_{\hat{x}} \hat{u}^\varepsilon|^2 - \varepsilon \Delta_{\hat{x}} \phi
			\\
		&\quad + \varepsilon \nabla_{\hat{x}} \hat{u}^\varepsilon \cdot  
		\nabla_{\hat{x}} \big( \hat{g}^\varepsilon \sqrt{2W} \big) - \varepsilon \left(W'+G'\right)  \hat{g}^\varepsilon \sqrt{2W} . 
	\end{split}
\end{equation}
Differentiating \eqref{eq:auxfunction1} with respect to $\hat{x}_j$ and by using the Cauchy-Schwarz inequality we have 
\begin{equation}\label{eq:auxfunction4}
	\begin{split}
	 |\nabla_{\hat{x}} \hat{u}^\varepsilon|^2 \,  |\nabla^2_{\hat{x}} \hat{u}^\varepsilon|^2
	 &\geq \sum_{j=1}^n \left( \sum_{i=1}^n \partial_{\hat{x}_i}  \hat{u}^\varepsilon \,  \partial_{\hat{x}_i\hat{x}_j}  \hat{u}^\varepsilon\right)^2
	 = \sum_{j=1}^n \left(  \partial_{\hat{x}_j} \hat{\xi}^\varepsilon + \left(W'+G'\right)  \partial_{\hat{x}_j} \hat{u}^\varepsilon -  \varepsilon  \partial_{\hat{x}_j} \phi  \right)^2
	 	\\
	 &\geq 2 \left( \left(W'+G'\right)  \nabla_{\hat{x}} \hat{u}^\varepsilon -  \varepsilon  \nabla_{\hat{x}} \phi  \right) \cdot  \nabla_{\hat{x}} \hat{\xi}^\varepsilon 
	 +  (W'+G')^2 \,  |\nabla_{\hat{x}} \hat{u}^\varepsilon|^2 
	 	\\
	 &\quad 
	 - 2 \varepsilon \left(W'+G'\right) 
	 \nabla_{\hat{x}} \hat{u}^\varepsilon \cdot \nabla_{\hat{x}}\phi. 
	\end{split}
\end{equation}
On $\{ \vert \nabla_{\hat{x}} \hat{u}^\varepsilon \vert >0 \}$, we divide \eqref{eq:auxfunction4} 
by $|\nabla_{\hat{x}} \hat{u}^\varepsilon |^2$ and we substitute it into \eqref{eq:auxfunction3} to obtain 
\begin{equation}\label{eq:auxfunction5}
	\begin{split}
		\partial_{\hat{t}} \hat{\xi}^\varepsilon - \Delta_{\hat{x}} \hat{\xi}^\varepsilon 
		&\leq - (G')^2 -W'G' - \frac{2 \left( \left(W'+G'\right)  \nabla_{\hat{x}} \hat{u}^\varepsilon -  \varepsilon  \nabla_{\hat{x}} \phi  \right) \cdot  \nabla_{\hat{x}} \hat{\xi}^\varepsilon }{|\nabla_{\hat{x}} \hat{u}^\varepsilon |^2}
			\\
		&\quad + \frac{2 \varepsilon \left( W' + G'\right) }{|\nabla_{\hat{x}} \hat{u}^\varepsilon |^2} 
		 \nabla_{\hat{x}} \hat{u}^\varepsilon \cdot  \nabla_{\hat{x}} \phi  + G''  |\nabla_{\hat{x}} \hat{u}^\varepsilon|^2 - \varepsilon \Delta_{\hat{x}} \phi 
		 \\
		 &\quad + \varepsilon \nabla_{\hat{x}} \hat{u}^\varepsilon \cdot  
		 \nabla_{\hat{x}} \big( \hat{g}^\varepsilon \sqrt{2W} \big) - \varepsilon \left(W'+G'\right)  \hat{g}^\varepsilon \sqrt{2W} . 
	\end{split}
\end{equation}
Choosing 
\begin{equation*}
	G(s):= \varepsilon^{\frac{1}{2}} \left( 1-\frac{1}{8} (s-\gamma)^2\right), \qquad s \in \R, 
\end{equation*}
with $\gamma$ as in assumption (W2), we have 
\begin{equation}\label{eq:auxfunction6}
	0<G\leq  \varepsilon^{\frac{1}{2}}, \quad G'W' \geq 0, \quad G''= - \frac{ \varepsilon^{\frac{1}{2}}}{4}
\end{equation}
for $s\in (-1,1)$.  Let $\phi$ be defined by $\phi(\hat{x}):= \kappa \left(c_{10}^2+1 \right) \psi\left( 
\dist(\partial \Omega^\varepsilon, \hat{x}) \right) $, where $\psi \in C^\infty ([0,\infty) ; \R^+)$ 
is given such that 
\begin{equation*}
	\psi(s)=s  \quad \text{for } \: s\in \left[0,\frac{c_4}{2}\right] , \qquad \psi'(s)=0   \quad \text{for } \: s\in [c_4, \infty), \qquad |\psi'| \leq 1, \qquad |\psi''| \leq \frac{4}{c_4}. 
\end{equation*}
By using the arguments in the proof of \cite[Proposition 5.1]{kagaya}, we find that $\phi$ is smooth  and satisfies 
\begin{equation}\label{eq:auxfunction7}
0 \leq\phi\leq  M_1, \qquad |\nabla_{\hat{x}} \phi | \leq M_1, \qquad |\Delta_{\hat{x}} \phi | \leq M_1 \quad \text{in } \: \overline{\Omega^\varepsilon}
\end{equation}
and $	\nabla_{\hat{x}} \phi \cdot \nu^\varepsilon = -\kappa (c_{10}^2+1  )$ on $\partial{\Omega^\varepsilon}$, 
where $M_1>0$ is a constant depending only on $n, \kappa, c_4, c_{10}$. Substituting inequalities 
\eqref{eq:auxfunction6} and \eqref{eq:auxfunction7} into \eqref{eq:auxfunction5}, we have 
\begin{equation}\label{eq:auxfunction8} 
	\begin{split}
		\partial_{\hat{t}} \hat{\xi}^\varepsilon - \Delta_{\hat{x}} \hat{\xi}^\varepsilon 
		\leq  &- \frac{2 \left( \left(W'+G'\right)  \nabla_{\hat{x}} \hat{u}^\varepsilon -  \varepsilon  \nabla_{\hat{x}} \phi  \right) \cdot  \nabla_{\hat{x}} \hat{\xi}^\varepsilon }{|\nabla_{\hat{x}} \hat{u}^\varepsilon |^2} + \frac{M_2  \varepsilon}{|\nabla_{\hat{x}} \hat{u}^\varepsilon |}		
		- \frac{\varepsilon^\frac{1}{2}}{4} |\nabla_{\hat{x}} \hat{u}^\varepsilon |^2 + M_1  \varepsilon 
			\\
		& + \varepsilon |\nabla_{\hat{x}} \hat{u}^\varepsilon | |\nabla_{\hat{x}} \big( \hat{g}^\varepsilon \sqrt{2W} \big)| + \varepsilon \sqrt{2W} \big( |W'| + |G' |\big) |\hat{g}^\varepsilon |
     \end{split} 
\end{equation}
for any point $\hat{x}$ such that $|\nabla_{\hat{x}} \hat{u}^\varepsilon (\hat{x}) |>0$, where $M_2>0$ is a constant depending only on  $M_1$ and $\sup_{|s|\leq 1} |W'(s)|$. The last two terms in the above inequality can be estimated as follows. For the first term we have due to Lemma \ref{lem:rebound}, \eqref{eq:CMV} and \eqref{eq:reproperties}, 
\begin{equation}\label{eq:auxfunction9} 
	\begin{split}
		\varepsilon |\nabla_{\hat{x}} \hat{u}^\varepsilon | |\nabla_{\hat{x}} \big( \hat{g}^\varepsilon \sqrt{2W} \big)| 
		& \leq  \varepsilon c_{10} \left(  \sqrt{2W} | \nabla_{\hat{x}}  \hat{g}^\varepsilon| + |\hat{g}^\varepsilon | \frac{|W'|}{ \sqrt{2W}} |\nabla_{\hat{x}} \hat{u}^\varepsilon |\right) 
		\\
		& \leq \varepsilon c_{10} \left( c_2 \sup_{|s|\leq 1} \sqrt{2W(s)} + c_1c_{10} \sup_{|s|\leq 1} |W''(s)|^\frac12\right).
	\end{split}
\end{equation}
As for the second term, we have by \eqref{eq:reproperties} and  \eqref{eq:auxfunction6}
\begin{equation}\label{eq:auxfunction10} 
	\varepsilon \sqrt{2W} \big( |W'| + |G' |\big) |\hat{g}^\varepsilon | 
	\leq \varepsilon \sup_{|s|\leq 1} \sqrt{2W(s)} \left( \sup_{|s|\leq 1} |W'(s)| +1\right)c_1. 
\end{equation}
Inserting the estimates \eqref{eq:auxfunction9}  and \eqref{eq:auxfunction10} into  \eqref{eq:auxfunction8}  we find 
\begin{equation}\label{eq:auxfunction11} 
	\partial_{\hat{t}} \hat{\xi}^\varepsilon - \Delta_{\hat{x}} \hat{\xi}^\varepsilon 
	\leq  - \frac{2 \left( \left(W'+G'\right)  \nabla_{\hat{x}} \hat{u}^\varepsilon -  \varepsilon  \nabla_{\hat{x}} \phi  \right) \cdot  \nabla_{\hat{x}} \hat{\xi}^\varepsilon }{|\nabla_{\hat{x}} \hat{u}^\varepsilon |^2} 
	+ \frac{M_2  \varepsilon}{|\nabla_{\hat{x}} \hat{u}^\varepsilon |}		
	- \frac{\varepsilon^\frac{1}{2}}{4} |\nabla_{\hat{x}} \hat{u}^\varepsilon |^2 
	+ (M_1+M_3)  \varepsilon, 
\end{equation}
where $M_3>0$ is a constant depending only on $c_1,c_2,c_{10},c, \sup_{|s|\leq 1} \sqrt{2W(s)}$, $\sup_{|s|\leq 1} |W'(s)|$, and $\sup_{|s|\leq 1} |W''(s)|^\frac12$. To derive a contradiction, we fix an arbitrarily large $\hat{T}>0$ and suppose that 
\begin{equation}\label{eq:recontra}
	\max_{\hat{x} \in \overline{\Omega^\varepsilon}, \ \hat{t}\in [0,\hat{T}]} \hat{\xi}^\varepsilon 
	(\hat{x},\hat{t}) \geq C \varepsilon^{\frac{1}{2}}
\end{equation}
for $0<\varepsilon<1$ and a positive constant $C$ to be chosen. By arguing exactly as in the proof of \cite[Proposition 5.1]{kagaya}, we can prove that there exists a maximum point 
$(\hat{x}_0,\hat{t}_0)\in \Omega^\varepsilon \times(0,\hat{T}]$ of $\hat{\xi}^\varepsilon $ 
where 
\begin{equation}\label{eq:remax}
	\nabla_{\hat{x}} \hat{\xi}^\varepsilon (\hat{x}_0,\hat{t}_0) =0, \quad 
		\Delta_{\hat{x}} \hat{\xi}^\varepsilon (\hat{x}_0,\hat{t}_0) \leq 0, \quad 
	\partial_{\hat{t}} \hat{\xi}^\varepsilon (\hat{x}_0,\hat{t}_0) \geq 0 \quad \text{and}\quad 
	\hat{\xi}^\varepsilon (\hat{x}_0,\hat{t}_0) \geq C \varepsilon^{\frac{1}{2}}
\end{equation}
hold. It follows from the definition of $\hat{\xi}^\varepsilon $, \eqref{eq:auxfunction7}, \eqref{eq:remax},  and the positivity of $W$ and $G$ that 
\begin{equation}\label{eq:remaxbound}
	|\nabla_{\hat{x}} \hat{u}^\varepsilon (\hat{x}_0,\hat{t}_0) |^2  \geq 2  \varepsilon^{\frac{1}{2}} 
	(C-M_1). 
\end{equation}
For $C$ large enough so that $2  \varepsilon^{\frac{1}{2}} (C-M_1)>0$ we must have 
$|\nabla_{\hat{x}} \hat{u}^\varepsilon|>0$ in the neighborhood of  $(\hat{x}_0,\hat{t}_0)$. Hence, 
by applying \eqref{eq:remax} and  \eqref{eq:remaxbound} to \eqref{eq:auxfunction11} we obtain
\begin{equation*}
	0 \leq \frac{M_2 \varepsilon^{\frac{3}{4}}}{\sqrt{2(C-M_1)}} - \frac{\varepsilon (C-M_1)}{2} + 
	(M_1+M_3)  \varepsilon. 
\end{equation*}
Now choosing $C$ sufficiently large depending only on $M_1$, $M_2$ and $M_3$, we get a 
contradiction. Thus, we have proved 
\begin{equation*}
	\max_{\hat{x} \in \overline{\Omega^\varepsilon}, \ \hat{t}\in [0,\hat{T}]} \hat{\xi}^\varepsilon 
	(\hat{x},\hat{t}) \leq C \varepsilon^{\frac{1}{2}}
\end{equation*}
for $0<\varepsilon<1$ and sufficiently large $C$ depending only on $n,\kappa, c_1,c_2,c_3,c_4,W$ and $\Omega$. Because $G\leq  \varepsilon^{\frac{1}{2}}$, $\phi$ is nonngetive and $\hat{T}$ is arbitrary, we infer, by choosing $c_{11}:=C+1$, the desired bound 
\eqref{eq:bounddiscrepancy}. 
\end{proof}


\subsection{Density ratio upper bound} 
\label{sub:density}
This subsection concerns the uniform density ratio upper bound independent of $\varepsilon$ of the Allen-Cahn equation \eqref{prob:AC}. We define 
\begin{equation}\label{eq:drub}
	D^\varepsilon (t) := \max\left\{ \sup_{y\in N_{\frac{c_4}{2}} \cap \overline{\Omega},\, 0<r<c_4} \dfrac{\mu_t^\varepsilon(B_r(y)) + \mu_t^\varepsilon(\tilde{B}_r(y))}{\omega_{n-1}r^{n-1}}, 
	\sup_{y\in \Omega \setminus N_{\frac{c_4}{2}},\, 0<r<c_4} \dfrac{\mu_t^\varepsilon(B_r(y))}{\omega_{n-1}r^{n-1}}
	\right\}
\end{equation}
for $t\in [0,\infty)$. For $t=0$ in \eqref{eq:drub}, we conveniently choose $c_4>R_2$, as indicated by \eqref{density} and Proposition \ref{prop:initial} (5). 

\begin{proposition}\label{prop:drub}
Let $u^\varepsilon$ be a solution to \eqref{prob:AC}. For any $T\in (0,\infty)$, there exist $c_{12}>0$ and $\epsilon_1\in (0,1)$ depending only on $T, n, D_0, E_1, \alpha, W, \kappa, c_1, c_2, c_3, c_4$ and $\Omega$ such that 
\begin{equation}\label{eq:drub 1}
		D^\varepsilon (t) \leq c_{12}
\end{equation}
for all $t\in[0,T]$ and $\varepsilon\in (0,\epsilon_1)$. 
\end{proposition}

The idea of proof of the above proposition is coming from \cite{kagaya}, where the author investigates the Allen-Cahn equation \eqref{prob:AC} with no obstacles ($g^\varepsilon \equiv 0$). The proof mainly relies on the monotonicity formula proven in Proposition \ref{prop:mf} and on the uniform boundedness of the discrepancy measure from above as shown in Proposition \ref{prop:bounddiscrepancy}
(these are almost identical to those obtained in \cite{kagaya}).
Hence its proof is done in the same way, however only the main points that have to be changed are illustrated below for reader's convenience.

\begin{proof}
The proof of the upper bound of the density of $\mu_t^\varepsilon$ follows by the same arguments as the proof of \cite[Proposition 6.1]{kagaya}, replacing $\lambda$ with $\tfrac{1}{2}$ and $\lambda'$ with  $\tfrac{3}{4}$. The only difference lies in the proof of \cite[Lemma 6.6]{kagaya}, which we need to modify to include the presence of $g^\varepsilon$. 

Firstly, by Proposition \ref{prop:initial} (5) and 
\cite[Lemma 6.3]{kagaya}, we have 
\begin{equation}\label{eq:drub 2}
	D^\varepsilon (0) \leq (1+5^{n-1})2 D_0. 
\end{equation}
From now on, until the proof of estimate \eqref{eq:drub 4} below is completed, we assume that 
\begin{equation}\label{eq: drub 3}
	\sup_{t\in [0,T_1]}  D^\varepsilon (t) \leq D_1
\end{equation}
holds for some constants $T_1>0$ and $D_1>0$. Here, $D_1 > D^\varepsilon (0)$ is a constant depending only on $T,n,D_0,\alpha, W, \kappa, E_1, c_1,c_4, \Omega$, and not on $\varepsilon$, and which will be determined later. We note that such $T_1>0$  exists because $D_1 > D^\varepsilon (0)$  and by the continuity of $D^\varepsilon (t)$ in time. Such continuity follows from that of $u^\varepsilon$ established in Theorem \ref{thm:regularity solution}. Although the dependence of $T_1$ on $\varepsilon$ is unclear 
in this discussion, it will be shown later that $T_1$ is in fact independent of $\varepsilon$. 

Let us now explain how to modify the proof of \cite[Lemma 6.6]{kagaya}. Since 
$g^\varepsilon=0$  for $y\in N_{c_4} \cap \overline{\Omega}$, the only place in the proof of \cite[Lemma 6.6]{kagaya} where $g^\varepsilon$ is involved is in the proof of estimate \cite[(6.26)]{kagaya}. Instead of \cite[(6.26)]{kagaya} we have the following estimate: 
\begin{align}\label{eq:drub 4}
	& \int_{(B_r(y) \cap \Omega)\setminus A_2} \dfrac{\varepsilon}{2\sigma} \, |\nabla u^\varepsilon (x,t) |^2 \, \dx 
	\nonumber \\[0.3cm]
	& \qquad  \leq 
	\begin{cases} 
		\dfrac{9D_1\omega_{n-1}(2r)^{n-1}}{\beta} \, \varepsilon^{\frac{1}{4}} \quad \text{ if }  y\in N_{\frac{c_4}{2}} \cap \overline{\Omega},	\\[0.4cm]
		\left( 	\dfrac{11 D_1\omega_{n-1}(2r)^{n-1}}{\beta} + \beta^{-1} c_2^2 c_9 \right)\varepsilon^{\frac{1}{4}} \quad \text{ if } y\in \Omega \setminus N_{\frac{c_4}{2}}
		\end{cases}
\end{align}
for sufficiently small $\varepsilon$ depending only on $\beta$, $c_1$ and $W$.  Indeed, 
let  $y \in \overline{\Omega}$, $r\in (\varepsilon^{\frac{3}{4}}, \frac{c_4}{2})$ and $t \in [2\varepsilon^{\frac{3}{2}}, \infty) \cap [0,T_1]$ with $T_1\geq 2\varepsilon^{\frac{3}{2}}$. 
We recall from \cite[p.1503]{kagaya} the definition
\begin{equation*}
	A_2 := \left\{ x \in B_{ 10r + 2c' \varepsilon^{\frac{3}{4}}} (y) \cap \Omega \mid \dist (A_1, x)< 2 c_{10}^K\, \varepsilon^{\frac{3}{4}} \right\}, 
\end{equation*}
where $ A_1:= \big\{x \in B_{10r}(y) \cap \Omega \mid \text{ for some } \tilde{t} \text{ with } t - \varepsilon^{\frac{3}{2}} \leq \tilde{t} \leq t, \, |u^\varepsilon(x,\tilde{t})| \leq \alpha \big\}$ and 
 $c_{10}^K>1$ is the constant from \cite[Lemma 6.4]{kagaya} ($c_{10}^K$ denotes the constant $c_{10}$ from \cite{kagaya}).  We define $\phi \in \text{Lip}(B_{2r}(y))$ such that 
 \begin{equation*}
 	\phi(x) = 
 	\begin{cases}
 		1   \quad \text{ if } x \in B_r(y) \setminus A_2,\\[0.2cm]
 		0  \quad \text{ if }  \dist (B_r(y) \setminus A_2) \geq \varepsilon^{\frac{3}{4}},
 	\end{cases}
 \quad |\nabla \phi| \leq 2 \varepsilon^{-\frac{3}{4}}, \quad 0\leq \phi\leq 1. 
 \end{equation*}
First, we consider the case that $y \in N_{\frac{c_4}{2}} \cap \overline{\Omega}$. Then, since 
$g^\varepsilon=0$ on $(B_r(y) \cap \Omega)\setminus A_2$, the proof of \eqref{eq:drub 4} is the same as the one of \cite[(6.26)]{kagaya}. Next, we consider the case that $y \in \Omega \setminus N_{\frac{c_4}{2}}$. By $r > \varepsilon^{\frac{3}{4}}$, $2c_{10}^K\,\varepsilon^{\frac{3}{4}} > \varepsilon^{\frac{3}{4}}$ and the definitions of $A_1$ and $\phi$, we have $\text{spt} \phi \cap A_1 
= \emptyset$, thus 
\begin{equation}\label{eq:drub 5}
	| u^\varepsilon (x,s) | \geq \alpha \qquad \text{for } x \in \text{spt} \phi \cap \Omega, \, s \in 
	[t-\varepsilon^{\frac{3}{2}}, t].
\end{equation}
For each $j$ we differentiate the Allen-Cahn equation \eqref{prob:AC} with respect to $x_j$, multiply with $\phi^2 \partial_{x_j} u^\varepsilon$, sum over $j$ and integrate over $\Omega$ to obtain 
\begin{equation}	\label{eq:drub 6}
	\begin{split} 
	\frac{\d}{\dt} \int_{\Omega} \frac{1}{2} \, | \nabla u^\varepsilon|^2 \, \phi^2 \, \dx 
	&= \int_{\Omega} \left( \nabla u^\varepsilon \cdot \Delta \nabla u^\varepsilon - 
	\frac{W''(u^\varepsilon)}{\varepsilon^2} \,  | \nabla u^\varepsilon|^2\right) \phi^2 \, \dx
	\\
	&\quad + \int_{\Omega} \left( \frac{\sqrt{2W(u^\varepsilon)}}{\varepsilon} \, \nabla g^\varepsilon \cdot  \nabla u^\varepsilon + g^\varepsilon  \frac{W'(u^\varepsilon)}{\varepsilon\sqrt{2W(u^\varepsilon)}} \,  | \nabla u^\varepsilon|^2 \right) \phi^2 \, \dx . 
\end{split}
\end{equation}
By integration by parts, the Cauchy-Schwarz inequality and the Neumann boundary condition in  \eqref{prob:AC}, \eqref{eq:drub 6} gives 
\begin{equation}	\label{eq:drub 7}
	\begin{split} 
	\frac{\d}{\dt} \int_{\Omega} \frac{1}{2} \, | \nabla u^\varepsilon|^2 \, \phi^2 \, \dx 
	& \leq  \int_{\Omega} | \nabla \phi|^2  \, | \nabla u^\varepsilon|^2 \, \dx 
	- \int_{\Omega}  \frac{W''(u^\varepsilon)}{\varepsilon^2}  \, | \nabla u^\varepsilon|^2 \, \phi^2 \, \dx 
	\\
	& \quad + \int_{\Omega} | \nabla g^\varepsilon| \, | \nabla u^\varepsilon| \, \frac{\sqrt{2W(u^\varepsilon)}}{\varepsilon} \, \phi^2 \, \dx  
	+  \int_{\Omega} | g^\varepsilon| \, \frac{|W'(u^\varepsilon)|}{\varepsilon\sqrt{2W(u^\varepsilon)}} \,  | \nabla u^\varepsilon|^2 \, \phi^2 \, \dx . 
\end{split}
\end{equation}
By \eqref{eq:drub 5} and the assumption (W3), we have $W''(u^\varepsilon) \geq \beta$ on $\text{spt} \phi \cap \Omega$  for $s \in 	[t-\varepsilon^{\frac{3}{2}}, t]$.
Using this together with Young's inequality and the definitions of $g^\varepsilon$ and $\phi$, we have by \eqref{eq:CMV} and \eqref{eq:drub 7} 
\begin{align}	\label{eq:drub 8}
	\nonumber 
	&\frac{\d}{\dt} \int_{\Omega} \frac{1}{2} \, | \nabla u^\varepsilon|^2 \, \phi^2 \, \dx 
	\\
	&\quad \leq 4 \varepsilon^{-\frac{3}{2}} \int_{\text{spt} \phi \cap \Omega}  | \nabla u^\varepsilon|^2 \, \dx 
	 - \frac{\beta}{\varepsilon^2}  \int_{\Omega} | \nabla u^\varepsilon|^2 \, \phi^2 \, \dx 
	 \\
	 \nonumber 
	 & \qquad + \int_{\Omega} \frac{1}{2} \varepsilon^{-\frac{3}{2}} \, | \nabla u^\varepsilon|^2 \, \phi^4
	 + \frac{1}{2} \varepsilon^{\frac{3}{2}} |\nabla g^\varepsilon | ^2 \frac{2W (u^\varepsilon)}{\varepsilon^2} \, 
	  \, \dx
      + \frac{c_1}{\varepsilon} \,  \sup_{|v| \leq 1}  |W''(v)|^{\frac12} \,  \int_{\Omega} | \nabla u^\varepsilon|^2 \, \phi^2 \, \dx
	 \\
	 \nonumber 
	 & \quad \leq  5 \varepsilon^{-\frac{3}{2}} \int_{\text{spt} \phi \cap \Omega}  | \nabla u^\varepsilon|^2 \, \dx  
	 - \frac{\beta}{\varepsilon^2}  \int_{\Omega} \frac{1}{2}| \nabla u^\varepsilon|^2 \, \phi^2 \, \dx  
	 + c_2 ^2 \varepsilon ^{-\frac{1}{2}} \sigma \mu _t ^\varepsilon (\Omega) 
\end{align}
for sufficiently small $\varepsilon$ depending only on $\beta$, $c_1$ and $W$.  Integrating \eqref{eq:drub 8} over $[t-\varepsilon^{\frac{3}{2}}, t]$ and using Corollary \ref{cor:boundmeas} yields 
\begin{equation}	\label{eq:drub 9}
	\begin{split} 
&\int_{\Omega} \frac{1}{2} \, | \nabla u^\varepsilon|^2 \, \phi^2(x,t) \, \dx  
\\[0.2cm]
&\quad \leq e^{-\beta \varepsilon^{-\frac{1}{2}}} \int_{\Omega} \frac{1}{2} \,  | \nabla u^\varepsilon|^2 \, \phi^2(x,t-\varepsilon^{\frac{3}{2}}) \, \dx  + 
\int_{t-\varepsilon^{\frac{3}{2}}}^t  e^{-\frac{\beta}{ \varepsilon^2} (t-s)} \, 5 \varepsilon^{-\frac{3}{2}} \left( \int_{\text{spt} \phi \cap \Omega}  | \nabla u^\varepsilon|^2 \, \dx   \right) \ds 
\\
&\qquad + \beta^{-1} \sigma c_2^2 c_9 \varepsilon ^{\frac32}. 
\end{split}
\end{equation}
By $\text{spt} \phi  \subset B_r(y)$, $r<\frac{c_4}{2}$ and the assumption \eqref{eq: drub 3} (see \cite[(6.4)]{kagaya}) we infer 
\begin{equation}\label{eq:drub 10}
	\sup_{s \in [t-\varepsilon^{\frac{3}{2}}, t]}\int_{\Omega} \frac{\varepsilon}{2\sigma} \, 
	| \nabla u^\varepsilon|^2(x,s) \, dx \leq D_1 \omega_{n-1} (2r)^{n-1}. 
\end{equation}
Combining \eqref{eq:drub 9}, \eqref{eq:drub 10} and the fact that $(B_r(y) \cap \Omega)\setminus A_2 \subset \{ \phi=1\}$, we get 
\begin{equation*}
\begin{split}
 \int_{(B_r(y) \cap \Omega)\setminus A_2} \frac{\varepsilon}{2\sigma} \, |\nabla u^\varepsilon (x,t) |^2 \, \dx  
& \leq  \int_{\Omega} \frac{\varepsilon}{2\sigma} \, |\nabla u^\varepsilon|^2 \, \phi^2(x,t) \, \dx 
 \\
& \leq   D_1 \omega_{n-1} (2r)^{n-1} \left( e^{-\beta \varepsilon^{-\frac{1}{2}}}  + \frac{10}{\beta} \, \varepsilon^{\frac{1}{2}} \right)  + \beta^{-1} c_2^2 c_9 \varepsilon ^{\frac52}
 \\
& \leq 	\left( 	\dfrac{11 D_1\omega_{n-1}(2r)^{n-1}}{\beta} + \beta^{-1} c_2^2 c_9 \right)\varepsilon^{\frac{1}{4}}. 
\end{split}
\end{equation*}
for sufficiently small $\varepsilon$ depending only on $\beta$, $c_1$ and $W$. This completes the proof of \eqref{eq:drub 4}. 

Let us now prove the statement of the proposition. Note that we do not assume \eqref{eq: drub 3} anymore. For $T>0$, we take $c_{12}$ as
\begin{equation*}
	c_{12}:= \max \left\{ \frac{(4\pi)^{\tfrac{n-1}{2}}  \, e^{c_5\big(T+\tfrac{c_4^2}{16}\big)^{\tfrac{1}{4}}} \left( (1+5^{n-1}) \, 2D_0 + c_6 T+1 \right)}{\omega_{n-1}\, e^{-\tfrac{1}{4}}}, \frac{4^{n-1} \, 2c_9}{c_4^{n-1}\omega_{n-1}}	\right\},
\end{equation*}
which is independent of $D_1$ of \eqref{eq: drub 3} and $\varepsilon$, and we set $D_1:= c_{12}+1$. For this $c_{12}$, we suppose the conclusion \eqref{eq:drub 1} to be false. Then, by the continuity of $D^\varepsilon(t)$, there exist $y\in \overline{\Omega}$, $\tilde t \in (0,T]$, $0<r<c_4$ and sufficiently small $\varepsilon$ such that 
\begin{equation}\label{eq:drub 11}
	\frac{\mu_{\tilde t}^\varepsilon (B_r(y)) + \mu_{\tilde t}^\varepsilon (\tilde B_r(y))}{\omega_{n-1} r^{n-1}} >c_{12} \quad \text{if } y \in N_{\tfrac{c_4}{2}}\cap \overline{\Omega} \quad \text{or} \quad	\frac{\mu_{\tilde t}^\varepsilon (B_r(y))}{\omega_{n-1} r^{n-1}} >c_{12} \quad \text{if } y \in \Omega\setminus N_{\tfrac{c_4}{2}},
	\end{equation}
and $\sup_{t\in [0,\tilde t]} D^\varepsilon(t) \leq D_1$. We consider first the case 
$y \in N_{\tfrac{c_4}{2}}\cap \overline{\Omega}$. For $r' \geq \tfrac{c_4}{4}$, it follows  from Proposition \ref{prop:energyestimate} and the choice of $c_{12}$ that 
\begin{equation}\label{eq:drub 12}
		\frac{\mu_{\tilde t}^\varepsilon (B_{r'}(y)) + \mu_{\tilde t}^\varepsilon (\tilde B_{r'}(y))}{\omega_{n-1} {r'}^{n-1}} \leq \frac{4^{n-1} \, 2c_9}{c_4^{n-1}\omega_{n-1}} \leq c_{12}.
\end{equation}
By \eqref{eq:drub 11} and \eqref{eq:drub 12} we infer that $0<r<\tfrac{c_4}{4}$.  Now, by integrating the monotonicity formula \eqref{eq:mformula_b} over $t\in(0,\tilde t)$ with $s =\tilde t +r^2$ and applying the inequality 
\begin{equation*}
	\begin{split}
	\frac{1}{\sigma}	\int_0^{\tilde t}  \int_\Omega  \frac{\rho_{1,(y,s)}(x,t) + \rho_{2,(y,s)}(x,t)}{2(s-t)} \left( \frac{\varepsilon |\nabla u|^2}{2} - \frac{W(u^\varepsilon)}{\varepsilon}\right)^+ \, \d x\d t 
	 \leq c_{13}^K \, \varepsilon^{\tfrac{1}{4}} \left(  1+ |\log \varepsilon| + (\log s)^+ \right)
	\end{split}
\end{equation*}
which is obtained by the same proof as in \cite[Lemma 6.7]{kagaya}, we get by $\tilde t\leq T$ and $s \leq T+\frac{c_4^2}{16}$ that 
\begin{equation}\label{eq:drub 13}
	\begin{split}
		&e^{c_5(s-t)^{\frac{1}{4}}} \int_{\Omega} \left( \rho_{1,(y,s)}(x,t) + \rho_{2,(y,s)}(x,t) \right) \d \mu_t^\varepsilon(x) \bigg |_{t=0}^{\tilde t}  \\
		& \quad \leq \int_{0}^{\tilde t} e^{c_5(s-t)^{\frac{1}{4}}} 
		\left( c_6 + \int_{\Omega}  \frac{\rho_{1,(y,s)}(x,t) + \rho_{2,(y,s)}(x,t)}{2(s-t)}  \d \xi_t^\varepsilon(x) \right) \dt \\
		& \quad \leq e^{c_5\big(T+\frac{c_4^2}{16}\big)^{\frac{1}{4}}} \left\{ c_6T+   c_{13}^K \, \varepsilon^{\tfrac{1}{4}} \left( 1+ |\log \varepsilon| + \left(\log \Big( T+\tfrac{c_4^2}{16} \Big) \right)^+ \right)
		\right\}. 
	\end{split}
\end{equation}
Here, $c_{13}^K$ is constant $c_{13}$ from \cite{kagaya}. 
By virtue of $s\leq T + \tfrac{c_4^2}{16}$ and \eqref{eq:drub 2} we infer that 
\begin{equation}\label{eq:drub 14}
	\begin{split}
&e^{c_5 s^{\frac{1}{4}}} \int_{\Omega} \left( \rho_{1,(y,s)}(x,0) + \rho_{2,(y,s)}(x,0) \right) \d \mu_0^\varepsilon(x) \\
&\quad \leq \frac{e^{c_5\big(T+\frac{c_4^2}{16}\big)^{\frac{1}{4}}}}{(4\pi s)^{\frac{n-1}{2}}} \left( \int_{\Omega \cap B_{\tfrac{c_4}{2}}(y)} 
e^{-\frac{|x-y|^2}{4s}} \, \d  \mu_0^\varepsilon(x) + \int_{\Omega \cap \tilde B_{\tfrac{c_4}{2}}(y)} 
e^{-\tfrac{|\tilde x-y|^2}{4s}} \, \d  \mu_0^\varepsilon(x)
\right)\\
& \quad \leq \frac{e^{c_5\big(T+\frac{c_4^2}{16}\big)^{\frac{1}{4}}}}{(4\pi s)^{\frac{n-1}{2}}} \int_0^{e^{-\frac{c_4^2}{16 s}}} 
\mu_0^\varepsilon\left( \left\{ x\in \big(\Omega \cap B_{\tfrac{c_4}{2}}(y)\big) \mid  e^{-\frac{| x-y|^2}{4s}}  \geq l   \right\} \right) \\
& \hspace{6cm}+ \mu_0^\varepsilon\left( \left\{ x\in (\Omega \cap \tilde B_{\tfrac{c_4}{2}}(y)) \mid  e^{-\frac{| \tilde x-y|^2}{4s}}  \geq l   \right\} \right) \d l \\
& \quad \leq e^{c_5\big(T+\frac{c_4^2}{16}\big)^{\frac{1}{4}}} (1+5^{n-1}) 2D_0 \, 
\frac{\omega_{n-1}}{\pi^{\frac{n-1}{2}}}  \, \int_0^1 (-\log l)^{\frac{n-1}{2}} \, \d l \\
& \quad  = e^{c_5\big(T+\frac{c_4^2}{16}\big)^{\frac{1}{4}}}   (1+5^{n-1}) 2D_0, 
	\end{split}
\end{equation}
where we used $\int_0^1 (-\log l)^{\frac{n-1}{2}} \, \d l  = \Gamma \big( \frac{n+1}{2} \big) = \frac{\pi^{\frac{n-1}{2}}}{\omega_{n-1}}$.  By $s=\tilde t +r^2$, $r<\tfrac{c_4}{4}$, $\eta =1$ on $B_{\frac{c_4}{4}}$ and \eqref{eq:drub 11}, we obtain
\begin{equation}\label{eq:drub 15}
	\begin{split}
		&e^{c_5 (s- \tilde t)^{\frac{1}{4}}} \int_{\Omega} \left( \rho_{1,(y,s)}(x,\tilde t) + \rho_{2,(y,s)}(x,\tilde t) \right) \d \mu_{\tilde t}^\varepsilon(x) \\
		&\quad \geq \frac{1}{(4\pi r^2)^{\frac{n-1}{2}}} \left( \int_{\Omega \cap B_{r}(y)} 
		e^{-\frac{|x-y|^2}{4r^2}} \, \d  \mu_{\tilde t}^\varepsilon(x) + \int_{\Omega \cap \tilde B_{r}(y)} 
		e^{-\tfrac{|\tilde x-y|^2}{4r^2}} \, \d  \mu_{\tilde t}^\varepsilon(x)
		\right)\\
		& \quad \geq \frac{e^{-\frac{1}{4}}}{(4\pi)^{\frac{n-1}{2}}r^{n-1}} \left( 
		\mu_{\tilde t}^\varepsilon(B_r(y)) + \mu_{\tilde t}^\varepsilon(\tilde B_r(y)) 
		\right) \\
		& \quad > \frac{e^{-\frac{1}{4}}}{(4\pi)^{\frac{n-1}{2}}} \omega_{n-1} \, c_{12} .
	\end{split}
\end{equation}
We choose now $0< \epsilon_1 \leq \epsilon_3$ so that 
\begin{equation*}
c_{13}^K \, \varepsilon^{\frac{1}{4}} \left\{ 1+ |\log \varepsilon| + \left(\log \Big( T+\tfrac{c_4^2}{16} \Big) \right)^+  
\right\} \leq 1
\end{equation*}
for $\varepsilon \in (0, \epsilon_1)$. 
Then, by combining  \eqref{eq:drub 13}-\eqref{eq:drub 15} and the choice of $c_{12}$, we deduce a contradiction for 
$\varepsilon \in (0, \epsilon_1)$.  In the case $y \in \Omega\setminus N_{\tfrac{c_4}{2}}$, we derive a contradiction by similar arguments. 
\end{proof}
 


\section{Vanishing of the discrepancy}
\label{sec:vanishing}
The aim of this section is to prove that 
\begin{equation}
\label{eq:vd}
\lim_{i\rightarrow\infty} \int_0^T \int_{\overline{\Omega}}  \frac{1}{\sigma} 
 \left|	\frac{\varepsilon_i | \nabla 	u^{\varepsilon_i}(x,t)|^2 }{2} - \frac{W(u^{\varepsilon_i }(x,t) )}{\varepsilon_i} \right| \phi(x,t) \, \dx\, \dt=0
\end{equation}
for all $\phi \in C(\overline{\Omega}\times[0,\infty))$ and any $0<T<\infty$. 
The property \eqref{eq:vd} is called the {\it vanishing of the discrepancy measure} $\xi_t^{\varepsilon_i}$ 
 and is the key to show the main result stated in Section \ref{sec:intro}. The proof of \eqref{eq:vd} relies on the upper bound of the density for the measure $\mu_t^{\varepsilon_i}$, which we established in Proposition \ref{prop:drub}. 

The content of this section is based on \cite[Section 7]{kagaya} with some modifications coming from the forcing term. Where proofs are the same as in \cite{kagaya} we omit the details and refer directly to  \cite{kagaya}. 
\subsection{Existence of limit measures}
To prove property \eqref{eq:vd}, we first show that there exists a family of Radon measures $\{\mu_t\}_{t \geq 0}$ such that after taking a subsequence, $\mu_t^{\varepsilon_i} \rightharpoonup \mu_t$ as $i \rightarrow \infty$ for all $t\geq 0$ on $\overline{\Omega}$. Notice that, as we want to consider up to the boundary convergence of 
$\mu_t^{\varepsilon_i}$, we select a test function that is not identically zero near the boundary $\partial \Omega$. 

\begin{proposition}
	\label{prop:limitmeas}
Let $u^{\varepsilon_i}$ be a solution to \eqref{prob:AC}. Then there exists a family of Radon measures $\{\mu_t\}_{t \geq 0}$  and a subsequence (denoted by the same index) such that $\mu_t^{\varepsilon_i} \rightharpoonup \mu_t$ as $i \rightarrow \infty$ for all $t\geq 0$ on $\overline{\Omega}$.
\end{proposition}

The proof of Proposition \ref{prop:limitmeas} follows the line of proof of \cite[Proposition 5.2]{mizuno-tonegawa} and relies on the semidecreasing property of $\mu_t^{\varepsilon}$,  which we give in the next reult. 

\begin{lemma}[Semidecreasing property]
	\label{lem:semidec} 
Let $u^{\varepsilon}$ be a solution to \eqref{prob:AC}. 
For all $\phi \in C^2(\overline{\Omega})$ with $\phi \geq 0$ in $\overline{\Omega}$, there exists  a constant $c_{13} >0$ depending only on $c_1$ and $c_9$ such that 
\begin{equation}
\label{eq:semidec}
 \mu_t^{\varepsilon}(\phi) - c_{13} \Vert\phi\Vert_{C^2(\overline{\Omega})} t
\end{equation}
is decreasing with respect to $t$ for any  $0<\varepsilon <1$. 
\end{lemma}

Note that in the proof below we have to take the forcing term  $g^{\varepsilon}$ into account, which leads to a different proof than the one given in \cite[Lemma 5.1]{mizuno-tonegawa}. 

\begin{proof}
By integration by parts and the boundary condition in \eqref{prob:AC}, we get 
\begin{equation}
	\label{eq:semidecproof1}
\begin{split}
 \frac{ \d}{\d t} \mu_t^{\varepsilon}(\phi) 
 & = \frac{1}{\sigma} \int_{\Omega} \phi \, \frac{\partial}{\partial t} \left( \frac{\varepsilon \vert \nabla u^{\varepsilon} \vert^2}{2} + \frac{W(u^{\varepsilon})}{\varepsilon} \right) \d x
 \\
 & =  \frac{1}{\sigma} \int_{\Omega} \phi \left( \varepsilon \nabla u^{\varepsilon} \cdot \nabla u_t^{\varepsilon} + \frac{W'(u^{\varepsilon})}{\varepsilon} u_t^{\varepsilon} \right) \d x
 \\
 & =  \frac{1}{\sigma} \int_{\Omega} \left \{ \varepsilon \phi \left(  - \Delta u^{\varepsilon} + 
 \frac{W'(u^{\varepsilon})}{\varepsilon^2}  \right) u_t^{\varepsilon} - \varepsilon (\nabla \phi \cdot \nabla u^{\varepsilon} )  u_t^{\varepsilon}   \right \} \d t
  \\
  &=  \frac{1}{\sigma} \int_{\Omega} \varepsilon \phi \left( -u_t^{\varepsilon} + g^{\varepsilon} \frac{\sqrt{2W(u^{\varepsilon})}}{\varepsilon} \right) u_t^{\varepsilon} \, \d x  
 -  \frac{1}{\sigma} \int_{\Omega} \varepsilon (\nabla \phi \cdot \nabla u^{\varepsilon} )  u_t^{\varepsilon}  \, \d x. 
\end{split}
\end{equation}
For the first integral on the right-hand side of \eqref{eq:semidecproof1}, we use Young's inequality, 
the estimate $\|g^{\varepsilon} \| _{L_\infty(\Omega)} \leq c_1$ and Corollary \ref{cor:boundmeas} to obtain 
\begin{equation*}
\begin{split} 
	\frac{1}{\sigma} &\int_{\Omega} \varepsilon \phi \left( -u_t^{\varepsilon} + g^{\varepsilon} \frac{\sqrt{2W(u^{\varepsilon})}}{\varepsilon} \right) u_t^{\varepsilon} \, \d x  
\\
& = - 	\frac{1}{\sigma} \int_{\Omega} \varepsilon \phi (u_t^{\varepsilon})^2 \, \d x 
+ \frac{1}{\sigma} \int_{\Omega} \sqrt{\varepsilon}\,  \phi \, g^{\varepsilon} \frac{\sqrt{2W(u^{\varepsilon})}}{\sqrt{\varepsilon}} u_t^{\varepsilon} \, \d x  
\\
& \leq  - 	\frac{1}{\sigma} \int_{\Omega} \varepsilon \phi (u_t^{\varepsilon})^2 \, \d x  
+ 	\frac{1}{2\sigma} \int_{\Omega} \varepsilon \phi (u_t^{\varepsilon})^2 \, \d x 
+ \frac{c_1^2}{\sigma} \int_{\Omega} \phi \frac{W(u^{\varepsilon})}{\varepsilon} \, \d x 
\\
& \leq  - 	\frac{1}{2\sigma} \int_{\Omega} \varepsilon \phi (u_t^{\varepsilon})^2 \, \d x  
+ c_1^2 \Vert\phi\Vert_{C(\overline{\Omega})}  \mu_t^{\varepsilon}(\Omega) 
\\
& \leq - 	\frac{1}{2\sigma} \int_{\Omega} \varepsilon \phi (u_t^{\varepsilon})^2 \, \d x  
+ c_1^2 c_9 \Vert\phi\Vert_{C(\overline{\Omega})}. 
\end{split}
\end{equation*}
The second integral on the right-hand side of \eqref{eq:semidecproof1} is estimated by Young's inequality, $\sup_{x \in \{x \mid \phi(x) >0\} } \frac{|\nabla \phi|^2}{\phi}  \leq  2 \Vert \nabla^2 \phi\Vert_{C(\overline{\Omega})}$ and Corollary \ref{cor:boundmeas} 
as 
\begin{equation*}
\begin{split}
 -  \frac{1}{\sigma} \int_{\Omega} \varepsilon (\nabla \phi \cdot \nabla u^{\varepsilon} )  u_t^{\varepsilon}  \, \d x 
 & \leq  \frac{1}{\sigma} \int_{\Omega \cap\{ \phi >0 \}}  \varepsilon \left \vert \frac{\nabla \phi}{\sqrt{\phi}} \cdot
  \nabla u^{\varepsilon} \right \vert \, \left\vert \sqrt{\phi} \, u_t^{\varepsilon}\right \vert \d x
\\
& \leq  \frac{1}{\sigma} \int_{\Omega \cap\{ \phi >0 \}} \frac{\varepsilon \vert \nabla  u^{\varepsilon} \vert ^2}{2} \, 
 \frac{|\nabla \phi|^2}{\phi} \d x + 	\frac{1}{2\sigma} \int_{\Omega} \varepsilon \phi (u_t^{\varepsilon})^2 \, \d x  
 \\
 & \leq 2  c_9 \Vert \nabla^2 \phi\Vert_{C(\overline{\Omega})} + 	\frac{1}{2\sigma} \int_{\Omega} \varepsilon \phi (u_t^{\varepsilon})^2 \, \d x . 
\end{split} 
\end{equation*}
Consequently we find that 
\begin{equation*}
	\begin{split}
\frac{ \d}{\d t} \mu_t^{\varepsilon}(\phi)  
&\leq - 	\frac{1}{2\sigma} \int_{\Omega} \varepsilon \phi (u_t^{\varepsilon})^2 \, \d x  
+ c_1^2 c_9 \Vert\phi\Vert_{C(\overline{\Omega})} + 2  c_9 \Vert \nabla^2 \phi\Vert_{C(\overline{\Omega})} + 	\frac{1}{2\sigma} \int_{\Omega} \varepsilon \phi (u_t^{\varepsilon})^2 \, \d x 
\\
& \leq 
c_9 (c_1^2 +1) \Vert\phi\Vert_{C^2(\overline{\Omega})}. 
\end{split} 
\end{equation*}
\end{proof}

\noindent \textit{Proof of Proposition \ref{prop:limitmeas}.} Having Lemma \ref{lem:semidec} at hand, the proof of this proposition is identical to the proof of \cite[Proposition 5.1]{mizuno-tonegawa}, so we omit it. \qed \bigskip

We now define the Radon measure $\mu^{\varepsilon_i}$ and $|\xi^{\varepsilon_i}|$ on $\overline{\Omega} \times [0,\infty)$ as 
\begin{equation*}
	\d \mu^{\varepsilon_i} := \d \mu_t^{\varepsilon_i} \, \d t \quad \text{and} \quad \d |\xi^{\varepsilon_i}| := \d |\xi_t^{\varepsilon_i}| \, \d t. 
\end{equation*}
Since, by Corollary \ref{cor:boundmeas},  $\mu_t^{\varepsilon_i}(\Omega)$ is uniformly bounded with respect to $\varepsilon_i$ and $t\in [0,\infty)$ (actually, locally in time), the weak compactness theorem of Radon measures ensures the existence of subsequential limits $\mu$ and $|\xi|$ of $\mu^{\varepsilon_i}$ and $|\xi^{\varepsilon_i}|$ on $\overline{\Omega} \times [0,\infty)$ (locally in time), respectively. By the boundedness of $\sup_{i \in \N} \mu_t^{\varepsilon_i}(\Omega)$, see Corollary \ref{cor:boundmeas} again  and Proposition \ref{prop:limitmeas}, the dominated convergence theorem implies $\d \mu= \d \mu_t \, \d t$. 

Note that in Proposition \ref{prop:limitmeas}, we verified that there exists $\mu_t = \lim_{i \to \infty} \mu_t^{\varepsilon_i}$ for any $t\geq 0$, however such a property does not necessarily hold true for 
$\xi_t^{\varepsilon_i}$. Our goal is to show that $|\xi|$ is identically $0$ (see Theorem \ref{thm:vd} below). First, we collect for the measures $\mu$ and $\mu_t$ the following result.

\begin{lemma}
	\label{lem:sptlimitmeas}
For all $t\geq 0$, we have $	\textnormal{spt}\mu_t \subset  \{ x \in \overline{\Omega} \mid  (x,t) \in \textnormal{spt} \mu\}$.
\end{lemma}

\begin{proof}
The proof follows as the one in \cite[Lemma 7.1]{kagaya} by using Proposition \ref{prop:limitmeas} and Lemma \ref{lem:semidec}.
\end{proof}
\subsection{Vanishing of $|\xi|$}

\begin{lemma}
\label{lem:boundsol}
Assume $(y,s) \in \textnormal{spt} \mu$ with $y \in \overline{\Omega}$ and $s>0$. Then there exists a sequence $\{ (x_i,t_i)\}_{i=1}^\infty$ and a subsequence $\varepsilon_i$ (denoted by the same index) 
such that $t_i>0$, $x_i \in \Omega$, $\lim\limits_{i \to \infty} (x_i,t_i) =(y,s)$ and $|u^{\varepsilon_i} (x_i, t_i)| < \alpha$ for all $i\in \N$. 
\end{lemma}

\begin{proof}
The proof uses arguments similar to the proof of \cite[Lemma 7.5]{kagaya} but some modifications are needed. Hence, we include the proof for completeness. 

If the claim were not true, there would exist $0<r_0< \sqrt{s}$ such that 
\begin{equation}
\label{eq:boundsol1}
\inf_{(B_{r_0}(y)\cap \Omega)\times (s-r_0^2,s+r_0^2)} |u^{\varepsilon_i}|\geq \alpha
\end{equation}
for all sufficiently large $i$. Let us take a function $\phi \in C^1_c(B_{r_0}(y))$ such that 
\begin{equation*}
	|\nabla \phi| \leq \frac{3}{r_0}, \quad \phi \equiv 1 \quad \text{on} \quad  B_{\frac{r_0}{2}}(y). 
\end{equation*}
We multiply the equation $\varepsilon_i \, u_t^{\varepsilon_i} = \varepsilon_i \, \Delta u^{\varepsilon_i} - \tfrac{W'(u^{\varepsilon_i})}{\varepsilon_i} + g^{\varepsilon_i} \sqrt{2 W(u^{\varepsilon_i})}$ by the  
function $\phi^2 W'(u^{\varepsilon_i})$, integrate over $\Omega$ and use integration by parts, the Neumann boundary condition for $u^{\varepsilon_i}$, assumption (W3), Proposition \ref{prop:cp}, and \eqref{eq:boundsol1} to obtain that 
\begin{equation}
\begin{split}
	\label{eq:boundsol2}
\frac{\d }{\d t} &\int_{\Omega} \varepsilon_i \, \phi^2 \, W(u^{\varepsilon_i}) \, \d x 
\\
& = \int_{\Omega} \varepsilon_i  \phi^2 W'(u^{\varepsilon_i}) u_t^{\varepsilon_i} \, \d x 
\\
&= \int_{\Omega} \left\{ \varepsilon_i \phi^2 W'(u^{\varepsilon_i})  \Delta u^{\varepsilon_i} - 
\phi^2 \frac{(W'(u^{\varepsilon_i}))^2}{\varepsilon_i} + \phi^2 W'(u^{\varepsilon_i})  g^{\varepsilon_i} \sqrt{2W(u^{\varepsilon_i})} \right\} \d x
\\
& = -\int_{\Omega} \left\{ \varepsilon_i \phi^2 W''(u^{\varepsilon_i}) |\nabla u^{\varepsilon_i}|^2 + 
2\varepsilon_i \phi W'(u^{\varepsilon_i}) \nabla u^{\varepsilon_i} \cdot \nabla \phi + \phi^2 
\frac{(W'(u^{\varepsilon_i}))^2}{\varepsilon_i} \right\} \d x 
\\
& \quad + \int_{\Omega} \phi^2 W'(u^{\varepsilon_i})
g^{\varepsilon_i} \sqrt{2 W(u^{\varepsilon_i})} \, \d x
\\
&\leq - \int_{\Omega} \left\{  \varepsilon_i \beta \phi^2 |\nabla u^{\varepsilon_i}|^2  + 
2\varepsilon_i \phi W'(u^{\varepsilon_i}) \nabla u^{\varepsilon_i} \cdot \nabla \phi + \phi^2 
\frac{(W'(u^{\varepsilon_i}))^2}{\varepsilon_i} \right\} \d x 
\\
& \quad + \int_{\Omega} \phi^2 W'(u^{\varepsilon_i})
g^{\varepsilon_i} \sqrt{2 W(u^{\varepsilon_i})} \, \d x,
\end{split}
\end{equation}
for $s-r_0^2 <t <s+r_0^2$. By the Cauchy-Schwarz and Young inequalities, we have 
\begin{equation*}
- \int_{\Omega}  2\varepsilon_i \phi W'(u^{\varepsilon_i}) \nabla u^{\varepsilon_i} \cdot \nabla \phi 
\, \d x \leq \int_{\Omega}  \phi^2 \frac{(W'(u^{\varepsilon_i}))^2}{4\varepsilon_i} \d x 
+ 4 \varepsilon_i^3 \int_{\Omega} |\nabla \phi|^2 |\nabla u^{\varepsilon_i}|^2 \, \d x
\end{equation*}
and 
\begin{equation*}
\int_{\Omega} \phi^2 W'(u^{\varepsilon_i})
g^{\varepsilon_i} \sqrt{2 W(u^{\varepsilon_i})} \, \d x \leq 
\int_{\Omega}  \phi^2 \frac{(W'(u^{\varepsilon_i}))^2}{4\varepsilon_i} \, \d x 
+ 2 \varepsilon_i \Vert g^{\varepsilon_i}\Vert_{L_\infty(\Omega)}^2 \int_{\Omega} \phi^2 W(u^{\varepsilon_i}) \, \d x,
\end{equation*}
for $s-r_0^2 <t <s+r_0^2$. Inserting these two estimates into \eqref{eq:boundsol2}, we get after rearranging that 
\begin{equation*}
\begin{split}
\int_{\Omega}& \phi^2 \left( \varepsilon_i \, \beta |\nabla u^{\varepsilon_i}|^2 + \frac{(W'(u^{\varepsilon_i}))^2}{2\varepsilon_i} \right) \d x 
\\
& \leq  4\varepsilon_i^3  \int_{\Omega} |\nabla \phi|^2 |\nabla u^{\varepsilon_i}|^2 \, \d x 
+ 2 \varepsilon_i c_1^2 \sup_{|v|\leq 1} W(v)  \int_{\Omega} \phi^2 \, \d x - \frac{\d }{\d t} \int_{\Omega} \varepsilon_i \, \phi^2 \, W(u^{\varepsilon_i}) \, \d x , 
\end{split}
\end{equation*}
for $s-r_0^2 <t <s+r_0^2$. Integrating the above inequality from $s-r_0^2$ to $s+r_0^2$ with respect to $t$ and using assumption (W1) and Corollary \ref{cor:boundmeas}, we arrive at 
\begin{equation}
\label{eq:boundsol3}
\lim_{i\to \infty} \int_{s-r_0^2}^{s+r_0^2} \int_{B_{\frac{r_0}{2}}(y)}  \left( \varepsilon_i |\nabla u^{\varepsilon_i}|^2 + \frac{(W'(u^{\varepsilon_i}))^2}{\varepsilon_i} \right) \d x \, \d t=0.
\end{equation}
Next, due to \eqref{eq:boundsol1}, Proposition \ref{prop:cp} and the continuity of $u^{\varepsilon_i}$ we may assume that $\alpha \leq u^{\varepsilon_i} \leq 1$ on $(B_{r_0}(y)\cap \Omega)\times (s-r_0^2,s+r_0^2)$ without loss of generality. Otherwise, we have $-1 \leq u^{\varepsilon_i} \leq- \alpha$ and we may argue similarly. It follows from assumption (W1) that there exists a positive constant $c(W)$ only depending on $W$ such that $W(s) \leq c(W) (s-1)^2$ for all $s\in [\alpha,1]$. Moreover, assumptions (W1) and (W3) yield with Cauchy's mean value theorem that $W'(s) = W'(s) -W'(1)\leq \beta (s-1)\leq 0$ for all $s\in [\alpha,1]$. Hence, we obtain the inequality 
\begin{equation}
\label{eq:boundsol4}	
\int_{s-r_0^2}^{s+r_0^2} \int_{B_{\frac{r_0}{2}}(y)} \frac{W(u^{\varepsilon_i})}{\varepsilon_i} \, 
\d x \, \d t \leq c(W,\beta) \int_{s-r_0^2}^{s+r_0^2} \int_{B_{\frac{r_0}{2}}(y)} \frac{(W'(u^{\varepsilon_i}))^2}{\varepsilon_i} \, 
\d x \, \d t
\end{equation}
for some positive constant $c(W,\beta)$ only depending on $W$ and $\beta$. Combining \eqref{eq:boundsol3} and \eqref{eq:boundsol4} we conclude $\mu(B_{\frac{r_0}{2}}(y) \times (s-r_0^2,s+r_0^2))=0$ which is clearly contradicting $(y,s) \in \text{spt} \mu$. 
\end{proof}

\begin{theorem}\label{thm:vd}
We have  $|\xi|=0$ on $\overline{\Omega} \times(0,\infty)$.
\end{theorem}

\begin{proof} 
It is enough to show $|\xi|=0$ on $\overline{\Omega} \times(0,T)$ for all $0<T<\infty$.  
From now on in this proof, let us fix $T$.  Note that, with the definition of $|\xi^{\varepsilon_i}|$, we have that $\text{spt} |\xi| \subset \overline{\Omega} \times [0,\infty)$. Let $y \in \Omega\setminus N_{\frac{c_4}{2}}$ and $0 \leq t<s<T$. It follows from the monotonicity formula \ref{eq:mformula_i} that 
\begin{equation}
\begin{split}
\label{eq:vanishdiscre1} 
		\frac{\d}{\dt} & \left( e^{c_7 (s-t)} \int_{\Omega} \rho_{1,(y,s)}(x,t) \,    \d \mu_t^{\varepsilon_i}(x) \right) 
+ e^{c_7 (s-t)}  \int_{\Omega} \frac{\rho_{1,(y,s)}(x,t)}{2(s-t)} \, \d | \xi_t^{\varepsilon_i}|(x) 
\\
&\leq \frac{2}{\sigma} e^{c_7 (s-t)}  \int_{\Omega} \frac{\rho_{1,(y,s)}(x,t)}{2(s-t)} \, 
\left( \frac{\varepsilon_i |\nabla u^{\varepsilon_i}|^2}{2} - \frac{W(u^{\varepsilon_i})}{\varepsilon_i}\right)^+ \d x 
+ c_7  e^{c_7 (s-t)} e^{- \frac{c_4^2}{64(s-t)}}  \mu_t^{\varepsilon}\big( B_{\frac{c_4}{2}}(y)  \big).
\end{split}
\end{equation}
Applying Proposition \ref{prop:drub} and $s-T<T$ to \eqref{eq:vanishdiscre1}, we get 
\begin{equation}
\begin{split}
\label{eq:vanishdiscre2} 
\frac{\d}{\dt} & \left( e^{c_7 (s-t)} \int_{\Omega} \rho_{1,(y,s)}(x,t) \,    \d \mu_t^{\varepsilon_i}(x) \right) 
+  \int_{\Omega} \frac{\rho_{1,(y,s)}(x,t)}{2(s-t)} \, \d | \xi_t^{\varepsilon_i}|(x) 
\\
&\leq \frac{2}{\sigma} e^{c_7 T}  \int_{\Omega} \frac{\rho_{1,(y,s)}(x,t)}{2(s-t)} \, 
\left( \frac{\varepsilon_i |\nabla u^{\varepsilon_i}|^2}{2} - \frac{W(u^{\varepsilon_i})}{\varepsilon_i}\right)^+ \d x 
+ c_7  c_{12} \left( \frac{c_4}{2}\right)^{n-1} \omega_{n-1} e^{c_7 T}. 
\end{split}
\end{equation}
Integrating \eqref{eq:vanishdiscre2}  with respect to $t$ on $(0,s)$ and using the inequality 
\begin{equation*}
	\begin{split}
		\frac{1}{\sigma}	\int_0^{s}  \int_\Omega  \frac{\rho_{1,(y,s)}(x,t)}{2(s-t)} \left( \frac{\varepsilon_i |\nabla u^{\varepsilon_i}|^2}{2} - \frac{W(u^{\varepsilon_i})}{\varepsilon_i}\right)^+ \, \d x\d t 
		\leq c_{13}^K {\varepsilon_i}^{\tfrac{1}{4}} \left(  1+ |\log \varepsilon_i| + (\log s)^+ \right)
	\end{split}
\end{equation*}
which is shown by the same proof as in \cite[Lemma 6.7]{kagaya} and with $c_{13}^K$ denoting 
constant $c_{13}$ from \cite{kagaya}, we obtain that 
\begin{equation*}
 \int_{(0,s)\times\Omega} \frac{\rho_{1,(y,s)}(x,t)}{2(s-t)} \, \d | \xi^{\varepsilon_i}|(x,t) 
 \leq c(c_7,c_{12},c_{13}^K,D_0,T)\left( 1+ {\varepsilon_i}^{\tfrac{1}{4}} \Big( 1+ |\log \varepsilon_i| + (\log s)^+ \Big) \right). 
\end{equation*}
For $y\in N_{\frac{c_4}{2}} \cap \overline{\Omega}$ and $0\leq t<s<T$, by arguing exactly as in 
\cite[pp.1519-1520]{kagaya}, we obtain the estimate 
\begin{equation}
\label{eq:vanishdiscre3}
\begin{split}
\int_{(0,s)\times\Omega} &\left(  \frac{\rho_{1,(y,s)}(x,t)}{2(s-t)} + \frac{\eta(|x-\tilde{y}|) \, e^{- \frac{|x-\tilde{y}|^2}{4(s-t)}}}{2^n \,  \pi^{\frac{n-1}{2}} \, (s-t)^{\frac{n+1}{2}}}
 \right)  \d | \xi^{\varepsilon_i}|(x,t) 
\\
&\leq c(c_7,c_{12},c_{13}^K,D_0,T)\left( 1+ {\varepsilon_i}^{\tfrac{1}{4}} \Big( 1+ |\log \varepsilon_i| + (\log s)^+ \Big) \right). 
\end{split}
\end{equation}
Letting $i\to \infty$ and integrating the limit of \eqref{eq:vanishdiscre3} over $(y,s) \in \overline{\Omega}\times (0,T)$ with respect to $ \d \mu_s(y) \d s$, we have
\begin{equation*}
\int_0^T \d s \int_{\overline{\Omega}} \d \mu_s(y) \int_{\overline{\Omega} \times (0,T)} 
\left(  \frac{\rho_{1,(y,s)}(x,t)}{2(s-t)} + \frac{\eta(|x-\tilde{y}|) \, e^{- \frac{|x-\tilde{y}|^2}{4(s-t)}}}{2^n \,  \pi^{\frac{n-1}{2}} \, (s-t)^{\frac{n+1}{2}}}
\right)  \d | \xi|(x,t) < \infty. 
\end{equation*}
The remainder of the proof is exactly the same as in \cite[pp.1520-1521]{kagaya}, hence we omit it. 
\end{proof}

\section{Subsolution and supersolution}
\label{sec:subsupersolution}
In this section, we construct sub- and supersolutions to \eqref{prob:AC} that correspond to the obstacles. For $0< \delta< \min \{ \text{dist} ( O_+,\partial \Omega ), \text{dist} ( O_-,\partial \Omega ) \}$ 
and $x \in B_{R_0+\delta}(y) \subset \Omega$, we set 
\begin{equation*}
	r_{y,-}(x):= - c_{14}\, \frac{1}{(R_0+\delta)^2-|x-y|^2} +c_{15}, 
\end{equation*}
where
\begin{equation*}
c_{14}:= \frac{\big((R_0+\delta)^2-R_0^2\big)^2}{2R_0} \quad \text{and}\quad c_{15}:= \frac{(R_0+\delta)^2-R_0^2}{2R_0} +s^\varepsilon _\gamma
\end{equation*}
with  $s^\varepsilon _\gamma\in (-1,1)$ introduced in Section \ref{sec:AC},  and we define 
\begin{equation*}
	u^{\varepsilon}_{y,-}(x) := q^\varepsilon(r_{y,-}(x)) 	\qquad \text{and} \qquad  u^{\varepsilon}_{y,+}(x) := q^\varepsilon(-r_{y,-}(x) + 2 s^\varepsilon _\gamma). 
\end{equation*}
\smallskip

\begin{lemma}\label{lemm:subsol}
Assume that $B_{R_0}(y) \subset O_+$. Then there exists $\epsilon_1=\epsilon_1(\delta)>0$ 
such that $u^{\varepsilon}_{y,-}$ is a subsolution to \eqref{prob:AC} on $B_{R_0+\delta}(y)$ for any 
$\varepsilon \in (0,\epsilon_1)$. 
\end{lemma}

\begin{proof}
Without loss of generality, we assume that $y=0$. Let $u^\varepsilon$ be a solution to \eqref{prob:AC}. Using \eqref{eq:prop1q}, we compute that 
\begin{align*}
\varepsilon q^\varepsilon_r r_t^\varepsilon	
&= 
\varepsilon q^\varepsilon_r \Delta r^\varepsilon + \varepsilon q^\varepsilon_{rr} |\nabla r^\varepsilon|^2 
-\dfrac{W'(q^\varepsilon)}{\varepsilon} + g^\varepsilon \sqrt{2W(q^\varepsilon)}\\
&= \sqrt{2W(q^\varepsilon)} \Delta r^\varepsilon + \dfrac{W'(q^\varepsilon)}{\varepsilon}  ( |\nabla r^\varepsilon|^2 -1) + g^\varepsilon \sqrt{2W(q^\varepsilon)}
\end{align*}
for any $(x,t)\in \Omega \times (0,\infty)$. Hence, the first equation in \eqref{prob:AC} is equivalent to 
\begin{equation}\label{eq:ACr}
r_t^\varepsilon	- \Delta r^\varepsilon - \dfrac{G(q^\varepsilon(r^\varepsilon))}{\varepsilon  } 
( |\nabla r^\varepsilon|^2 -1) - g^\varepsilon=0, \qquad (x,t)\in \Omega\times(0,\infty),
\end{equation}
where we have set $G(s):=\dfrac{W'(s)}{\sqrt{2W(s)} }$ for any $s\in (-1,1)$. Therefore we only need to prove that 
\begin{equation}\label{eq:subACr}
- \Delta r_{0,-} - \dfrac{G (q^\varepsilon(r_{0,-}))}{\varepsilon} 
	( |\nabla r_{0,-}|^2 -1) - g^\varepsilon\leq 0, \qquad x \in B_{R_0+\delta}(0),
\end{equation}
for sufficiently small $\varepsilon >0$.  We note that $ r_{0,-} > s^\varepsilon _\gamma $ in $B_{R_0}(0)$, $ r_{0,-} \leq s^\varepsilon _\gamma $ in $B_{R_0}(0)^c$ and $ r_{0,-} =-\infty$ in $\partial B_{R_0+\delta}(0)$. In addition, $ |\nabla r_{0,-}| < 1$ in $B_{R_0}(0)$ and $ |\nabla r_{0,-}| \geq 1$ in $B_{R_0}(0)^c$. A simple calculation shows that 
\begin{align*}
	- &\Delta r_{0,-} - \dfrac{G(q^\varepsilon(r_{0,-}))}{\varepsilon} 
	( |\nabla r_{0,-}|^2 -1) - g^\varepsilon 
	\\
	& \leq 2c_{14} \  \dfrac{\big(n(R_0+\delta)^2 + |4-n| |x|^2\big)}{\big| (R_0+\delta)^2 - |x|^2\big|^3} - \dfrac{G(q^\varepsilon(r_{0,-}))}{\varepsilon} \left( \dfrac{4c_{14}^2 \ |x|^2}{\big( (R_0+\delta)^2 - |x|^2\big)^4} -1 \right)  - g^\varepsilon . 
\end{align*}
First, we consider the case of $|x| \leq R_0+ \sqrt{\varepsilon}$ with $\sqrt{\varepsilon} < \tfrac{\delta}{2}$. By \eqref{eq:prop4q} and the properties of $ r_{0,-}$ in $B_{R_0}(0)$ and $B_{R_0}(0)^c$ 
we have
\begin{equation*}
	- \dfrac{G(q^\varepsilon(r_{0,-}))}{\varepsilon} 	( |\nabla r_{0,-}|^2 -1) \leq 0. 
\end{equation*}
Note that 
$G(s)<0$ for $s>\gamma $ (resp. $G(s)>0$ for $s< \gamma$).
Using this, along with the definition of $g^\varepsilon$, we find that 
\begin{align*}
	- &\Delta r_{0,-} - \dfrac{G(q^\varepsilon(r_{0,-}))}{\varepsilon} 
	( |\nabla r_{0,-}|^2 -1) - g^\varepsilon 
	\\
	& \leq 2c_{14}  \, \dfrac{\big( n(R_0+\delta)^2 + |4-n| (R_0+\sqrt{\varepsilon})^2\big)}{\delta^3 \big( R_0+\frac{3}{4}\delta\big)^3} - g^\varepsilon \leq c_1 - g^\varepsilon =0,
\end{align*}
where $c_1$ is defined in \eqref{eq:c1}. 
Next, we consider the case of $R_0+ \sqrt{\varepsilon} \leq |x| < R_0+\delta$ with $\sqrt{\varepsilon} < \tfrac{\delta}{2}$.  Since 
\begin{equation*}
	|\nabla r_{0,-}|^2  = \dfrac{4c_{14}^2 \ |x|^2}{\big( (R_0+\delta)^2 - |x|^2\big)^4}  \qquad \text{and} \qquad |\Delta r_{0,-}| \leq 2c_{14}  \ \dfrac{\big(n(R_0+\delta)^2 + |4-n| |x|^2\big)}{\big| (R_0+\delta)^2 - |x|^2\big|^3}, 
\end{equation*}
we find that 
\begin{equation}\label{eq:est1sub}
	|\nabla r_{0,-}|^2 \gg 1 \quad \text{and} \quad |\nabla r_{0,-}|^2 \gg |\Delta r_{0,-}|  \qquad \text{for } \ |x| \to R_0+\delta. 
\end{equation}
Since $r_{0,-}(x) \to -\infty$ as $|x| \to R_0+\delta$, we get by the properties of $q^\varepsilon$ that 
\begin{equation*}
q^\varepsilon(r_{0,-}(x)) \to -1  \qquad \text{as } \ |x| \to R_0+\delta. 
\end{equation*}
Moreover, for sufficiently small $\varepsilon$, we have that 
\begin{equation}\label{eq:est2sub}
	G(q^\varepsilon(r_{0,-}(x))) \geq \sqrt{\beta} >0 .
\end{equation}
Indeed, defining $s_\alpha^\varepsilon \in \R$ by $q^\varepsilon(s_\alpha^\varepsilon) = - \alpha$, we have 
$|s_\alpha^\varepsilon| \leq c_\# \varepsilon$ for some positive constant $c_\#$ independent of $\varepsilon$. 
 Computing 
\begin{equation*}
	 r_{0,-}(x) \leq \frac{((R_0+\delta)^2- R_0^2)}{2R_0} \left( \frac{-2R_0\sqrt{\varepsilon} - \varepsilon }{2R_0\delta + \delta^2-2R_0\sqrt{\varepsilon} - \varepsilon}\right) +s_\gamma^\varepsilon 
\end{equation*}
and using $|s_\gamma^\varepsilon |\leq c_{\#\#} \varepsilon$ with $c_{\#\#}>0$ independent of $\varepsilon$, we see that there exists a constant $C>0$, independent of $\varepsilon$, such that $r_{0,-}(x) \leq -C \sqrt{\varepsilon}$ for small $\varepsilon$. Hence, 
\begin{equation*}
	r_{0,-}(x) \leq -C \sqrt{\varepsilon} \leq s_\alpha^\varepsilon,
\end{equation*}
which, together with $q^\varepsilon(s_\alpha^\varepsilon) = - \alpha$, gives 
$q^\varepsilon(r_{0,-}(x)) \leq - \alpha$. By \eqref{eq:prop3q}, we get \eqref{eq:est2sub}. 

Therefore, \eqref{eq:est1sub}, \eqref{eq:est2sub} and the fact that $ g^\varepsilon(x) \geq 0$ imply 
\eqref{eq:subACr} if $R_0+ \sqrt{\varepsilon} \leq |x| < R_0+\delta$. 
\end{proof}

Applying similar arguments, we obtain the following result. 

\begin{lemma}\label{lemm:supersol}
	Assume that $B_{R_0}(y) \subset O_-$. Then $u^{\varepsilon}_{y,+}$ is a supersolution to \eqref{prob:AC} on $B_{R_0+\delta}(y)$ for any 
	$\varepsilon \in (0,\epsilon_1)$, where $\epsilon_1$ is as in Lemma \ref{lemm:subsol}. 
\end{lemma}

By the usual comparison principle for parabolic problems, we can show the following result. 

\begin{lemma}
	Assume that $B_{R_0} (y) \subset O_+$. Then for sufficiently small $\varepsilon>0$ it holds that
	\begin{equation}\label{lem6.3}
		u_{y,-} ^\varepsilon \leq u ^\varepsilon (\cdot ,t) \quad \text{on} \ B_{R_0+\delta} (y)
	\end{equation}
	for any $ t \in [\varepsilon ^{\beta_\ast},\infty) $.
	Here, $\beta_\ast \in (0,\frac12)$ is the constant given in Appendix A.
\end{lemma}
\begin{proof}
	First, we remark that $u_0 ^{\varepsilon_i} = q^{\varepsilon_i} (r^i ) $ and
	\[
	r^i (x) = \frac23 c_{\ast \ast } \varepsilon _i ^{\beta _\ast} \qquad \text{for any} \ x \in O_+, 
	\]
	where $r^i$, $\beta_\ast$ and $c_{\ast \ast}$ are introduced in Appendix A. 
	Suppose without losing generality that $y=0$.
	Set  for $r\in \R$ 
	\[
	\phi _0 ^{\varepsilon} (r) :=
	\begin{cases}
		r & \text{if} \ r \leq \frac{1}{3} c_{\ast \ast} \varepsilon  ^{\beta_\ast}, \\
		\frac{1}{3} c_{\ast \ast} \varepsilon ^{\beta_\ast}
		\left(1+\tanh \left( \frac{r- \frac{1}{3} c_{\ast \ast} \varepsilon ^{\beta_\ast}}{\frac{1}{3} c_{\ast \ast} \varepsilon ^{\beta_\ast}} \right) \right) & 
		\text{if} \ r > \frac{1}{3} c_{\ast \ast} \varepsilon  ^{\beta_\ast} .
	\end{cases}
	\]
	Note that $\phi _0 ^{\varepsilon} $ is monotone increasing and $C^2$ since $\tanh' (0)=1$ and $\tanh'' (0)=0$, and
	$\phi _0 ^\varepsilon (r) \leq \frac{2}{3} c_{\ast \ast} \varepsilon ^{\beta_\ast}$
	for any $r \in \R$. In addition, there exists $C>0$ such that $| (\phi _0 ^{\varepsilon} (r))''| \leq C \varepsilon ^{-\beta_\ast}$ for any $r \in \R$.
	Setting, for $r\in \R$ and $t\in [0,\varepsilon  ^{\beta_\ast} ]$, 
	\[
	\phi^{\varepsilon} (r,t):= \phi _0 ^{\varepsilon} (r) 
	+ \frac{t}{\varepsilon ^{\beta_\ast}} (r -\phi _0 ^{\varepsilon} (r) )
	\quad
	\text{and}
	\quad
	\tilde r_{-} ^{\varepsilon } (x,t)
	:= \phi^{\varepsilon} (r_{0,-} (x),t), 
	\]
	one can check that
	\begin{enumerate}
		\item $\tilde r_{-} ^{\varepsilon} (x,0) \leq \frac{2}{3} c_{\ast \ast} \varepsilon ^{\beta_\ast} $
		for any $x \in B_{R_0 +\delta} (0)$. Note that 
		$\tilde r_{-} ^{\varepsilon_i} (x,0) \leq r^i (x)$ for any $ x\in B_{R_0 +\delta} (0)$ and  any $i\geq1 $. 
		\\[-0.2cm]
		\item $\tilde r_{-} ^{\varepsilon} (x,t) = r_{0,-} (x)$ when 
		$r_{0,-} (x) \leq \frac{1}{3} c_{\ast \ast} \varepsilon ^{\beta_\ast}$, that is, 
		\[
		|x|^2 \geq 
		R_0 ^2 + \frac{2R_0 (2\delta R_0 +\delta^2) (s_\gamma ^{\varepsilon } - \frac{c_{\ast\ast}}{3} \varepsilon  ^{\beta_\ast})}{ (R_0 +\delta)^2 -R_0 ^2 +2R_0 (s_\gamma ^{\varepsilon } - \frac{c_{\ast\ast}}{3} \varepsilon  ^{\beta_\ast}) }.
		\]
		In addition, there exists $C'>0$ depending only on $R_0, \delta,c_{\ast \ast}$
		such that for sufficiently small $\varepsilon >0$ it holds that
		\[
		R_0 ^2 + \frac{2R_0 (2\delta R_0 +\delta^2) (s_\gamma ^{\varepsilon} - \frac{c_{\ast\ast}}{3} \varepsilon  ^{\beta_\ast})}{ (R_0 +\delta)^2 -R_0 ^2 +2R_0 (s_\gamma ^{\varepsilon } - \frac{c_{\ast\ast}}{3} \varepsilon  ^{\beta_\ast}) }
		\leq R_0 ^2 - C' \varepsilon  ^{\beta_\ast},
		\]
		where we used 
		$|s_\gamma^{\varepsilon}| < \frac{c_{\ast\ast}}{6} \varepsilon ^{\beta_\ast}$
		for sufficiently small $\varepsilon >0$.\\[-0.2cm]
		\item $\tilde r_{-} ^{\varepsilon} (x,t) \geq \frac{1}{3} c_{\ast \ast} \varepsilon ^{\beta_\ast}$  if
		\begin{equation}\label{property:tilder}
			|x|^2 \leq 
			R_0 ^2 + \frac{2R_0 (2\delta R_0 +\delta^2) (s_\gamma ^{\varepsilon } - \frac{c_{\ast\ast}}{3} \varepsilon  ^{\beta_\ast})}{ (R_0 +\delta)^2 -R_0 ^2 +2R_0 (s_\gamma ^{\varepsilon } - \frac{c_{\ast\ast}}{3} \varepsilon  ^{\beta_\ast}) }<R_0 ^2.
		\end{equation}
		\item There exists $C''>0$ independent of $\varepsilon$ such that $0 \leq \partial _t \tilde r_{-} ^{\varepsilon  } \leq C'' \varepsilon  ^{-\beta_\ast}$.\\[-0.2cm]
		\item $\tilde r_{-} ^{\varepsilon } (x,\varepsilon ^{\beta_\ast}) = r_{0,-} (x)$ for any $x \in B_{R_0 +\delta} (0)$.
	\end{enumerate}
	Therefore, we only need to check that $\tilde r_{-} ^{\varepsilon  } (x,t)$ is a subsolution of equation 
	\eqref{eq:ACr} on $ B_{R_0 +\delta} (0)$. 
	When $|x|^2 \geq R_0 ^2 - C' \varepsilon  ^{\beta_\ast}$, by (2), we have 
	$\tilde r_{-} ^{\varepsilon  } (x,t) = r_{0,-} (x)$. Thus, $\tilde r_{-} ^{\varepsilon  } (x,t)$ is a subsolution
	in this case. Next we consider the case of $|x|^2 \leq R_0 ^2 - C' \varepsilon ^{\beta_\ast}$.
	We compute that
	\begin{align*}
		-\Delta \tilde r ^\varepsilon _-
		= & \, 
		\left(1- \frac{t}{\varepsilon ^{\beta_\ast}} \right) |\nabla r_{0,-}|^2 (-(\phi _0 ^\varepsilon)'')
		+ \left( (1- \frac{t}{\varepsilon ^{\beta_\ast}}) (\phi _0 ^\varepsilon )' + \frac{t}{\varepsilon^{\beta_\ast}} \right) (-\Delta r_{0,-})\\
		\leq & \,
		C \varepsilon ^{-\beta_\ast} -\Delta r_{0,-},
	\end{align*}
	where we used $|\nabla r_{0,-}| \leq 1$ in this case.
	By this and the definition of $\tilde r ^\varepsilon _-$, we have
	\[
	\frac{1}{3} c_{\ast \ast} \varepsilon  ^{\beta_\ast}\leq \tilde r_{-} ^{\varepsilon } \leq r_{0,-}
	\quad
	\text{and}
	\quad
	-\Delta \tilde r_{-} ^{\varepsilon } -g ^{\varepsilon} 
	\leq C \varepsilon ^{-\beta_\ast}  -\Delta r_{0,-}-g ^{\varepsilon} \leq C \varepsilon ^{-\beta_\ast} 
	\]
	for any $|x|^2 \leq R_0^2 - C' \varepsilon  ^{\beta_\ast}$ and $t \in [0, \varepsilon  ^{\beta_\ast}]$.
	In addition, since $|(\phi ^\varepsilon _0)' |\leq 1$ and $\tfrac{t}{\varepsilon ^{\beta_\ast}} \leq 1$, we get 
	\[
	|\nabla \tilde r_{-} ^{\varepsilon }| ^2
	\leq
	|\nabla r_{0,-}|^2 = 
	\dfrac{4c_{14}^2 \ |x|^2}{\big( (R_0 +\delta)^2 - |x|^2\big)^4} 
	\leq  
	\dfrac{4c_{14}^2  (R_0 ^2 - C' \varepsilon  ^{\beta_\ast})
	}{\big( (R_0+\delta)^2 - R_0^2  + C' \varepsilon  ^{\beta_\ast}
		\big)^4}
	\leq \frac{R_0 ^2 -C' \varepsilon  ^{\beta_\ast}}{R_0 ^2}.
	\]
	By an argument similar to that in \eqref{eq:est2sub},
	$G(q^{\varepsilon}( \frac{1}{3} c_{\ast \ast} \varepsilon  ^{\beta_\ast} )) \leq -\sqrt{\beta}$
	for sufficiently small $\varepsilon>0$.
	Hence, for sufficiently small $\varepsilon>0$ we have
	\begin{align*}
		& \partial _t \tilde r_{-} ^{\varepsilon  } - \Delta \tilde r_{-} ^{\varepsilon  } - \dfrac{G(q^{\varepsilon}(\tilde r_{-} ^{\varepsilon  }))}{\varepsilon } 
		( |\nabla \tilde r_{-} ^{\varepsilon  }|^2 -1) - g^{\varepsilon} 
		\\
		\leq & \,
		(C+C'') \varepsilon ^{-\beta_\ast} 
		+ \dfrac{G(q^{\varepsilon}( \frac{1}{3} c_{\ast \ast} \varepsilon  ^{\beta_\ast} ))}{\varepsilon } 
		\frac{C'}{R_0 ^2} \varepsilon ^{\beta_\ast} \leq 0,
	\end{align*}
	where we used $\beta _\ast \in (0,\frac12)$.
	Thus $\tilde r_{-} ^{\varepsilon  } (x,t)$ is a subsolution of 
	\eqref{eq:ACr} on $ B_{R_0 +\delta} (0)$ for sufficiently small $\varepsilon>0$.
\end{proof}
The estimate for the supersolution $ u_{y,+} ^\varepsilon $ is similar to that of $u_{y,-} ^\varepsilon$, so we omit it here. Altogether, we have shown the following proposition.
\begin{proposition}
	\label{prop:obs}
	Let $u^{\varepsilon }$ be a solution to \eqref{prob:AC}.
	Assume that $B_{R_0} (y) \subset O_+$
	(resp. $B_{R_0} (y) \subset O_-$). Then
	\begin{equation*}
		u_{y,-} ^\varepsilon \leq u ^\varepsilon \quad \text{in} \ B_{R_0+\delta} (y) \times [\varepsilon^{\beta_\ast},\infty)
		\qquad
		\left( 
		\text{resp.}
		\quad
		u_{y,+} ^\varepsilon \geq u ^\varepsilon \quad \text{in} \ B_{R_0+\delta} (y) \times [\varepsilon^{\beta_\ast},\infty)
		\right)
	\end{equation*}
	for sufficiently small $\varepsilon >0$.
\end{proposition}

 As a consequence of Proposition \ref{prop:obs} we have the following result. 

\begin{corollary}\label{cor:convoutsideobs}
Assume that $B_{R_0} (y) \subset O_+$
(resp. $B_{R_0} (y) \subset O_-$). Set $r \in (0, R_0)$. Then the following hold.
\begin{enumerate}
\item There exists a constant $c>0$ such that $\sup\limits _{x \in B_{r} (y)} | u^\varepsilon (x,t) - 1 | \leq c e^{-\frac{\sqrt{\beta}c}{\varepsilon^{1-\beta_\ast}}}$ 
(resp. $\sup\limits _{x \in B_{r} (y)} | u^\varepsilon (x,t) + 1 | \leq c e^{-\frac{\sqrt{\beta}c}{\varepsilon^{1-\beta_\ast}}}$)
for sufficiently small $\varepsilon >0$ and for any $t \in [\varepsilon^{\beta_\ast},\infty )$.
\\[-0.2cm]
\item $\lim \limits_{\varepsilon \to 0} \mu _t ^\varepsilon (B_{r} (y)) =0$ for any $t\geq 0$.
\end{enumerate}
\end{corollary}

\begin{proof}
Without loss of generality we may assume that $y=0$ and that $B_{R_0} (0) \subset O_+$.  Let $r \in (0, R_0)$.
Recall that by \eqref{property:tilder}, 
$ \tilde r_{-} ^\varepsilon \geq \frac{1}{3}c_{\ast\ast} \varepsilon ^{\beta_\ast} $ in $B_{r}(0)$
for sufficiently small $\varepsilon>0$. Assume that $t \in [\varepsilon^{\beta_\ast},\infty )$.
Then, for any $x \in B_r (0)$ 
and any sufficiently small $\varepsilon$, we have by Propositions \ref{prop:obs} and \ref{prop:cp} that
		\[
		q^\varepsilon ( \frac{1}{3}c_{\ast\ast} \varepsilon ^{\beta_\ast} ) 
		\leq 
		q^\varepsilon(\tilde r_{-} ^\varepsilon (x,t))  
		\leq 
		u ^\varepsilon (x,t) 
		\leq 1.
		\]
		On the other hand, \eqref{eq:prop5q} implies 
		\begin{equation}\label{eq:cor4.4}
			0 \leq 1 - q^\varepsilon (\frac{1}{3}c_{\ast\ast} \varepsilon ^{\beta_\ast} ) 
			= 1 - q ( \frac{1}{3}c_{\ast\ast} \varepsilon ^{\beta_\ast}  / \varepsilon) 
			\leq c e^{-\frac{\sqrt{\beta}c}{\varepsilon^{1-\beta_\ast}}}
		\end{equation}
		for some constant $c>0$ independent of $\varepsilon$. 
		Hence we obtain (1). Next we show (2). If $t=0$, then then claim is true by Proposition \ref{prop:initial}. Hence we may assume $t \in [\varepsilon^{\beta_\ast},\infty )$ for sufficiently small $\varepsilon>0$.
		Using assumptions (W1)-(W3) and  \eqref{eq:cor4.4}, we compute that
		\[
		\int _{B_r (0)} \frac{W(u^\varepsilon)}{\varepsilon} \dx
		\leq 
		\int _{B_r (0)} \frac{W( q^\varepsilon (\frac{1}{3}c_{\ast\ast} \varepsilon ^{\beta_\ast} ) )}{\varepsilon} \dx
		\leq \frac12 \sup _{|s| \leq 1} |W'' (s)| 
		\int _{B_r (0)} \frac{ ( q^\varepsilon (\frac{1}{3}c_{\ast\ast} \varepsilon ^{\beta_\ast} ) -1 )^2 }{\varepsilon} \dx
		\to 0 
		\]
		as $\varepsilon \to 0$.
		An integration by parts yields
		\[
		\int _{B_r (0)} \varepsilon |\nabla u^\varepsilon| ^2 \, \dx
		= \int _{B_r (0)} \varepsilon \nabla u^\varepsilon \cdot \nabla (u^\varepsilon -1) \, \dx
		= - \int _{B_r (0)} \varepsilon (u^\varepsilon -1) \Delta u^\varepsilon  \, \dx
		+ \int _{\partial B_r (0)} \varepsilon \nabla u^\varepsilon \cdot \hat{\nu} (u^\varepsilon -1) 
		{\rm d} \mathscr{H}^{n-1},
		\]
		where $\hat{\nu}$ is the outer unit normal of $\partial B_r (0)$.
		By an argument similar to that in \cite[Lemma 4.1]{takasao-tonegawa} (see Lemma \ref{lem:rebound} and \cite{ladyzenskaja} also), we have the interior $C^\frac12$-estimate of $\nabla u^\varepsilon$. Therefore
		the standard interior Schauder estimate implies
		\[
		\sup_{(x,t) \in B_r (0) \times [\varepsilon ^2,T)} | \nabla ^k u ^\varepsilon (x,t) | \leq c(k) \varepsilon ^{-k},
		\qquad k=1,2,
		\]
		for any $T\in(0,\infty)$.  
		Hence, the first claim (1) yields
		\[
		\lim _{\varepsilon \to 0} \int _{B_r (0)} \varepsilon |\nabla u^\varepsilon| ^2 \, \dx=0
		\]
		and this completes the proof.
	\end{proof}


Finally, at the end of this section, we show the monotone decreasing property of the limiting measure of  $\mu _t^\varepsilon$.
\begin{corollary}\label{cor:monotoneofmu}
Let $\{ \mu_t \}_{t\geq 0}$ be a family of Radon measures given by Proposition \ref{prop:limitmeas}.
Then it holds that
\[
\mu _s (\overline{\Omega}) \leq \mu _t (\overline{\Omega}) \quad \text{for any } \: 0 \leq t <s <\infty. 
\]
\end{corollary}

	\begin{proof}
		By Corollary \ref{cor:convoutsideobs}, we have
		$u^\varepsilon (x,t) \to 1$ as $\varepsilon \to 0$ for a.e. $x \in O_+$
		(resp. $u^\varepsilon (x,t) \to -1$ as $\varepsilon \to 0$ for a.e. $x \in O_-$)
		for any $t \in (0,\infty)$. 
		Thus the dominated convergence theorem shows that 
		\begin{equation}\label{eq:rmk3.4}
			\lim _{\varepsilon \to 0} \int_{\Omega} g^\varepsilon (x) \, k(u^{\varepsilon }(x,t)) \, \dx 
			=
			c_1 \big(k(1) \mathscr{L}^n (O_+ ) -  k(-1) \mathscr{L}^n (O_- )\big)
		\end{equation}
		for any $t \in (0,\infty)$. In addition, \eqref{eq:rmk3.4} also holds for $t=0$ by \eqref{appendix:A5}. Recall that the constant $c_1$ is defined in \eqref{eq:c1} and that the function $k(s)$ is defined in the begining of Subsection \ref{subsec:energyinequality}. 
		Let $\phi =\chi_{\overline{\Omega}} \in C (\overline{\Omega})$ be a constant function. 
		The weak convergence property yields that 
		\[
		\mu_t (\overline \Omega) = \mu_t (\phi) = \lim_{\varepsilon \to 0} \mu _t ^\varepsilon (\phi)
		=\lim_{\varepsilon \to 0} \mu _t ^\varepsilon (\overline \Omega)
		\]
		for any $t\geq 0$. 
		Therefore, Corollary \ref{cor:boundmeas} and \eqref{eq:rmk3.4} imply that 
		$
		\mu _T (\overline \Omega) \leq \mu _0 (\overline \Omega)
		$ for any $T \in (0,\infty)$. Repeating the same argument, we obtain
		\[
		\mu _s (\overline \Omega) \leq \mu _t (\overline \Omega) \quad \text{for any } \: 0 \leq t <s <\infty.
		\]
	\end{proof}

\section{Existence of weak solution to mean curvature flow with obstacles and right-angle boundary condition}
\label{sec:mainresults}

 In this section we prove the main result, which states the global existence of the weak solution to the mean curvature flow with obstacles and with a right-angle boundary condition, in the sense of Brakke. 
First we define the varifold corresponding to the solution $u^\varepsilon$ to \eqref{prob:AC} by
\[
V_t ^{\varepsilon} (\phi)
:=
\int_{\Omega \cap \{ |\nabla u^\varepsilon| \not =0 \}} 
\phi \left(x, \text{Id} -  \frac{\nabla u^\varepsilon}{|\nabla u^\varepsilon|} \otimes \frac{\nabla u^\varepsilon}{|\nabla u^\varepsilon|}\right) \, \d \mu _t ^\varepsilon \qquad
\text{for any} \ \phi \in C (\mathbf{G}_{n-1} (\overline{\Omega})).
\]
One can check that $C_c (\mathbf{G}_{n-1} (\overline{\Omega}))=C (\mathbf{G}_{n-1} (\overline{\Omega}))$
and $\| V_t ^\varepsilon \| = \mu _t ^\varepsilon \lfloor_{\{ |\nabla u^\varepsilon| \not =0 \}}
= \mu _t ^\varepsilon + \xi_t ^\varepsilon \lfloor_{\{ |\nabla u^\varepsilon| \not =0 \}}$.
Hence we may assume that $\lim\limits_{i \to \infty} \| V_t ^{\varepsilon_i} \| =\mu _t $ for a.e. $t\geq 0$ 
since $\lim\limits_{i\to \infty} |\xi_t ^{\varepsilon_i}| =0$ for a.e. $t\geq 0$ by Theorem \ref{thm:vd}.
By the definition of the first variation of varifolds,
\[
\delta V_t ^{\varepsilon} (g)
=\int _{\mathbf{G}_{n-1} (\overline{\Omega})} \nabla g (x) \cdot S \, \d V_t ^{\varepsilon} (x,S)
=\int_{\Omega \cap \{ |\nabla u^\varepsilon| \not =0 \}} 
\nabla g (x) \cdot \left(\text{Id} -  \frac{\nabla u^\varepsilon}{|\nabla u^\varepsilon|} \otimes \frac{\nabla u^\varepsilon}{|\nabla u^\varepsilon|}\right) \, \d \mu _t ^\varepsilon
\]
for any $g \in C^1 (\overline{\Omega};\mathbb{R}^n)$.
For  $\delta V_t ^{\varepsilon}(g)$, we have the following two lemmata.
\begin{lemma}[See \cite{kagaya,mizuno-tonegawa}]
	For any $g \in C^1 (\overline{\Omega} ;\mathbb{R}^n)$ and for any $t \in [0,\infty)$, we have
\begin{equation}\label{firstvariation}
		\begin{split}
			\sigma \delta V_t ^{\varepsilon} (g)
			= & \,
			\int _\Omega \varepsilon \left( \Delta u^\varepsilon -\frac{W'(u^\varepsilon)}{\varepsilon^2}\right) 
			\nabla u ^\varepsilon \cdot g \, \dx 
			+\sigma \int _{\Omega \cap \{ |\nabla u^\varepsilon |\not =0 \}}
			\nabla g \cdot \left( \frac{\nabla u^\varepsilon}{|\nabla u^\varepsilon|} \otimes \frac{\nabla u^\varepsilon}{|\nabla u^\varepsilon|} \right) \, \d \xi _t ^\varepsilon \\
			& \, + \int_{\partial \Omega} g \cdot \nu 
			\left( \frac{\varepsilon |\nabla u^\varepsilon|^2}{2}
			+ \frac{W(u^\varepsilon)}{\varepsilon} \right) \, \d \mathscr{H}^{n-1}
			- \int_{\Omega \cap \{ |\nabla u^\varepsilon | =0 \}}
			\mathrm{div}\, g \frac{W(u^\varepsilon)}{\varepsilon} \, \dx.
		\end{split}
	\end{equation}
\end{lemma}

The proof is identical to the proof of \cite[Lemma 7.2]{mizuno-tonegawa}, so we omit it.

\begin{lemma}[See \cite{kagaya,mizuno-tonegawa}]
	There exists a constant $C>0$ independent of $\varepsilon$ and $t$ such that
	\begin{equation}\label{firstvariation2}
		\int _{\partial \Omega} 
		\left( \frac{\varepsilon |\nabla u^\varepsilon|^2}{2}
		+ \frac{W(u^\varepsilon)}{\varepsilon} \right) \, \d \mathscr{H}^{n-1} 
		\leq \int_{\Omega} \varepsilon 
		\left( \Delta u^\varepsilon -\frac{W'(u^\varepsilon)}{\varepsilon^2}\right)^2 \, \dx
		+C
	\end{equation}
	for any $t\geq 0$.
\end{lemma}

\begin{proof}
	Let $\phi \in C^2 (\mathbb{R}^n)$ be a function which satisfies $\nabla \phi = \nu$ on $\partial \Omega$.
	Substituting $g=\nabla \phi$ into \eqref{firstvariation}, we have
	\begin{equation*}
		\begin{split}
			&\int_{\partial \Omega} \left( \frac{\varepsilon |\nabla u^\varepsilon|^2}{2}
			+ \frac{W(u^\varepsilon)}{\varepsilon} \right) \, \d \mathscr{H}^{n-1}\\
			= & \, \sigma \delta V_t ^{\varepsilon} (\nabla \phi)
			- \int _\Omega \varepsilon \left( \Delta u^\varepsilon -\frac{W'(u^\varepsilon)}{\varepsilon^2}\right) 
			\nabla u ^\varepsilon \cdot \nabla \phi \, \dx \\
			& \, - \sigma \int _{\Omega \cap \{ |\nabla u^\varepsilon |\not =0 \}}
			\nabla^2 \phi \cdot \left( \frac{\nabla u^\varepsilon}{|\nabla u^\varepsilon|} \otimes \frac{\nabla u^\varepsilon}{|\nabla u^\varepsilon|} \right) \, \d \xi _t ^\varepsilon
			+ \int_{\Omega \cap \{ |\nabla u^\varepsilon | =0 \}}
			\Delta \phi \frac{W(u^\varepsilon)}{\varepsilon} \, \dx\\
			\leq & \,
			C(\| \phi \|_{C^2 (\overline{\Omega})}) \sigma \mu_t ^\varepsilon (\Omega)
			+ \int _\Omega \varepsilon \left( \Delta u^\varepsilon -\frac{W'(u^\varepsilon)}{\varepsilon^2}\right)^2 \, \dx \\
			\leq & \,
			C(\| \phi \|_{C^2 (\overline{\Omega})}) \sigma c_9
			+ \int _\Omega \varepsilon \left( \Delta u^\varepsilon -\frac{W'(u^\varepsilon)}{\varepsilon^2}\right)^2 \, \dx
		\end{split}
	\end{equation*}
	where we used Young's inequality, Corollary \ref{cor:boundmeas}, and $|\xi_t ^\varepsilon | (\Omega) \leq \mu _t ^\varepsilon (\Omega)$.
\end{proof}

Next we show the rectifiability of our varifolds.

\begin{theorem}\label{thm:rectifiability}
	For a.e. $t\geq 0$, there exists a unique $(n-1)$-rectifiable varifold $V_t \in  \mathbf{V}_{n-1} (\overline{\Omega})$
	such that $\|V_t\|=\mu_t$.
\end{theorem}
\begin{proof}
	Let $T \in (0,\infty)$. 
	By Theorem \ref{thm:vd}, $|\xi_t ^{\varepsilon_i} (\Omega)| \to 0$ in $L_1 (0,T)$. Hence
	there exists a subsequence $\{ \varepsilon_i \}_{i=1} ^\infty$ (denoted by the same $i$) such that
	\begin{equation}\label{thm:recti2}
		\lim_{i\to \infty} |\xi ^{\varepsilon _i} _t| =0
	\end{equation}
	for a.e. $t \in [0,T)$. 
	One can easily check that
	\begin{equation}\label{L2bddMC}
		\liminf_{i\to \infty} \int _0 ^T \int _\Omega \varepsilon_i \left( \Delta u^{\varepsilon_i} -\frac{W'(u^{\varepsilon_i})}{\varepsilon^2 _i }\right)^2 \, \dx \dt <\infty
	\end{equation}
	by Corollary \ref{cor:boundmeas} and 
	\[
	\varepsilon_i \left( \Delta u^{\varepsilon_i} -\frac{W'(u^{\varepsilon_i})}{\varepsilon^2 _i }\right)^2
	\leq 4
	\left( 
	\varepsilon _i (u^{\varepsilon_i} _t)^2
	+ c_1 ^2 \frac{2 W (u^{\varepsilon_i})}{\varepsilon_i}
	\right).
	\] 
	By this, \eqref{firstvariation2}, and Fatou's lemma, we obtain
	\begin{equation}\label{thm:recti1}
		\liminf_{i\to \infty} \left(
		\int _\Omega \varepsilon_i \left( \Delta u^{\varepsilon_i} -\frac{W'(u^{\varepsilon_i})}{\varepsilon^2 _i }\right)^2 \, \dx 
		+
		\int _{\partial \Omega} 
		\left( \frac{\varepsilon_i |\nabla u^{\varepsilon_i}|^2}{2}
		+ \frac{W(u^{\varepsilon_i})}{\varepsilon_i} \right) \, \d \mathscr{H}^{n-1} \right) 
		<\infty
	\end{equation}
	for a.e. $t \in [0,T)$. 
	Note that \eqref{thm:recti2} and \eqref{thm:recti1} hold for a.e. $t \in [0,T)$
	and we fix such $t$. 
	By the compactness of Radon measures, we may assume that there exist a subsequence $\{ \varepsilon_i \}_{i=1} ^\infty$ (denoted by the same $i$) and $\tilde V_t \in  \mathbf{V}_{n-1} (\overline{\Omega})$
	such that $\lim_{i\to \infty} V_t ^{\varepsilon _i} =\tilde V_t$. By \eqref{thm:recti2}, we have
	$\|\tilde V_t\| =\mu_t$.
	For any $g \in C^1 (\mathbb{R}^n;\mathbb{R}^n)$, there exists $C>0$
	depending only on $n$ such that
	\begin{equation}\label{firstvariation3}
		\begin{split}
			|\sigma \delta V_t ^{\varepsilon_i} (g)|
			\leq & \,
			\|g\|_{C^0} \left( \int _\Omega \varepsilon_i \left( \Delta u^{\varepsilon_i} -\frac{W'(u^{\varepsilon_i})}{\varepsilon^2 _i}\right)^2 
			\, \dx\right)^{\frac12} \left( \sigma \mu_t ^{\varepsilon _i} (\Omega) \right)^{\frac12} \\
			&\, + C\| g \|_{C^1} \sigma |\xi _t ^{\varepsilon_i}| (\Omega) 
			+ \| g \|_{C^0} \int_{\partial \Omega} \left( \frac{\varepsilon |\nabla u^{\varepsilon_i}|^2}{2}
			+ \frac{W(u^{\varepsilon_i})}{\varepsilon_i} \right) \, \d \mathscr{H}^{n-1},
		\end{split}
	\end{equation}
	where we used \eqref{firstvariation} and the Cauchy-Schwarz inequality.
	Hence, by \eqref{thm:recti2}, \eqref{thm:recti1}, and \eqref{firstvariation3},
	for any $g \in C^1 (\mathbb{R}^n;\mathbb{R}^n)$
	\[
	| \delta \tilde V_t (g)|
	=\lim_{i\to \infty} |\delta V_t ^{\varepsilon_i} (g)|
	\leq C \| g \|_{C^0}
	\]
	holds for some $C>0$ independent of $g$.
	Then Allard's rectifiability theorem tells us that 
	\[\tilde V_t \lfloor _{\{ x \mid \limsup_{r \downarrow 0} \frac{\|\tilde V _t \| (B_r (x))}{\omega_{n-1}r^{n-1}} >0 \}
		\times \mathbf{G} (n,n-1) }
	\] 
	is $(n-1)$-rectifiable (see \cite{allard}). Next we show 
	$\limsup_{r \downarrow 0} \frac{\|\tilde V _t \| (B_r (x))}{\omega_{n-1}r^{n-1}} >0$ for $\|\tilde V_t\|$-a.e. $x$.
	By the monotonicity formula, one can show $\mathscr{H}^{n-1} (\textnormal{spt}  \mu_t )<\infty$
	(see \cite[Corollary 7.8]{kagaya}, for example). For the measure $\mu_t$, for any $\alpha>0$ we have
	\[
	\mu_t \left(
	\{ x \in \textnormal{spt}  \mu_t \mid \limsup_{r \downarrow 0} \frac{\mu_t (B_r (x))}{\omega_{n-1} r^{n-1}} <\alpha  \}
	\right)
	\leq 2^{n-1} \alpha \mathscr{H}^{n-1} (\textnormal{spt} \mu_t ),
	\]
	where we used the standard estimate of measure theory (see \cite[Theorem 3.2]{simon}).
	By this and $\|\tilde V_t\| =\mu_t$, we have $\limsup_{r \downarrow 0} \frac{\|\tilde V _t \| (B_r (x))}{\omega_{n-1}r^{n-1}} >0$ for $\|\tilde V_t\|$-a.e. $x$. Hence 
	\[
	\tilde V_t = \tilde V_t \lfloor _{\{ x \mid \limsup_{r \downarrow 0} \frac{\|\tilde V _t \| (B_r (x))}{\omega_{n-1} r^{n-1}} >0 \}
		\times \mathbf{G} (n,n-1) }
	\] 
	and $\tilde V_t$ is $(n-1)$-rectifiable. The uniqueness of $\tilde V_t$ follows from $\|\tilde V_t\| =\mu_t$.
	Therefore, we may write $\tilde V_t$ as $V_t$.
	Since $T$ can be taken arbitrarily, $V_t$ exists for a.e. $t\geq 0$.
\end{proof}

For the following result, let $BV _{loc} (\Omega \times [0,\infty) )$ stand for the space of locally bounded 
variation functions on $\Omega \times [0,\infty)$.

\begin{theorem}\label{maintheorem1}
Suppose all the assumptions of Sections \ref{sec:preliminaries} and \ref{sec:AC}.
Then the following hold:
\begin{enumerate}
\item There exist a subsequence $\{ \varepsilon _{i_j} \}_{j=1} ^\infty$ and
a family of Radon measures $\{ \mu_t \}_{t \in [0,\infty)}$ on $\overline{\Omega}$
such that
$\mu _t ^{\varepsilon _{i_j}} \rightharpoonup \mu_t$ for any $t\geq 0$.\\[-0.2cm]
\item There exists 
$\psi \in BV _{loc} (\Omega \times [0,\infty) ) \cap C_{loc} ^{\frac12} ([0,\infty);L_1 (\Omega)) $ such that 
$\frac12 (u^{\varepsilon _{i_j}} +1) \to \psi$ in $L_{1,loc} (\Omega \times (0,\infty))$ and
a.e. pointwise. Moreover, $\psi =0$ or $1$ a.e. on $\Omega \times [0,\infty)$.  \\[-0.2cm]
\item $\| \nabla \psi (\cdot ,t) \| (\phi) \leq \mu _t (\phi)$ for any 
non-negative  $\phi \in C_c (\Omega)$ 
and any $t \geq 0$.  \\[-0.2cm]
\item $\mu_0 = \mathscr{H}^{n-1} \lfloor_{M_0}$
and $\mu _t (\overline{\Omega})$ is monotone decreasing. \\[-0.2cm]
\item $\textnormal{spt} \mu_t \cap O_\pm = \emptyset$ for any $t\geq 0$. In addition, 
$\psi = 1$ a.e. on $O_+ \times [0,\infty)$ and 
$\psi = 0$ a.e. on $O_- \times [0,\infty)$. \\[-0.2cm]
\item For a.e. $t\geq 0$, there exists a $(n-1)$-rectifiable varifold $V_t$ on $G_{n-1} (\overline{\Omega})$. 
Moreover, for a.e. $t\geq 0$, $V_t$ is $(n-1)$-integral on 
$\Omega \setminus \overline{O_+ \cup O_-} $. 
\end{enumerate}
\end{theorem}

\begin{proof}
The statements (1) and (5) were already proved in Proposition \ref{prop:limitmeas} and 
Corollary \ref{cor:convoutsideobs}.
The standard argument implies (2) and (3) 
(see \cite[Proposition 8.3]{takasao-tonegawa}). However, 
we give the brief proof of (2) and (3) for readers' convenience.
Set $w^{\varepsilon_i} = \sigma^{-1} \int _{-1} ^{u^{\varepsilon _i}} \sqrt{2W (y)} \, \d y$.
Then we have
\[
\limsup_{i \to \infty} \| w^{\varepsilon _i} \|_{L_\infty (\Omega \times (0,T))}
+
\limsup_{i \to \infty} \int _0 ^ T \int _{\Omega} |\nabla w ^{\varepsilon _i}| \, \d x \d t 
<\infty
\]
for any $T \in (0,\infty)$. Here we used Proposition \ref{prop:cp}, Corollary \ref{cor:boundmeas}, and the estimate
\begin{equation}\label{eq:8.1}
|\nabla w^{\varepsilon _i}| = \sigma^{-1} \sqrt{2W (u^{\varepsilon _i})} | \nabla u^{\varepsilon _i} |
\leq \sigma ^{-1} \left( \frac{\varepsilon_i |\nabla u^{\varepsilon_i}|^2}{2} + \frac{W(u^{\varepsilon_i})}{\varepsilon_i}\right).
\end{equation}
Similarly we obtain $\limsup_{i \to \infty} \int _0 ^ T \int _{\Omega} 
|\partial _t  w ^{\varepsilon _i} | \, \d x \d t 
<\infty$.
Therefore, the compactness of $BV$ functions implies that there exist 
$\psi \in BV_{loc} (\Omega \times (0,\infty))$ and a subsequence 
(denoted by the same $i$) such that
$w^{\varepsilon_i} \to \psi $ strongly in $L _{1,loc} (\Omega \times (0,\infty))$ 
and for a.e. $(x,t) \in \Omega \times (0,\infty)$ (see \cite[Theorem 5.5]{evans-gariepy}). 
Hence, there exists a function $u^\infty$ such that 
$u^{\varepsilon_i} \to u^\infty $ for a.e. $(x,t) \in \Omega \times (0,\infty)$.
In addition, by $\sup_{i \in \N} \| u^{\varepsilon_i} \|_{L_\infty} \leq 1$ and the convergence theorem,
$u^{\varepsilon_i} \to u^\infty $ strongly in $L _{1,loc} (\Omega \times (0,\infty))$.
Since there exists $C>0$ such that
\[
\int _0 ^T \int _{\Omega} W (u^{\varepsilon _i}) \, \d x \d t \leq CT\varepsilon _i
\] 
holds for any $i\geq 1$ and for any $T\in (0,\infty)$, Fatou's lemma implies that
$
\liminf_{i\to \infty} W (u^{\varepsilon _i}) = 0 
$
for a.e. $(x,t) \in \Omega \times (0,\infty)$. Hence $u^\infty =\pm 1$
for a.e. $(x,t) \in \Omega \times (0,\infty)$.
Because $u^{\varepsilon_i} \to 1$ if and only if $w^{\varepsilon _i} \to 1$
and $u^{\varepsilon_i} \to -1$ if and only if $w^{\varepsilon _i} \to 0$, 
we have
$
\psi = \frac{1}{2} (u^\infty +1) = \lim_{i\to \infty} \frac{1}{2}(u^{\varepsilon _i} +1)
$
for a.e. $(x,t) \in \Omega \times (0,\infty)$
and thus $\psi =0$ or $1$ for a.e. $(x,t) \in \Omega \times (0,\infty)$. 
Next, since $w^{\varepsilon_i} \to \psi $ for a.e. $(x,t) \in \Omega \times (0,\infty)$, 
there exist $C_T >0$ and a null set $N \subset [0,\infty)$
such that for any $t_1,t_2 \in [0,T) \setminus N$ with $0\leq t_1 < t_2 <T$, 
\begin{equation}\label{psi-cauchy}
\begin{split}
\int _{\Omega} |\psi (x,t_2) - \psi (x,t_1)| \, d x 
=\, & \lim_{i\to \infty} \int _{\Omega} |w^{\varepsilon _i} (x,t_2) - w^{\varepsilon _i} (x,t_1)| \, d x
\leq \liminf _{i\to \infty} \int_{\Omega} \int _{t_1} ^{t_2} |\partial _t  w^{\varepsilon _i} | \, \d t \d x \\
\leq \, &  \liminf _{i\to \infty} \frac{1}{\sigma} \int_{\Omega} \int _{t_1} ^{t_2} 
\left(
\frac{\varepsilon _i | \partial _t  u^{\varepsilon _i} | ^2}{2} \sqrt{t_2 -t_1}
+ \frac{W (u^{\varepsilon _i})}{\varepsilon _i \sqrt{t_2 -t_1}}
\right)
\, \d t \d x \\
\leq \, &
C_T \sqrt{t_2 -t_1},
\end{split}
\end{equation}
where we used Corollary \ref{cor:boundmeas}. 
Note that if $\{ t_i \}_{i=1} ^\infty \in [0,\infty) \setminus N$ converges to some $t \in N$,
$\{ \psi (\cdot,t_i) \}_{i=1} ^\infty$ is a Cauchy sequence in $L_1 (\Omega)$. 
Therefore, For $t \in N$, we may redefine $\psi (x,t)$ such that \eqref{psi-cauchy}
holds for any $0\leq t_1 <t_2<T$. 
Hence we obtain (2).
From the $L_1$-convergence in $\Omega\times (0,T)$, 
$w^{\varepsilon_i} (\cdot,t) \to \psi (\cdot,t) $ strongly in $L_1 (\Omega )$ for a.e. $t$. Moreover, 
this convergence is true for any $t$, because of the continuity of $\psi$ with respect to $t$.
Let $\phi \in C_c (\Omega)$ be a non-negative
function. 
From the lower semicontinuity of the BV function, for any $t$, we have
\[
\int_{\Omega} \phi \, \d \| \nabla \psi (\cdot ,t) \|
\leq 
\liminf_{i\to \infty} \int_{\Omega} \phi |\nabla w^{\varepsilon _i}| \, \dx
\leq \lim_{i\to \infty} \int_{\Omega} \phi \, \d \mu _t ^{\varepsilon _i}= \int _{\Omega} \phi \, \d \mu _t.
\]
Here we used \eqref{eq:8.1}.
Therefore we obtain (3). 
By Proposition \ref{prop:initial} and Corollary \ref{cor:monotoneofmu} we have (4).
Except the integrality, (6) was shown in Theorem \ref{thm:rectifiability}. For any integer $j\geq 1$, set 
$\Omega_j := \{ x \in \Omega \mid \dist (x, O_\pm )>1/j \} $. Then, for sufficiently small $\varepsilon>0$,
$g^{\varepsilon} =0$ on $\Omega_j$. Hence, for such $\varepsilon$, 
\eqref{prob:AC} is the standard Allen-Cahn equation in $\Omega_j \times (0,\infty)$
and the integrality on $\Omega_j$ is obtained by \cite{takasao-tonegawa} (see also \cite{tonegawa2003}). Since $j$ is arbitrary, we have (6).

\end{proof}

 Next, we will prove that the set of limit varifolds $\{V_t\}_{t\in [0,\infty)}$ obtained in Theorem \ref{maintheorem1} is a Brakke flow on $\Omega \setminus \overline{O_+ \cup O_-} $ with a generalized right-angle condition on $\partial \Omega$. In order to discuss the right-angle condition of $V_t$ on the boundary, we have to introduce a tangential component of the first variation $\delta V_t$ on $\partial\Omega$. This is done in the following definition. 

\begin{definition}[See \cite{mizuno-tonegawa}]
	For $t\geq 0$ such that $\| \delta V_t \| (\overline{\Omega}) <\infty$ holds, we define
	\[
	\delta V_t \lfloor_{\partial \Omega} ^\top (g) 
	:=
	\delta V_t \lfloor_{\partial \Omega} (g -(g \cdot \nu)\nu) \quad \text{for} \ g \in C (\partial \Omega;\R^n),
	\]
	where we recall that $\nu$ denotes the outward unit normal of $\partial\Omega$.
\end{definition}

As within the proof of Theorem \ref{maintheorem1} above, in $\Omega_j \times (0,\infty)$,  problem \eqref{prob:AC} is the Allen-Cahn equation studied in \cite{kagaya,mizuno-tonegawa}
for sufficiently small $\varepsilon>0$.
Thus, with exactly the same proof as in \cite{kagaya,mizuno-tonegawa}, we obtain the following theorem.

\begin{theorem}\label{maintheorem2}
For the varifolds obtained in Theorem \ref{maintheorem1}, the following hold.
\begin{enumerate}
\item For a.e. $t\geq 0$,  $\| \delta V_t \| (\overline{\Omega})$ is finite and belongs to $L_{1,loc}  (0,\infty)$. \\[-0.2cm]
\item (generalized Neumann boundary condition) For a.e. $t\geq 0$, 
\[
\| \delta V_t \lfloor_{\partial \Omega} ^\top + \delta V_t \lfloor_{\Omega} \| \ll \|V_t\|.
\]
\item (modified generalized mean curvature vector) 
For a.e. $t\geq 0$, there exists $h_b =h_b (\cdot,t) \in L_2 (\|V_t\|)$ such that
\begin{equation}\label{eq:MCVec}
\delta V_t \lfloor_{\partial \Omega} ^\top + \delta V_t \lfloor_{\Omega} 
=
-h_b \| V_t \|, 
\end{equation}
$h_b (\cdot,t) = h (V_t, \cdot)$ in  $\Omega$ (see \eqref{def:MCVec}), 
and for any $\phi \in C(\overline{\Omega};[0,\infty))$,
\[
\int _{\overline{\Omega}} \phi |h_b| ^2 \, \d\|V_t\| \leq 
\liminf_{j\to \infty} \int _\Omega \varepsilon_{i_j} \left( \Delta u^{\varepsilon_{i_j}} -\frac{W'(u^{\varepsilon_{i_j}})}{\varepsilon^2 _{i_j}}\right)^2 \phi \, \dx < \infty,
\]
where $\{ \varepsilon_{i_j} \}_{j=1} ^\infty$ is the subsequence given by Theorem \ref{maintheorem1}. \\[-0.2cm]

\item (Brakke's inequality) 
Let $\phi \in C^1 (\overline{\Omega} \times [0,\infty) ;[0,\infty))$ satisfy
$\nabla \phi (\cdot , t) \cdot \nu=0$ on $\partial \Omega$ and $\textnormal{spt} \phi \cap O_\pm =\emptyset$.
Then 
\begin{equation*}
\int _{\overline \Omega} \phi (\cdot ,t) \, \d \| V_t \| \Big| _{t=t_1} ^{t_2}
\leq
\int_{t_1} ^{t_2} \int_{\overline \Omega } 
(-\phi |h_b| ^2 +\nabla \phi \cdot h_b +\partial_t \phi ) 
\, \d \|V_t\| \d t
\end{equation*}
for any $0 \leq t_1 < t_2 <\infty$.
\end{enumerate}
\end{theorem}

	\begin{remark}
		Let $V_t$ be as in Theorem \ref{maintheorem1}. By Theorem \ref{maintheorem2} (2)
		we can apply the Riesz representation theorem along with the Lebesgue decomposition theorem, as done in \eqref{def:MCVec}, to infer that 
		\begin{equation}\label{eq:MCVec1}
			(  \delta V_t \lfloor_{\partial \Omega} ^\top + \delta V_t \lfloor_{\Omega})(g) = 	- \int_{\overline{\Omega}} g \cdot h_b(\cdot ,t)   \, \d \|V_t\|  \quad \text{for } g \in C(\overline{\Omega}; \R^n). 
		\end{equation}
		Since $\delta V_t(g)$ coincides with $(\delta V_t \lfloor_{\partial\Omega}^{\top}+\delta V_t \lfloor_{\Omega})(g)$ for any $g\in C^{1}(\overline{\Omega};\mathbb{R}^{n})$ with
		$g \cdot \nu =0$ on $\partial\Omega$, it follows from \eqref{def:MCVec} and \eqref{eq:MCVec1} that 
		\begin{equation}\label{eq:MCVec2}
			\begin{split}
			- &\int _{\overline{\Omega}} g \cdot h(V_t, \cdot )   \, \d \| V_t\| 
			+ \int _{Z_t}  \nu^{\text{sing}}_t\cdot g \, \d \|\delta V_t\| \\
			&\quad = \delta V_t (g) 
		     = \delta V_t \lfloor_{\partial \Omega} ^\top (g) + \delta V_t \lfloor_{\Omega} (g) 
			= - \int_{\overline{\Omega}} g \cdot h_b(\cdot ,t)   \, \d \|V_t\| 
		\end{split}
		\end{equation}
		for any $g\in C^{1}(\overline{\Omega};\mathbb{R}^{n})$ with
		$g \cdot \nu =0$ on $\partial\Omega$. A simple computation (see, e.g., \cite[Remark 3.7]{kagaya}) provides that
		\begin{itemize}
			\item[(i)] the generalized boundary $Z_t$ (corresponding to $\partial M_t$ in Section \ref{sec:intro}) 
			is a subset of $\partial\Omega$, \\[-0.2cm]
			\item[(ii)] the generalized normal vector field $\nu^{\text{sing}}$ (corresponding to $\nu_{_{\partial M_t}}$
			in Section \ref{sec:intro}) is perpendicular to $\partial\Omega$ $\| \delta V_t\|$-a.e. on $Z_t$, and\\[-0.2cm]
			\item[(iii)] the vector field $h_{b}(\cdot, t)$ coincides indeed with the generalized mean curvature vector $h(V_t, \cdot)$ $\| V_t\|$-a.e. in $\Omega$.
		\end{itemize} 
		Thus Theorem \ref{maintheorem2} (2) corresponds to the right-angle condition of $V_t$,  see also \cite{kagaya-tonegawa}. 
	\end{remark}

\begin{remark}
The proofs of statements (2) and (3) of Theorem \ref{maintheorem2} are given in \cite[Proposition 7.4]{mizuno-tonegawa}. 
In the case of $O_{\pm}=\emptyset$ (that is, $g^\varepsilon \equiv 0$ for any $\varepsilon>0$), 
it holds that
\[
\int_0 ^\infty \int_{\overline{\Omega}} |h_b|^2 \, \d \|V_t\| \dt <\infty.
\]
This estimate can be seen in \cite[Theorem 2.5]{mizuno-tonegawa}.
Roughly speaking, it is expected that the varifold $V_t$ converges to the minimal surface as $t\to \infty$.
On the other hand, in our case this is not expected due to the obstacles, but
\[
\int_0 ^T \int _{\overline{\Omega}} \phi |h_b| ^2 \, \d\|V_t\| \dt 
\leq
\liminf_{i\to \infty} \int _0 ^T \int _\Omega \varepsilon_{i_j} \left( \Delta u^{\varepsilon_{i_j}} -\frac{W'(u^{\varepsilon_{i_j}})}{\varepsilon^2 _{i_j} }\right)^2 \, \dx \dt <\infty
\]
holds for any $T \in (0,\infty)$, by (3), \eqref{L2bddMC}, and Fatou's lemma. 
\end{remark}

\appendix

\section{Proof of Proposition \ref{prop:initial}}
Here we prove Proposition \ref{prop:initial}, that is, we give a family of explicit initial data for the Allen-Cahn equation \eqref{prob:AC} and provide some important properties. 

\begin{proof}
For $\varrho>0$, we define $\Omega _\varrho : = \{ x \in \Omega \mid \dist (x,\partial \Omega) >\varrho \}$.
	Let $d =d (x)$ be the signed distance function for $M_0 $ with $d (x)>0$ for any $x$ 
	with $y_1 (x) >0$. Note that $y_1 (x)$ is the first component of the map $\Phi^{-1}$
	and $d(x) >0$ for any $x \in U_0$. 
	Set $N_\varrho :=\{y \mid \dist (y, M _0)<\varrho \}$. 
	Then we may assume that there exists $\varrho_\ast >0$ such that 
	$d $ is $C^1$ and $|\nabla d | =1$ on $N_\varrho \cap \Omega_\varrho$
	for any $\varrho \in (0,\varrho_\ast)$.
	
	Let $y' =y'(\theta_1,\theta_2,\dots, \theta_{n-2},r)$ be the polar coordinates for $\R^{n-1}$
	and define
	\[
	\tilde \Phi (y_1 ,\theta_1,\theta_2,\dots, \theta_{n-2},r):= 
	\Phi (y_1,y'( \theta_1,\theta_2,\dots, \theta_{n-2},1)) +
	(r-1) \nu (\Phi (y_1,y'(\theta_1,\theta_2,\dots, \theta_{n-2},1))).
	\]
	Note that $\Phi (y_1,y'( \theta_1,\theta_2,\dots, \theta_{n-2},1)) \in\partial \Omega$. 
	Since the Jacobian matrices satisfy
	\begin{equation}\label{matrix}
		\frac{\partial (\tilde \Phi)^{-1}}{\partial (x_1,\dots,x_n)}
		\frac{\partial \tilde \Phi}{\partial (y_1 ,\theta_1,\theta_2,\dots, \theta_{n-2},r)} =\text{Id},
	\end{equation}
	we have $\nabla y_1 (x) \cdot \nu (x) =0$ for any $x \in \partial \Omega \cap \Phi (A)$.
	Note that the $(1,n)$-entry of the left-hand side of \eqref{matrix} is $\nabla y_1 (x) \cdot \nu (x)$.
	Now, put $r_b (x) := c_\ast y_1 (x)$ with 
	$c_\ast>0$ being a constant and with $\sup_{x \in \im \tilde \Phi } |\nabla r_b (x) | \leq 
	\frac{9}{10}$.
	
	Let $\phi_\varrho \in C_c ^\infty (\R^n)$ be a non-negative cut-off function for $\partial M_0$ such that
	\begin{equation*}
		0\leq \phi_\varrho \leq 1, \quad 
		\phi_\varrho (x) =
		\begin{cases}
			1 & \text{if $\dist(x,\partial \Omega) \leq \tfrac{\varrho}{2}$}, \\[2mm]
			0 & \text{if $\dist(x,\partial \Omega) \geq \varrho$}, 
		\end{cases}
	\quad \text{and} \quad |\nabla \phi _\varrho| \leq c(n) \varrho^{-1}. 
	\end{equation*}
	Set $\beta_\ast \in (0,\frac12)$ ,
	$\varrho_i := \varepsilon _i ^{\frac{\beta_\ast}{2} }$ and
	\[
	\tilde r ^i (x) := (1-\varepsilon_i ^\frac{\beta_\ast }{3})(1-\phi _{\varrho_i} (x)) d (x) + \phi_{\varrho_i} (x) r_b(x).
	\]
	Since
	$|d| \leq \varepsilon_i ^{\beta_\ast} $ and $| r_b | \leq \varepsilon_i ^{\beta_\ast} $
	on $N_{\varepsilon _i ^{\beta_\ast}}$
	for sufficiently large $i$,
	we compute
	\begin{equation*}
	\begin{split}
	|\nabla \tilde r^i | \leq & \, |\nabla \phi _{\varrho_i} | (|d| + | r_b | ) 
	+ (1-\varepsilon_i ^\frac{\beta_\ast }{3}) (1-\phi _{\varrho_i} ) + \phi _{\varrho_i} |\nabla r_b | \\
	\leq & \, c(n) \varrho_i ^{-1} (|d| + | r_b | ) + 1-\varepsilon_i ^\frac{\beta_\ast }{3}
	\leq c(n) \varepsilon_i ^\frac{\beta_\ast }{2} + 1-\varepsilon_i ^\frac{\beta_\ast }{3}
	\leq 1
	\qquad \text{on} \ N_{\varepsilon ^{\beta_\ast} _i },
	\end{split}
	\end{equation*}
	where we used $|\nabla d |\leq 1$ and $|\nabla r_b| \leq \frac{9}{10}$.
	Next, we define a monotone function $\eta \in C^\infty (\R) $ by 
	\[
	\eta (r) :=
	\begin{cases}
		r & \text{if} \ |r| \leq \frac{1}{2} , \\[0.2cm]
		\frac{2}{3} & \text{if} \ r\geq 1, \\[0.2cm]
		-\frac{2}{3} & \text{if} \ r\leq - 1 ,
	\end{cases}
	\qquad \text{and} \qquad
	|\eta' (r)|\leq 1 \quad \text{for any} \ r \in \R
	\]
	and set $r^i := c_{\ast\ast} \varepsilon_i ^{\beta_\ast} 
	\eta \Big(\tfrac{\tilde r ^i}{ c_{\ast\ast} \, \varepsilon_i ^{\beta_\ast}}\Big)$ for $c_{\ast\ast}>0$. 
	We choose $c_{\ast\ast}>0$ sufficiently small so that 
	$r^i$ is constant on $N_{\varepsilon_i ^{\beta_\ast}} ^c$ and
	$r^i \in C^1 (\overline{\Omega})$
	for sufficiently large $i$. 
	Note that $r^i$ is not well-defined without $c_{\ast\ast}$
	( $r_b$ does not satisfy $|\nabla r_b| =1$ and is defined only on the neighborhood of $\partial M_0$.
	To define $r^i$ as a constant function on $N_{\varepsilon_i ^{\beta_\ast}} ^c$,  we need to take $c_{\ast \ast}>0$ so that 
	$\left| \tfrac{ (1-\varepsilon_i ^\frac{\beta_\ast }{3}) d }{ c_{\ast\ast} \, \varepsilon_i ^{\beta_\ast}}\right| \geq 1$ and
	$\left| \tfrac{ r_b}{ c_{\ast\ast} \, \varepsilon_i ^{\beta_\ast}}\right| \geq 1$ on $N_{\varepsilon_i ^{\beta_\ast}} ^c$). 
	Note that
	$r^i= \tilde r^i = (1-\varepsilon_i ^\frac{\beta_\ast }{3}) d $ 
	on $N_{c_{\ast\ast} \varepsilon_i ^{\beta_\ast} /2} \cap \Omega_{\varrho_i}$ 
	since
	$|\frac{\tilde r ^i (x) }{c_{\ast \ast } \varepsilon _i ^{\beta_\ast}}| \leq \frac12$ holds for any 
	$x \in N_{ c_{\ast\ast} \varepsilon_i ^{\beta_\ast}/2} \cap \Omega_{\varrho_i}$.
	By the definition of $r^i$ and by the properties of $q$, see  \eqref{eq:q} and \eqref{eq:prop1q},  we obtain
	$\sup | u_0 ^{\varepsilon_i} | < 1$, 
	$| \nabla u_0 ^{\varepsilon_i} | \leq \frac{c}{\varepsilon _i}$, 
	$|\nabla r ^i| \leq 1$ on $\overline \Omega$, 
	and
	\[
	\frac{\varepsilon _i | \nabla u_0 ^{\varepsilon_i} (x) |^2 }{2}
	-\frac{W(u_0 ^{\varepsilon_i} (x) )}{\varepsilon_i}
	=
	\frac{W(u_0 ^{\varepsilon_i} (x) )}{\varepsilon_i} (|\nabla r^i (x)|^2 -1)
	\leq 0
	\]
	for $x \in \overline \Omega$. 
	Hence we obtain statement (1). We can easily check statement (2) as well.
	Next, for $\phi \in C_c (\Omega)$, we set 
	\[
	I_1
	:=
	\frac{1}{\sigma} \int_{(N_{c_{\ast\ast} \varepsilon_i ^{\beta_\ast} /2})^c} \phi \left( 
	\frac{\varepsilon _i | \nabla u_0 ^{\varepsilon_i}  |^2 }{2}
	+\frac{W(u_0 ^{\varepsilon_i} )}{\varepsilon_i}
	\right)  \dx
	\]
	and
	\[
	I_2
	:=
	\frac{1}{\sigma} \int_{N_{c_{\ast\ast} \varepsilon_i ^{\beta_\ast} /2} \setminus \Omega_{\varrho_i} } 
	\phi \left( 
	\frac{\varepsilon _i | \nabla u_0 ^{\varepsilon_i}  |^2 }{2}
	+\frac{W(u_0 ^{\varepsilon_i} )}{\varepsilon_i}
	\right)  \dx.
	\] 
	Note that $\frac{\varepsilon _i | \nabla u_0 ^{\varepsilon_i} (x) |^2 }{2}
	\leq \frac{W(u_0 ^{\varepsilon_i} (x) )}{\varepsilon_i}$ on $\overline{\Omega}$ and that 
	\[
	\frac{W(u_0 ^{\varepsilon_i} )}{\varepsilon_i}
	=
	\frac{W(q (\pm \frac{2}{3} c_{\ast\ast} \varepsilon _i ^{\beta_\ast-1}))}{\varepsilon_i}
	\longrightarrow 0 \ \ \text{ on } \ (N_{c_{\ast\ast} \varepsilon_i ^{\beta_\ast} /2})^c \ \ \text{ as } \ \varepsilon_i \to 0.
	\]
	Therefore $I_1 \to 0$ as $\varepsilon_i \to 0$.
	On $N_{c_{\ast\ast} \varepsilon_i ^{\beta_\ast} /2} \cap \Omega_{\varrho_i}$, 
	we have $ r^i = (1-\varepsilon_i ^\frac{\beta_\ast }{3}) d$, 
	$|\nabla r^i | = (1-\varepsilon_i ^\frac{\beta_\ast }{3}) $, and 
	\begin{equation}\label{equiv}
		\frac{\varepsilon _i | \nabla u_0 ^{\varepsilon_i} |^2 }{2}
		+ \big(1-\varepsilon_i ^\frac{\beta_\ast }{3}\big)^2 \frac{W(u_0 ^{\varepsilon_i} )}{\varepsilon_i}
		= (1-\varepsilon_i ^\frac{\beta_\ast }{3}) 
		\sqrt{2W(u_0 ^{\varepsilon_i} ) } \,  | \nabla u_0 ^{\varepsilon_i}  |,
	\end{equation}
	since,  by \eqref{eq:prop1q}, 
	$
	\frac{\varepsilon _i | \nabla u_0 ^{\varepsilon_i} |^2 }{2}
	=
	\frac{W(u_0 ^{\varepsilon_i} )}{\varepsilon_i}
	|\nabla r^i |^2
	=
	\big(1-\varepsilon_i ^\frac{\beta_\ast }{3}\big)^2
	\frac{W(u_0 ^{\varepsilon_i} )}{\varepsilon_i}
	$.
	Hence, applying Young's inequality, we get
	\begin{equation}\label{A.3}
	\begin{split}
		\frac{\varepsilon _i | \nabla u_0 ^{\varepsilon_i} |^2 }{2}
		+ \frac{W(u_0 ^{\varepsilon_i} )}{\varepsilon_i}
		= & \, \sqrt{2W(u_0 ^{\varepsilon_i} ) } \,  | \nabla u_0 ^{\varepsilon_i}  |
		-\varepsilon_i ^\frac{\beta_\ast }{3}
		\sqrt{2W(u_0 ^{\varepsilon_i} ) } \,  | \nabla u_0 ^{\varepsilon_i}  |
		+ (1-(1-\varepsilon_i ^\frac{\beta_\ast }{3})^2) \frac{W(u_0 ^{\varepsilon_i} )}{\varepsilon_i}\\
		\leq & \, \sqrt{2W(u_0 ^{\varepsilon_i} ) } \,  | \nabla u_0 ^{\varepsilon_i}  |
		+C \varepsilon_i ^\frac{\beta_\ast }{3}
		\left\{ 
		\frac{\varepsilon _i | \nabla u_0 ^{\varepsilon_i} |^2 }{2}
		+ \frac{W(u_0 ^{\varepsilon_i} )}{\varepsilon_i} \right\}
	\end{split}
	\end{equation}
	and 
	\begin{equation}\label{A.4}
		\frac{\varepsilon _i | \nabla u_0 ^{\varepsilon_i} |^2 }{2}
		+ \frac{W(u_0 ^{\varepsilon_i} )}{\varepsilon_i}
		\geq \sqrt{2W(u_0 ^{\varepsilon_i} ) } \,  | \nabla u_0 ^{\varepsilon_i}  |
		-C \varepsilon_i ^\frac{\beta_\ast }{3}
		\left\{ 
		\frac{\varepsilon _i | \nabla u_0 ^{\varepsilon_i} |^2 }{2}
		+ \frac{W(u_0 ^{\varepsilon_i} )}{\varepsilon_i} \right\},
	\end{equation}
for some $C>1$ independent of $\varepsilon_i$. 
	Without loss of generality, we may assume that $\phi \geq 0$.
	By the coarea formula and \eqref{A.3}, we obtain
	\begin{equation*}
		\begin{split}
			&\int_{N_{c_{\ast\ast} \varepsilon_i ^{\beta_\ast} /2}
			 \cap \Omega_{\varrho_i}} \phi \, \d\mu _0 ^{\varepsilon _i}
			=
			\frac{1}{\sigma} \int_{N_{c_{\ast\ast} \varepsilon_i ^{\beta_\ast} /2}
			 \cap \Omega_{\varrho_i}} \phi \left( 
			\frac{\varepsilon _i | \nabla u_0 ^{\varepsilon_i}  |^2 }{2}
			+\frac{W(u_0 ^{\varepsilon_i}  )}{\varepsilon_i}
			\right) \dx \\[0.2cm]
			\leq & \,  
			\frac{1}{\sigma} \int_{ N_{c_{\ast\ast} \varepsilon_i ^{\beta_\ast} /2}
			 \cap \Omega_{\varrho_i} } \phi \, 
			\sqrt{2W(u_0 ^{\varepsilon_i} ) } \, | \nabla u_0 ^{\varepsilon_i}  | \, \dx 
			+C \varepsilon_i ^\frac{\beta_\ast }{3} \int_{N_{c_{\ast\ast} \varepsilon_i ^{\beta_\ast} /2}
			 \cap \Omega_{\varrho_i}} \phi \, \d\mu _0 ^{\varepsilon _i} \\[0.2cm]
			= & \,  
			\frac{1}{\sigma} \int_{-\infty} ^\infty 
			\int_{\{ x \in N_{c_{\ast\ast} \varepsilon_i ^{\beta_\ast} /2} \cap \Omega_{\varrho_i} 
				\mid u_0 ^{\varepsilon_i }(x)=t \} } \phi (x)
			\sqrt{2W(u_0 ^{\varepsilon_i} (x)) } \,\d\mathscr{H}^{n-1} (x) \, \dt 
			+C \varepsilon_i ^\frac{\beta_\ast }{3} \int_{N_{c_{\ast\ast} \varepsilon_i ^{\beta_\ast} /2}
			 \cap \Omega_{\varrho_i}} \phi \, \d\mu _0 ^{\varepsilon _i} \\
			= & \,  
			\frac{1}{\sigma} \int_{-1} ^1\sqrt{2W(t) } \, \dt
			\int_{M_0} \phi (x) \,\d\mathscr{H}^{n-1} (x) 
			+C \varepsilon_i ^\frac{\beta_\ast }{3} \int_{N_{c_{\ast\ast} \varepsilon_i ^{\beta_\ast} /2}
			 \cap \Omega_{\varrho_i}} \phi \, \d\mu _0 ^{\varepsilon _i} +o(1) \\
			= & \,
			\int_{M_0} \phi (x) \,\d\mathscr{H}^{n-1} (x)  
			+C \varepsilon_i ^\frac{\beta_\ast }{3} \int_{N_{c_{\ast\ast} \varepsilon_i ^{\beta_\ast} /2}
			 \cap \Omega_{\varrho_i}} \phi \, \d\mu _0 ^{\varepsilon _i} +o(1) ,
		\end{split}
	\end{equation*}
	where $o(1)\to 0$ as $\varepsilon_i \to 0$ and where we used 
	$|u_0 ^{\varepsilon_i} | \leq 1$, and
	\begin{equation}\label{appendix:A5}
	u_0 ^{\varepsilon_i} (x) \to 1 \quad \text{if} \ r^i (x)>0 
	\quad (\text{resp.} \ u_0 ^{\varepsilon_i} (x) \to -1 \quad \text{if} \ r^i (x)<0) \ \ \text{ as } \ \varepsilon_i \to 0.
	\end{equation}
	Therefore we obtain
	\[
	(1-C \varepsilon_i ^\frac{\beta_\ast }{3}) 
	\int_{N_{c_{\ast\ast} \varepsilon_i ^{\beta_\ast} /2}
			 \cap \Omega_{\varrho_i}} \phi \, \d\mu _0 ^{\varepsilon _i} 
	\leq 
	\int_{M_0} \phi (x) \,\d\mathscr{H}^{n-1} (x) +o(1).
	\]
	Similarly, by \eqref{A.4},
	\[
	(1+ C \varepsilon_i ^\frac{\beta_\ast }{3}) 
	\int_{N_{c_{\ast\ast} \varepsilon_i ^{\beta_\ast} /2}
			 \cap \Omega_{\varrho_i}} \phi \, \d\mu _0 ^{\varepsilon _i} 
	\geq 
	\int_{M_0} \phi (x) \,\d\mathscr{H}^{n-1} (x) +o(1).
	\]
	Finally, we show
	\[
	I_2 \to 0 \ \ \text{ as } \ \varepsilon_i \to 0.
	\]
	Note that there exists $c>0$ such that $c<|\nabla r^i (x)| \leq 1$ 
	for any $x \in  N_{c_{\ast\ast} \varepsilon_i ^{\beta_\ast} /2} \setminus \Omega_{\varrho_i}$.
	Thus 
	\begin{equation*}
	\begin{split}
		\frac{\varepsilon _i | \nabla u_0 ^{\varepsilon_i} |^2 }{2}
		+ \frac{W(u_0 ^{\varepsilon_i} )}{\varepsilon_i}
		= & \, \frac{\varepsilon _i | \nabla u_0 ^{\varepsilon_i} |^2 }{2}
		+ |\nabla r^i (x)|^2 \frac{W(u_0 ^{\varepsilon_i} )}{\varepsilon_i}
		+ (1-|\nabla r^i (x)|^2) \frac{W(u_0 ^{\varepsilon_i} )}{\varepsilon_i}\\
		= & \, |\nabla r^i (x)| \sqrt{2W(u_0 ^{\varepsilon_i} ) } \,  | \nabla u_0 ^{\varepsilon_i}  |
		+ (1-|\nabla r^i (x)|^2) \frac{W(u_0 ^{\varepsilon_i} )}{\varepsilon_i}\\
		\leq & \, \sqrt{2W(u_0 ^{\varepsilon_i} ) } \,  | \nabla u_0 ^{\varepsilon_i}  |
		+ (1-c^2) \left\{
		\frac{\varepsilon _i | \nabla u_0 ^{\varepsilon_i} |^2 }{2}
		+
		\frac{W(u_0 ^{\varepsilon_i} )}{\varepsilon_i} 
		\right\}.
	\end{split}
	\end{equation*}
	Hence
	\[
		\frac{\varepsilon _i | \nabla u_0 ^{\varepsilon_i} |^2 }{2}
		+ \frac{W(u_0 ^{\varepsilon_i} )}{\varepsilon_i}
		\leq
		\frac{1}{c^2}
		\sqrt{2W(u_0 ^{\varepsilon_i} ) } \,  | \nabla u_0 ^{\varepsilon_i}  |.
	\]
	As above, we have $| I_2 | \leq C \mathscr{H}^{n-1} (\{ r^i = 0 \} \setminus \Omega_{\varrho_i}) 
	\to 0 \ \ \text{ as } \ \varepsilon_i \to 0$.
	
	Therefore, 
	\begin{equation*}
	\begin{split}
		\int_{\Omega} \phi \, \d\mu _0 ^{\varepsilon _i}
		\to \int_{M_0} \phi (x) \,\d\mathscr{H}^{n-1} (x) .
		\end{split}
	\end{equation*}
	Hence, we obtain statement (3). 
	By \eqref{density} and statement (3), we deduce statements (4) and (5).
\end{proof}

\section*{Acknowledgements}
K.N. is partially supported by the Austrian Science Fund (FWF) project  F\,65. 
K.T. is supported by JSPS KAKENHI
Grant Numbers JP20K14343, JP18H03670, JP23K03180, JP23H00085.

\end{document}